\documentclass[11pt]{amsart}
\usepackage{amsmath,amssymb,amsthm}
\usepackage[latin1]{inputenc}
\usepackage{version,tabularx,multicol}
\usepackage{graphicx,float,psfrag}
\usepackage{stmaryrd}

\usepackage{enumitem}

\usepackage{mathtools}  
\usepackage{pgf,tikz}
\usetikzlibrary{arrows}

\usepackage{pifont} 

\definecolor{qqqqff}{rgb}{0,0,1}
\definecolor{ffqqqq}{rgb}{1,0,0}
\definecolor{ffqqtt}{rgb}{1,0,0.2}
\definecolor{zzttqq}{rgb}{0.6,0.2,0}
\definecolor{ttttff}{rgb}{0.2,0.2,1}
\definecolor{zzqqqq}{rgb}{0.6,0,0}
\definecolor{qqccqq}{rgb}{0,0.8,0}
\newcommand{\reff}[1]{(\ref{#1})}

\theoremstyle{plain}
\newtheorem{theo}{Theorem}[section]
\newtheorem*{theo*}{Theorem}
\newtheorem{cor}[theo]{Corollary}
\newtheorem*{cor*}{Corollary}
\newtheorem{prop}[theo]{Proposition}
\newtheorem{lem}[theo]{Lemma}

\theoremstyle{remark}
\newtheorem{rem}[theo]{Remark}

\newcommand\Ne{\mathbb{N}^{*}}

\newcommand\pro{\mathbb{P}}            
\newcommand\1[1]{{\bf 1}_{\left\{#1\right\}}}                            
\newcommand\fp{\mathfrak{p}}
\newcommand\fF{\mathfrak{F}}
\newcommand\ff{{\rm f}}

\DeclareMathOperator{\Var}{Var}
\DeclareMathOperator{\Cov}{Cov}
\newcommand\cvloi[1]{\,\xrightarrow[#1\rightarrow\infty]{(d)}\,}     
\newcommand\cv[1]{\underset{#1 \rightarrow \infty}{\longrightarrow}}                               
\newcommand{\cuthere}{%
\noindent
\raisebox{-2.8pt}[0pt][0.95\baselineskip]{\ding{34}}
\unskip{\tiny\dotfill}
}
\newcolumntype{M}[1]{>{\centering\arraybackslash}m{#1}}

\newcommand{\ca}{{\mathcal A}}

\newcommand{\cb}{{\mathcal B}}
\newcommand{\cc}{{\mathcal C}}
\newcommand{\cd}{{\mathcal D}}

\newcommand{\cf}{{\mathcal F}}

\newcommand{\ch}{{\mathcal H}}

\newcommand{\cn}{{\mathcal N}}
\newcommand{\cm}{{\mathcal M}}

\newcommand{\cs}{{\mathcal S}}

\newcommand{\cu}{{\mathcal U}}

\newcommand{\cw}{{\mathcal W}}

\newcommand{\E}{{\mathbb E}}

\newcommand{\Id}{{\rm Id}}

\newcommand{\N}{{\mathbb N}}
\renewcommand{\P}{{\mathbb P}}

\newcommand{\R}{{\mathbb R}}

\newcommand{\rp}{{\mathfrak{p}}}

\newcommand{\rz}{{\mathfrak{z}}}
\newcommand{\rx}{{\rm x}}
\newcommand{\ry}{{\rm y}}
\newcommand{\re}{{\rm e}}

\newcommand{\ind}{{\bf 1}}
\newcommand{\Dw}{{D}}

\newcommand{\tti}{\tilde t_{\rm inj}}
\newcommand{\ti}{t_{\rm inj}}
\newcommand{\tid}{t_{\rm ind}}

\newcommand{\norm}[1]{\mathop{\parallel\! #1 \! \parallel}\nolimits}
\newcommand{\val}[1]{\mathop{\left| #1 \right|}\nolimits}
\newcommand{\inv}[1]{\mathop{\frac{1}{ #1}}\nolimits}

\newcommand{\rb}

\newcommand{\cdf}{\Pi}
\newcommand{\convset}{\mathcal{K}}
\newcommand{\conv}{K}
\newcommand{\verti}[1]{{\left\vert\kern-0.25ex\left\vert\kern-0.25ex\left\vert #1 
    \right\vert\kern-0.25ex\right\vert\kern-0.25ex\right\vert}}

\title[Asymptotics for random graphs]{Asymptotic  for the cumulative distribution function of
  the degrees and
  homomorphism densities for random graphs sampled from a
  graphon}
\date{\today}

\author{Jean-Fran\c{C}ois Delmas}
\address{
 Jean-Fran\c cois Delmas,
Universit\'{e} Paris-Est, Cermics (ENPC), F-77455 Marne-la-Vall\'{e}e.}
\email{delmas@cermics.enpc.fr}

\author{Jean-St\'{e}phane Dhersin}
\address{
Jean-St\'{e}phane Dhersin,
Universit\'{e} Paris 13, Sorbonne Paris Cit\'{e}, LAGA, CNRS (UMR 7539), 93430 Villetaneuse, France}
\email{dhersin@math.univ-paris13.fr}

\author{Marion Sciauveau}
\address{
Marion Sciauveau, 
Universit\'{e} Paris-Est, Cermics (ENPC), F-77455 Marne-la-Vall\'{e}e.}
\email{marion.sciauveau@enpc.fr}

\begin{document}

\thanks{This work is partially supported by DIM RDMath IdF 
and by Agence Nationale de la Recherche via the grant ANR-14-CE25-0014
\textquotedblleft GRAAL\textquotedblright}



\begin{abstract} 
  We give asymptotics for the cumulative distribution function (CDF) for
  degrees of large dense random  graphs sampled from a graphon.  The proof
  is based on precise asymptotics for binomial random variables.

  Replacing the  indicator function in  the empirical CDF by  a smoother
  function,  we  get  general  asymptotic  results  for  functionals  of
  homomorphism  densities for  partially  labeled  graphs with  smoother
  functions. This  general setting allows  to recover recent  results on
  asymptotics for homomorphism densities of sampled graphon.
\end{abstract}

\maketitle
\section{Introduction}

The Internet, social networks or  biological networks can be represented
by large  random graphs. Understanding  their structure is  an important
issue in Mathematics.   Degree sequences is one of the  key objects used
to get  informations about graphs.   The degree sequences of  real world
networks have attracted a lot of attention during the last years because
their distributions are significantly  different from the Poisson degree
distributions studied in  the classical models of random  graphs such as
the Erd\"{o}s-R\'{e}nyi  model. They followed a  power-law distribution,
see  for   instance,  Newmann  \cite{Newmann_2003_book},  Chung   et  al
\cite{Chung_2006_book},         Diaconis          and         Blitzstein
\cite{Diaconis_Blitzstein_2010_article} and  Newman, Barabasi  and Watts
\cite{Newman_Barabasi_Watts_2011_book}.   See   also  Molloy   and  Reed
\cite{Molloy_Reed_1995_article,Molloy_Reed_1998_article}   and   Newman,
Strogatz  and  Watts  \cite{Newman_Strogatz_Watts_2001_article}  in  the
framework of sparse graphs.
\\

In this  paper, we shall  consider the cumulative  distribution function
(CDF) of  degrees of  large dense random  graphs sampled  from a graphon,
extending     results      from     Bickel,     Chen      and     Levina
\cite{Bickel_Chen_Levina_2011_article}.  The theory of graphon or limits
of sequence  of dense  graphs was developped  by Lov\'{a}sz  and Szegedy
\cite{Lovasz_Szegedy_2006_article} and Borg, Chayes, Lov\'{a}sz, S\'{o}s
and   Vesztergombi   \cite{Borgs_Chayes_Lovasz_Sos_Vesztergombi}.    The
asymptotics on the empirical CDF of  degrees, see the theorem in Section
\ref{subsec:CV_CDF_degrees},  could be  used to  test if  a large  dense
graph is sampled from  a given graphon. This result is  a first step for
giving  a non-parametric  test for  identifying the degree  function of  a
large random graph in the spirit  of the Kolmogorov-Smirnov test for the
equality of   probability distribution from a
sample of independent identically distributed random variables.
\\

 If we  replace the indicator  function in the empirical CDF by  a
 smoother function,  we get general
 results on  the fluctuations for functionals  of homomorphism densities
 for  partially   labeled  graphs.    As  an
 application, when considering homomorphism densities for sampled graphon,
 we   recover  results   from  F\'{e}ray,   M\'{e}liot  and   Nikeghbali
 \cite{Feray_Meliot_Nikeghbali_2017_article}.

\subsection{Convergence of CDF of empirical degrees for large random graphs}
\label{subsec:CV_CDF_degrees}

We consider simple finite graphs, that is graph without  self-loops and  multiple edges
between any pair of vertices. We denote by $\cf$ the set of all simple  finite graphs.

There exists several equivalent notions  of convergence for sequences of
finite dense graphs  (that is graphs where the number  of edges is close
to  the maximal  number  of  edges), for  instance  in  terms of  metric
convergence (with  the cut distance) or  in terms of the  convergence of
subgraph    densities,        see 
\cite{Borgs_Chayes_Lovasz_Sos_Vesztergombi} or Lov\'{a}sz \cite{Lovasz_2012_book}.

When it exists, the limit of a sequence of dense  graphs can be represented by a graphon
i.e. a  symmetric, measurable function $W:[0,1]^2\to  [0,1]$, up to 
a measure preserving bijection.  A graphon
$W$ may be  thought of as the  weight matrix of an  infinite graph whose
 set of vertices   is the continuous  unit interval, so that  $W(x,y)$ represents
the weight of the edge between vertices  $x$ and $y$.

Moreover, it is possible to sample simple graphs, with a given number of
vertices,  from   a  graphon   $W$  (called  $W$-random   graphs).   Let
$X=(X_{i}:i\in\Ne)$  be  a  sequence  of  independent  random  variables
uniformly  distributed  on  the   interval  $[0,1]$.  To  construct  the
$W$-random  graph  with vertices  $[n]:=\{1,  \ldots,  n\}$, denoted  by
$G_n$, for each pair of distinct  vertices $i\neq j$, elements of $[n]$,
connect $i$ and $j$ with  probability $W(X_i,X_j)$, independently of all
other edges  (see also Section  \ref{sec:Gn-def}).  If needed,  we shall
stress the  dependence in  $W$ and  write $G_n(W)$  for $G_n$.   By this
construction, we get a sequence of random graphs $(G_{n}:n\in\Ne)$ which
converges
almost surely towards the graphon $W$, see for instance  Proposition
11.32 in 
\cite{Lovasz_2012_book}.
\\

We define the  degree function $D=(D(x):x\in[0,1])$ of the graphon $W$ by:
\begin{equation*}
 D(x)=\int_{0}^{1}W(x,y)dy.
\end{equation*}
And we consider the empirical CDF
$\cdf_{n}=(\cdf_{n}(y):y\in[0,1])$ 
of the normalized degrees of the graph $G_{n}$ defined by 
\begin{equation*}
 \cdf_{n}(y)=\frac{1}{n}\sum_{i=1}^{n}\1{D_{i}^{(n)}\le D(y)},
\end{equation*}
where $nD_{i}^{(n)}$ is the  degree of the vertex $i$ in $G_n$.\\

Bickel, Chen and  Levina \cite{Bickel_Chen_Levina_2011_article}, Theorem
5  (with  $m=1$),  proved  the   convergence  in  distribution  and  the
convergence  of the  second moments  of $\cdf_{n}(y)$  towards $y$.   We
improve  the  results given  in  \cite{Bickel_Chen_Levina_2011_article}:
under the condition that $D$ is  increasing\footnote{Since the
  graphon is defined up to a measure preserving one-to-one map on $[0,1]$,
  there exists an equivalent version of the graphon for which the degree
  function is non-decreasing. If the degree function is increasing, then
  this version is unique in $L^1$ and this is the  version which is considered
  in this section.}
on $[0,1]$, we have
the almost sure  convergence of $\cdf_{n}(y)$ towards  $y$, uniformly on
$[0, 1]$.  This is  a consequence  of the more  general result  given by
Theorem  \ref{Theo_LLN} (see  Subsection \ref{subsec:hom_densities}  and
Remark  \ref{rem:casD} for  more  details).  In  a different  direction,
Chatterjee   and  Diaconis   \cite{Chatterjee_Diaconis_Sly_2011_article}
considered  the convergence  of uniformly  chosen random  graphs with  a
given CDF  of degrees towards  an exponential graphon with  given degree
function.
\\
 
We also get the fluctuations  associated to the almost sure convergence
of $\Pi_n$.
If $W$ satisfies some  regularity conditions
given by \reff{eq:condi_W},  which in particular imply
that $D$ is of class
$\cc^1$,  then we  have  the   following  result on the
convergence in distribution  of  finite-dimensional
marginals      for         $\cdf_{n}$. 

\begin{theo*}[Theorem \ref{thm:CLT_indicator_degree_sequence}]
Assume that $W$ satisfies condition \reff{eq:condi_W}. Then 
 we have the following convergence of finite-dimensional distributions:
 \[
  \left(\sqrt{n}\left(\cdf_{n}(y)-y\right): y\in(0,1)\right) 
  \,\xrightarrow[n\rightarrow+\infty]{(fdd)}\, \chi,
 \]
where $(\chi_{y}:y\in(0,1))$ is a centered Gaussian process defined, for all $y\in(0,1)$ by:
\[
 \chi_{y}=\int_{0}^{1}(\rho(y,u)-\bar \rho(y))dB_{u},
\]
with $B=(B_u, u\geq 0)$ a standard Brownian motion, and $(\rho(y, u):u\in
[0, 1]) $ and
$\bar \rho(y) $ defined for $y\in (0, 1)$   by:
\[
 \rho (y,u)=\ind_{[0,y]}(u)-\frac{W(y,u)}{D'(y)} \quad \text{ and } \quad
 \bar \rho (y)=\int_{0}^{1}\rho(y,u)du.
\]
\end{theo*}

The   covariance  kernel   $\Sigma=\Sigma_1+\Sigma_2+\Sigma_3$  of   the
Gaussian   process    $\chi$   is   explicitly   given    by   Equations
\eqref{eq:def-s1},    \eqref{eq:def-s2}   and    \eqref{eq:def-s3} which
define respectively $\Sigma_1$, $\Sigma_2$ and $\Sigma_3$.   In
particular, we deduce that the variance  of $\chi(y)$, for $y\in (0, 1)$
is given by the elementary formula:
\begin{equation*}
 \Sigma(y,y)=y(1-y)+\frac{1}{D'(y)^2}\left(\int_{0}^{1}W(y,x)^2dx-D(y)^2\right)
                  +\frac{2}{D'(y)}\left(D(y)y-\int_{0}^{y}W(y,x)dx\right).
 \end{equation*}
 
 The proof  of this  result relies on  uniform Edgeworth  expansions for
 binomial    random     variables,    see    Bhattacharya     and    Rao
 \cite{Bhattacharya_Rao_2010_book}, and Stein method for binomial random
 vectors, see  Bentkus \cite{Bentkus_2003_article}.  The  convergence of
 the  process in  the Skorokhod  space could  certainly be  proved using
 similar but more involved arguments.  More generally, following van der
 Vaart \cite{Van_der_Vaart_1998_article}, Chapter $19$ on convergence of
 empirical CDF of independent  identically distributed random variables,
 if   would   be  natural   to   study   the  uniform   convergence   of
 $\frac{1}{n}\sum_{i=1}^{n}f(D_{i}^{(n)})$
 when $f$ belongs to a certain class of functions.\\

 The asymptotics on  the CDF of empirical degrees appear  formally as a
 limiting        case        of         the        asymptotics        of
 $\frac{1}{n}\sum_{i=1}^{n}f(D_{i}^{(n)})$  with  $f$ smooth.  This is
 developed in Section \ref{intro:cv-ti}. We  shall in fact adopt in this
 section  a more  general point  of view  as we  replace the  normalized
 degree sequence by  a sequence of homomorphism  densities for partially
 labeled graphs.

\subsection{Convergence of sequence of dense graphs towards graphons}
\label{subsec:hom_densities}

Recall that one of the equivalent notions of convergence of sequences of
dense  graphs is  given by  the convergence  of subgraph densities.   It is  the
latter one  that will  interest us.   We first  recall  the  notion of
homomorphism densities.  For  two  simple finite graphs $F$  and $G$ with
respectively $v(F)$  and $v(G)$  vertices, let  $\text{Inj}(F,G)$ denote
the set of injective homomorphisms (injective adjacency-preserving maps)
from $F$  to $G$  (see Subsection  \ref{subsec:graph_hom} for  a precise
definition). We  define the injective  homomorphism density from  $F$ to
$G$ by the following normalized quantity:
\begin{equation*}
\ti(F,G)=\frac{\left|\text{Inj}(F,G)\right|}{A_{v(G)}^{v(F)}},
\end{equation*}
where we  have for all  $n\ge k  \ge 1$, $A_{n}^{k}=n!/(n-k)!$.   In the
same way, we can define the  density of induced homomorphisms (which are
injective   homomorphisms  that   also   preserve  non-adjacency),   see
\reff{hom6}. Some authors study subgraph counts rather than homomorphism
densities,  but the  two quantities  are related,  see Bollob\'{a}s  and
Riordan  \cite{Bollobas_Riordan_2009_article},  Section   2.1,  so  that
results on  homomorphism densities  can be  translated into  results for
subgraph counts.\\
 
A sequence  of dense  simple finite  graphs $(H_{n}:n\in\Ne)$  is called
convergent if the sequence $(t_{\text{inj}}(F,H_n):n\in\Ne)$ has a limit
for every $F\in \cf$.  The  limit can be represented by a graphon,
say $W$ and we have that for every  $F\in \cf$:
\begin{equation*}
 \lim_{n\to\infty}\ti(F,H_n)=t(F,W),
\end{equation*}
where
\begin{equation*}
 t(F,W)=\int_{[0,1]^{V(F)}}\prod_{\{i,j\}\in E(F)} W(x_i,x_j) \prod_{k\in V(F)}dx_k .
\end{equation*}
According to \cite{Lovasz_2012_book}, Proposition 11.32,  the sequence of 
$W$-random graphs $(G_{n}:n\in\Ne)$ converges a.s. towards $W$, that
is  for all  $F\in \cf$, a.s.: 
\begin{equation}\label{eq:cv_W_graph}
 \lim_{n\to\infty} \ti(F,G_{n})=t(F,W).
\end{equation}

In  the  Erd\"{o}s-R\'{e}nyi  case,  that is  when  $W\equiv  \fp$ is constant,  the
fluctuations associated  to this  almost sure  convergence are  of order
$n$: for all $F\in \cf$ with $p$ vertices and $e$ edges, we
have the following convergence in distribution:
\[
    n\left(t_{\text{inj}}(F,G_{n}(\fp))-\fp^{e}\right) \cvloi{n} 
    \cn\left(0, 2e^{2}\fp^{2e-1}(1-\fp)\right),
\]
where $\mathcal{N}(m,\sigma^2)$ denotes a Gaussian random variable 
with mean $m$ and variance $\sigma^2$.
There  are  several  proofs  of this  central  limit  theorem.   Nowicki
\cite{Nowicki_1989_article}       and      Janson       and      Nowicki
\cite{Janson_Nowicki_1991_article}  used the  theory of  U-statistics to
prove the asymptotic  normality of subgraph counts  and induced subgraph
counts.  They  also obtained  the  asymptotic  normality of  vectors  of
subgraph counts and induced subgraph  counts.  In the particular case of
the joint distribution  of the count of edges,  triangles and two-stars,
Reinert   and  R\"{o}llin   \cite{Reinert_Rollin_2010_article}, Proposition  2,  obtained
bounds  on  the  approximation.   Using  discrete
Malliavin calculus,  Krokowski and  Th\"{a}le \cite{kt} generalized the  result of
 \cite{Reinert_Rollin_2010_article} (in a different probability metric)  and get the
rate of convergence associated to the multivariate central limit theorem
given  in \cite{Janson_Nowicki_1991_article}. See also   F\'{e}ray,  M\'{e}liot and  Nikeghbali
 \cite{Feray_Meliot_Nikeghbali_2016_book}, Section 10,   for   the    mod-Gaussian
 convergence of homomorphism densities.

 The asymptotics  of normalized subgraph  counts have also  been studied
 when the parameter  $\fp$ of the Erd\"{o}s-R\'{e}nyi  graphs depends on
 $n$, see for example Ruci\'{n}ski \cite{Rucinski_1988_article}, Nowicki
 and Wierman  \cite{Nowicki_Wierman_1988_article}, Barbour, Karo\'{n}ski
 and Ruci\'{n}ski  \cite{Barbour_Karonski_Rucinski_1989_article}, and Gilmer
 and Kopparty \cite{Gilmer_Kopparty_2016_article}.
 \\

In the  general framework of  graphon, the  speed of convergence  in the
invariance principle  is of order  $\sqrt{n}$, but for  degenerate cases
such  as  the  Erd\"{o}s-R\'{e}nyi  case.   This  result  was  given  by
F\'{e}ray,             M\'{e}liot             and             Nikeghbali
\cite{Feray_Meliot_Nikeghbali_2017_article},  Theorem  21: for  all
  $F\in \cf$,  we  have  the   following  convergence  in
distribution:
\begin{equation}\label{eq:FM_cv}
 \sqrt{n}\left(\ti(F,G_{n})-t(F,W)\right) \cvloi{n} 
    \cn\left(0, \sigma(F)^{2}\right),
\end{equation}
where, with $V(F)$ the set of vertices of $F$ and $v(F)$ its cardinal, 
\begin{equation*}
 \sigma(F)^{2}
 =\sum_{q,q'\in V(F)}t\big((F\bowtie F)(q,q'),W\big) - v(F)^2\,  t(F,W)^{2}
\end{equation*}
and $(F\bowtie F')(q,q')$  is the  disjoint union  of the  two 
simple finite graphs  $F$ and $F'$  where we identify  the vertices  $q\in  F$ and
$q'\in F'$ (see   point (iii) of Remark \ref{rem:variance_TCL}, for more
details).  Notice that in the Erd\"{o}s-R\'{e}nyi case, that is when $W$
is a constant graphon, the  asymptotic variance $\sigma(F)^{2}$ is equal
to $0$, which is consistent with  the previous paragraph since the speed
is of order $n$.

Using Stein's method,  Fang and R\"ollin \cite{Fang_Rollin_2015_article}
obtained   the  rate   of  convergence   for  the   multivariate  normal
approximation of the joint distribution of the normalized edge count and
the corrected and normalized $4$-cycle count. As a consequence, they get
a  confidence interval  to  test if  a  given graph  $G$  comes from  an
Erd\H{o}s-R\'{e}nyi random graph model  or a non constant graphon-random
graph     model.       Maugis,     Priebe,     Olhede      and     Wolfe
\cite{Maugis_Priebe_Olhede_Wolfe_2017_article}   gave a central   limit
theorem for subgraph  counts observed in a network  sample of $W$-random
graphs drawn from  the same graphon $W$ when the  number of observations
in  the  sample  increases  but  the  number  of  vertices  in  each  graph
observation remains  finite.  They also  get a central limit  theorem in
the  case  where  all  the  graph observations  may  be  generated  from
different graphons.   This allows  to test if the  graph observations
come from  a specified  model.  When  considering sequences  of graphons
which tend  to $0$,  then there  is a  Poisson approximation  of subgraph
counts. In    this     direction,    Coulson,    Gaunt     and    Reinert
\cite{Coulson_Gaunt_Reinert_2016_article}, Corollary 4.1,   used  the   Stein  method  to
establish an effective Poisson approximation for the distribution of the
number of  subgraphs in the graphon  model which are isomorphic  to some
fixed strictly balanced graph. \\

Motivated by those results, we present in the next section an invariance
principle for the distribution of
homomorphism densities of partially labeled graphs for $W$-random graphs
which can be seen as
a generalization of \reff{eq:FM_cv}.

\subsection{Asymptotics for homomorphism densities of partially labeled graphs for large random graphs}
\label{intro:cv-ti}

Let $n\in\Ne$  and $k\in[n]$. We  define the set  $\mathcal{S}_{n,k}$ of
all $[n]$-words of length $k$ such that all characters are distinct, see
\reff{Snp}. Notice  that $\left|\mathcal{S}_{n,k}\right|=A_{n}^{k}=n!/(n-k)!$. 

We generalize homomorphism densities for partially labeled graphs.
Let $F,G\in\mathcal{F}$ be two simple graphs with $V(F)=[p]$ and $V(G)=[n]$.
Assume  $n\geq p>k\geq 1$.  Let
 $\ell\in\mathcal{S}_{p,k}$ and  $\alpha\in\mathcal{S}_{n,k}$. We define
 $\text{Inj}(F^{\ell},G^{\alpha})$  the set  of injective  homomorphisms
 $\ff$ from  $F$ into  $G$ such  that $\ff(\ell_{i})=\alpha_{i}$
 for all $i\in[k]$, and its density:
\[
 \ti (F^{\ell},G^{\alpha})=\frac{|\text{Inj}(F^{\ell},G^{\alpha})|}{A_{n-k}^{p-k}}\cdot
\]

We define the random probability measure $\Gamma_n^{F,\ell}$ on
$([0,1],\mathcal{B}([0, 1]))$, with $\cb([0, 1])$ the Borel
$\sigma$-field on $[0,1]$,  by: for all measurable non-negative 
function $g$ defined on $[0,1]$, 
\begin{equation}
\label{eq:intro_def_gamma}
\Gamma_n^{F,\ell}(g)=\inv{|\cs_{n,k}|}\sum_{\alpha\in\mathcal{S}_{n,k}}
g\left(\ti  (F^{\ell},G_{n}^{\alpha})\right).
\end{equation}

We prove,  see Theorem \ref{Theo_LLN},  the almost sure  convergence for
the         weak         topology        of         the         sequence
$\left(\Gamma_{n}^{F,\ell}(dx)  :n\in\Ne\right)$  of random  probability
measure  on  $[0,1]$  towards   the  deterministic  probability  measure
$\Gamma^{F,\ell}(dx)=\E\left[\Gamma_{n}^{F,\ell}(dx)\right]$.

\begin{itemize}
   \item[-] If we take $g=\Id$ in \reff{eq:intro_def_gamma}, we recover the almost sure convergence
     given in \reff{eq:cv_W_graph} as according to \reff{eq:integ-tiFW0}:
\[
\ti(F, G_n)=\inv{|\cs_{n,k}|}\sum_{\alpha\in\mathcal{S}_{n,k}}
\ti  (F^{\ell},G_{n}^{\alpha}).
\]
 
\item[-] If we take $g=\ind_{[0,D(y)]}$ with $y\in(0,1)$ and $F=K_2$ 
(where $K_2$ denotes the complete graph with two vertices)
in \reff{eq:intro_def_gamma} and using 
the expression of $\Gamma^{F,\ell}$ given in  Remark \ref{rem:CP_Gamma},
(ii), we have, with $\bullet$ any vertex of $K_2$, that:
  \begin{equation*}
   \Gamma_{n}^{K_2,\bullet}(g)=\cdf_{n}(y)
    \quad \text{ and } \quad \Gamma^{K_2,\bullet}(g)=y.
  \end{equation*}
 Then, by Theorem \ref{Theo_LLN}, under the condition that $D$ is strictly increasing on $(0,1)$, 
 we have the almost sure 
 convergence of $\cdf_{n}(y)$ towards $y$, see Remark   \ref{rem:casD}.
\end{itemize}

We also have the fluctuations associated to this almost sure
convergence, see Theorem \ref{Theo_CLT} for a multidimensional
version.

\begin{theo*}
   Let  $W\in \cw$  be a  graphon.  Let $F\in  \cf$ be  a       simple   finite   graphs     with     $V(F)=[p]$,
   $\ell\in\mathcal{M}_{p}$, with $k=|\ell|$. 
 Then, for  all
  $g\in\mathcal{C}^{2}([0,1])$,  we have  the  following convergence  in
  distribution:
\[
\sqrt{n}\left(\Gamma^{F,\ell}_{n}(g)-\Gamma^{F,\ell}(g)\right)
  \cvloi{n} \mathcal{N}\left(0, \sigma^{F,\ell}(g)^{2}\right),
\]
  with $\sigma^{F,\ell}(g)^{2}
  =\Var(\cu_g^{F,\ell})$ and $\cu_g^{F,\ell}$ is defined in
  \reff{eq:varU}. 
 \end{theo*}

 Notice $\sigma^{F,\ell}(g)^{2}$  is an  integral involving  $g$ and  $g'$. The
 asymptotic results are still true when we consider a family of $d\ge 1$
 simple graphs  $F=(F_{m}:1\le m  \le d)\in\mathcal{F}^d$ and  we define
 $\Gamma^{F,\ell}_{n}$  on $[0,1]^d$,  see  Theorems \ref{Theo_LLN}  and
 \ref{Theo_CLT}  for  the  muldimensional case.   The  case  $g=\rm{Id}$
 appears  already  in  \cite{Feray_Meliot_Nikeghbali_2017_article},  see
 Corollary \ref{cor:Theo_CLT_multi} for the  graphs indexed version.  We
 have the following convergence  of finite-dimensional distributions (or
 equivalently of the process since $\mathcal{F}$ is countable).
\begin{cor*}[Corollary \ref{cor:Theo_CLT_multi}]
 We have the following convergence of finite-dimensional distributions:
 \begin{equation*}
  \left(\sqrt{n}\left(\ti
      (F,G_{n})-t(F,W)\right): F\in\mathcal{F}\right)
  \,\xrightarrow[n\rightarrow\infty]{(fdd)}\, \Theta_{\rm{inj}} ,
 \end{equation*}
where $\Theta_{\rm{inj}}=(\Theta_{\rm{inj}}(F):  F\in\mathcal{F})$ is a centered Gaussian
process with covariance function $K_{\rm{inj}}$ given, for $F,F'\in\mathcal{F}$, by:
\[
    K_{\rm{inj}}(F,F')
=\sum_{q\in V(F)}\sum_{q'\in V(F')}t\left((F\bowtie F')(q,q'),W\right) -
v(F)v(F') \, t(F,W)t(F',W).
\]
\end{cor*}

As  a consequence,  we get  the central  limit theorem  for homomorphism
densities from  quantum graphs,  see \reff{eq:cv_quantum_graph}  and for
induced  homomorphism densities,  see Corollary  \ref{cor:CLT_tind}.  In
the Erd\"{o}s-R\'{e}nyi case, the  one-dimensional limit distribution of
induced homomorphism densities is not necessarily normal: it's behaviour
depends on  the number of  edges, two-stars  and triangles in  the graph
$F$,            see           \cite{Nowicki_1989_article}            and
\cite{Janson_Nowicki_1991_article}.
\\

Notice   that    because   $g=\ind_{[0,D(y)]}$    is   not    of   class
$ \mathcal{C}^{2}([0,1])$, we can not apply Theorem \ref{Theo_CLT} (with
$F=K_2$ and  $k=1$) directly to  get the convergence in  distribution of
$\sqrt{n}(\Pi_n(y)   -y)$    towards   $\chi(y)$   given    in   Theorem
\ref{thm:CLT_indicator_degree_sequence}.   Nevertheless, the  asymptotic
variance can be formally obtained by computing $\sigma^{K_2,\bullet}(g)$
given   in   Theorem   \ref{Theo_CLT}   with   $g=\ind_{[0,D(y)]}$   and
$g'(z)dz=(D'(y))^{-1}\delta_{D(y)}(dz)$,  with  $\delta_{D(y)}(dz)$  the
Dirac  mass  at   $D(y)$.   However,   the   proofs  of  Theorems
\ref{Theo_CLT}   and  \ref{thm:CLT_indicator_degree_sequence}   require
different         approachs.       

  Similarly         to        Theorem
\ref{thm:CLT_indicator_degree_sequence} and in the spirit of Theorem
\ref{Theo_CLT}, it could be interesting to  consider  the
convergence of  CDF for  triangles or more  generally for simple  finite 
graphs $F$, $V(F)=[p]$, and $\ell\in \cs_{p,k}$:
\[
\left(\frac{1}{|\mathcal{S}_{n,k}|}\sum_{\alpha\in\mathcal{S}_{n,k}}
  \1{\ti(F^{\ell},G_{n}^{\alpha})\le t_{x}(F^{\ell},W)}:
  x\in(0,1)^{k}\right),
\] 
where  $t_{x}(F^{\ell},W)=\E\big[\ti\big(F^{\ell},G_{n}^{[k]}\big)\big|\,
(X_1, \ldots, X_k)=x\big]$, see \reff{hom5} and the second equality in
\reff{eq:sum-ht2}.

\subsection{Organization of the paper}

We recall the definitions of graph homomorphisms, gra\-phons, $W$-random
graphs in  Section \ref{sec:def}. We  present our main result  about the
almost  sure convergence  for  the random  measure $\Gamma_n^{F,  \ell}$
associated  to  homomorphism  densities of  sampling  partially  labeled
graphs  from  a  graphon  in  Section  \ref{sec:main-res},  see  Theorem
\ref{Theo_LLN}. The proof is  given in Section \ref{sec:proof_LLN} after
a  preliminary  result  given  in  Section  \ref{sec:prem_result}.   The
associated fluctuations are stated  in Theorem \ref{Theo_CLT} and proved
in  Section   \ref{sec:proof_CLT}.   Section   \ref{sec:CDF_degrees}  is
devoted to the asymptotics for the empirical CDF of degrees $\Pi_n$, see
Theorem  \ref{thm:CLT_indicator_degree_sequence}  for  the  fluctuations
corresponding  to the  almsot  sure convergence.   After some  ancillary
results given  in Section \ref{sec:preliminaries_CDF_degrees},  we prove
Theorem      \ref{thm:CLT_indicator_degree_sequence}     in      Section
\ref{sec:proof_CLT_CDF_degrees}.  We add a notation  index at the end of
the paper for the reader convenience. We postpone to the appendices some
technical  results on   precise uniform   asymptotics  for the  CDF of
binomial distributions, see Section \ref{sec:preliminaries_CDF_bin}, and
a proof of Proposition \ref{prop:stein}  on approximation for the CDF of
multivariate binomial distributions.

\section{Definitions}
\label{sec:def}

\subsection{First notations}

We denote by $|B|$  the cardinal of
the  set  $B$.  For $n\in \N^*$, we  set  $[n]=\{1,\dots,n\}$.   Let $\mathcal{A}$  be  a
non-empty  set   of  characters,   called  the  alphabet.    A  sequence
$\beta=\beta_{1}\dots \beta_{k}$, with $\beta_{i}\in\mathcal{A}$ for all
$1\le i  \le k$, is  called a  $\mathcal{A}$-word (or string)  of length
$\val{\beta}=k\in\Ne$.   The word  $\beta$ is  also identified  with the
vector  $(\beta_{1},\dots,\beta_{k})$,  and for $q\in \ca$, we write
$q\in \beta$ if $q$ belongs to $\{\beta_1, \ldots, \beta_k\}$.
The  concatenation of  two $\mathcal{A}$-words  $\alpha$ and  $\beta$ is
denoted by $\alpha\,\beta$.

We  now define  several other  operations on  words.  Let  $\beta$ be  a
$\mathcal{A}$-word  of   length  $p\in \N^*$  and  $k\in[p]$.   For  $\alpha$  a
$[p]$-word   of  length   $k$,   we   consider  the   $\mathcal{A}$-word
$\beta_{\alpha}$, defined by
\begin{equation}\label{select_characters}
  \beta_{\alpha}=\beta_{\alpha_1}\dots\beta_{\alpha_{k}}\cdot
 \end{equation}
The word $\beta_{[k]}=\beta_{1}\dots\beta_{k}$ corresponds to the first $k$ terms of $\beta$, where by convention, $[k]$ denote the $\Ne$-word 
$1\dots k$.
We define, for $i,j\in[p]$, the transposition word $\tau_{ij}(\beta)$ of $\beta$, obtained by exchanging the place of the
$i$th character with the $j$th character in the word $\beta$: for $u\in[p]$,
\begin{equation}\label{word_transposition}
 \tau_{ij}(\beta)_u=
 \begin{cases}
   \beta_{u}  &\mbox{if } u\notin \{i,j\}, \\
   \beta_{i}  &\mbox{if } u=j, \\
   \beta_{j}  &\mbox{if } u=i.
 \end{cases}
\end{equation}
Finally, for  $q\in\mathcal{A}$ and  $i\in[p]$, we  define the  new $\ca$-word
$R_{i}(\beta,q)$,  derived  from  $\beta$   by  substituting  its  $i$th
character with $q$: for $u\in[p]$,
\begin{equation}\label{word_substitution}
 R_{i}(\beta,q)_u=
 \begin{cases}
  \beta_{u} &\mbox{if } u\neq i, \\
   q        &\mbox{if } u=i.
 \end{cases}
\end{equation}

Let $n\in\Ne$  and $p\in[n]$. We  define the set  $\mathcal{S}_{n,p}$ of
all $[n]$-words of length $p$ such that all characters are distinct:
\begin{equation}\label{Snp}
\mathcal{S}_{n,p}=\left\{\beta=\beta_{1}\dots\beta_{p}: \beta_{i}\in
  [n]\text{ for all }  i \in[p] \text{ and } \beta_{1},\dots,\beta_{p}
  \text{ are all distinct} \right\}. 
\end{equation}
Notice                                                              that
$\left|\mathcal{S}_{n,p}\right|=A_{n}^{p}=n!/(n-p)!$,   and   that
$\mathcal{S}_{n,1}=[n]$.   Moreover, for  $n\in\Ne$, $\mathcal{S}_{n,n}$
is  simply the  set of  all  permutations of  $[n]$ which  will be  also
denoted by  $\mathcal{S}_{n}$.  With these notations,  for $n\in\Ne$, we
define the set $\mathcal{M}_{n}$ of  all $[n]$-words with all characters
distinct:
\begin{equation}\label{[n]_words}
 \mathcal{M}_{n}=\bigcup_{p\in[n]}\mathcal{S}_{n,p}.
\end{equation}
Let   $n\ge  p\ge   k\geq  1$   and
$\ell\in\mathcal{S}_{p,k}$.   For  $\alpha\in\mathcal{S}_{n,k}$,        we        define       the        set
$\mathcal{S}_{n,p}^{\ell,\alpha}$ of all $[n]$-words  of length $k$ such
that  all   characters  are   distinct  and   for  all   $i\in[k]$,  the
$\ell_{i}$-th character is equal to $\alpha_{i}$:
\begin{equation}\label{Snp_beta_ell}
 \mathcal{S}_{n,p}^{\ell,\alpha}=\left\{\beta\in\mathcal{S}_{n,p}: \beta_{\ell}=\alpha\right\}\cdot
\end{equation}
We have $|\cs_{n,p}^{\ell,\alpha}|=A_{n-k}^{p-k}$. 
As  $A_{n}^{p}=A_{n}^{k}A_{n-k}^{p-k}$,  that is
$|\cs_{n,p}|=|\cs_{n,k}|\, |\cs^{\alpha, \ell}_{n,p}|$ for any
$\alpha\in \cs_{n,k}$, we get
that for all real-valued function $f$
defined on $\mathcal{S}_{n,k}$:
\begin{equation}
\label{eq_snk}
 \inv{|\cs_{n,p}|}\sum_{\beta\in\mathcal{S}_{n,p}}
f\left(\beta_{\ell}\right)=\inv{|\cs_{n,k}|}
\sum_{\alpha\in\mathcal{S}_{n,k}}f(\alpha)\cdot    
\end{equation}

Let $d\in \N^*$. For $x, y\in \R^d$, we
denote by $\langle x,y \rangle$ the usual scalar product on $\R^d$ and
$|x|=\sqrt{\langle x,x \rangle} $ the Euclidean norm in $\R^d$. \\

We use the convention $\prod_\emptyset=1$. 

\subsection{Graph homomorphisms}
\label{subsec:graph_hom}

A simple  finite graph  $G$ is  an ordered pair  $(V(G),E(G))$ of  a set
$V(G)$ of  $v(G)$ vertices,  and a  subset $E(G)$  of the  collection of
$\binom{v(G)}{2}$ unordered pairs of vertices. We usually shall identify
$V(G)$ with  $[v(G)]$.  The elements of  $E(G)$ are called edges  and we
denote by $e(G)=|E(G)|$ the number of  edges in the graph $G$.  Recall a
graph $G$  is simple when  it has no  self-loops, and no  multiple edges
between any pair of vertices. Let $\cf$ be the set of all simple  finite graphs.

Let  $F, G\in  \cf$ be  two simple  finite graphs  and set  $p=v(F)$ and
$n=v(G)$.    A   homomorphism   $\ff$   from    $F$   to   $G$   is   an
adjacency-preserving map from  $V(F)=[p]$ to $V(G)=[n]$ i.e.  a map from
$V(F)$    to   $V(G)$    such   that    if   $\{i,j\}\in    E(F)$   then
$\{\ff(i),\ff(j)\}\in  E(G)$. Let  $\text{Hom}(F,G)$ denote  the
set of homomorphisms from $F$ to $G$.  The homomorphism density from $F$
to $G$ is the normalized quantity:
\begin{equation}
\label{hom1}
t(F,G)=\frac{|\text{Hom(F,G)}|}{n^p}\cdot
\end{equation}
It is the probability that a uniform random map from $V(F)$ to $V(G)$ is
a  homomorphism.   We have similar definition when  $\ff$  is
restricted to  being injective. 
Let $\text{Inj}(F,G)$ denote  the set of injective  homomorphisms of $F$
into $G$ and define its density as:
\begin{equation}\label{hom2}
\ti (F,G)=\frac{|\text{Inj}(F,G)|}{A_{n}^{p}}\cdot
\end{equation}
For $\beta\in \cs_{n,p}$, we set, with $V(F)=[p]$ and  $V(G)=[n]$:
\begin{equation}
   \label{eq:def-Yb}
Y^\beta(F,G)= \prod_{\{i,j\}\in E(F)}\1{\{\beta_i,\beta_j\}\in E(G)}.
\end{equation}
When there is no risk of confusion, we shall write $Y^\beta$ for
$Y^\beta(F,G)$, and thus we have:
\begin{equation}
\label{inj_beta_ell-0}
\ti(F, G)=\inv{|\cs_{n,p}|} \sum_{\beta\in \cs_{n,p} }
 Y^\beta.
\end{equation}

We recall from  Lov\'{a}sz \cite{Lovasz_2012_book},   Section 5.2.3,
that:
 \begin{equation}
\label{eq:t-ti}
  \left|\ti (F,G)-t(F,G)\right|\le \frac{1}{n}\binom{p}{2}.
 \end{equation}

 In the  same way, we  can define homomorphism densities  from partially
 labeled graphs.  See Figure \ref{fig:figure_hom_label} for an injective 
 homomorphism  of partially  labeled graphs.   Assume $p>k\geq  1$.  Let
 $\ell\in\mathcal{S}_{p,k}$ and  $\alpha\in\mathcal{S}_{n,k}$. We define
 $\text{Inj}(F^{\ell},G^{\alpha})$  the set  of injective  homomorphisms
 $\ff$ from  $F$ into  $G$ such  that $\ff(\ell_{i})=\alpha_{i}$
 for all $i\in[k]$, and its density:
\begin{equation}\label{hom3}
 \ti (F^{\ell},G^{\alpha})
=\frac{|\text{Inj}(F^{\ell},G^{\alpha})|}{A_{n-k}^{p-k}}
= \inv{|\cs_{n,p}^{\ell, \alpha}|} \sum_{\beta\in \cs_{n,p}^{\ell, \alpha} }
 Y^\beta.
\end{equation}

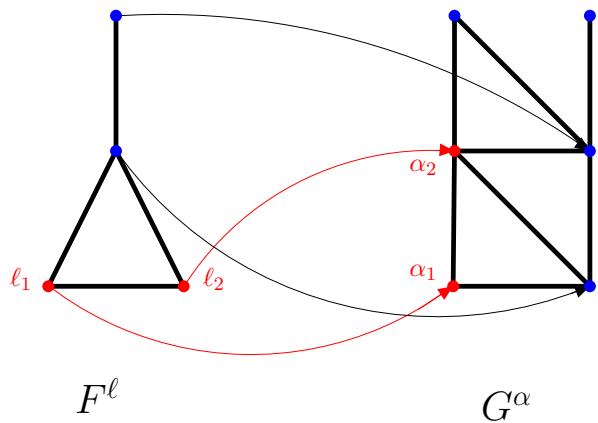
\begin{figure}[!ht]
\begin{center}
\scalebox{0.9}{
\begin{tikzpicture}[line cap=round,line join=round,>=triangle 45,x=1.0cm,y=1.0cm]
\clip(-10.51,-2.8) rectangle (0,5.02);
\draw [line width=2pt] (-9,0)-- (-8,2);
\draw [line width=2pt] (-8,2)-- (-7,0);
\draw [line width=2pt] (-9,0)-- (-7,0);
\draw [line width=2pt] (-8,4)-- (-8,2);
\draw [line width=2pt] (-3,4)-- (-3,2);
\draw [line width=2pt] (-3,2)-- (-3.02,0);
\draw [line width=2pt] (-3.02,0)-- (-1,0);
\draw [line width=2pt] (-1,2)-- (-1,0);
\draw [line width=2pt] (-3,2)-- (-1,0);
\draw [line width=2pt] (-3,2)-- (-1,2);
\draw [line width=2pt] (-1,2)-- (-1,4);
\draw [line width=2pt] (-3,4)-- (-1,2);
\draw [shift={(-3.34,-2.32)},color=ffqqqq]  plot[domain=1.49:2.58,variable=\t]({1*4.33*cos(\t r)+0*4.33*sin(\t r)},{0*4.33*cos(\t r)+1*4.33*sin(\t r)});
\draw [shift={(-6.01,3.89)},color=ffqqqq]  plot[domain=4.06:5.37,variable=\t]({1*4.9*cos(\t r)+0*4.9*sin(\t r)},{0*4.9*cos(\t r)+1*4.9*sin(\t r)});
\draw [->,color=ffqqqq] (-3.13,2.01) -- (-3,2);
\draw [->,color=ffqqqq] (-3.53,-0.34) -- (-3.02,0);
\draw (-8.74,-1.22) node[anchor=north west] {\LARGE $ F^{\ell} $};
\draw (-2.74,-1.42) node[anchor=north west] {\LARGE $ G^{\alpha}$};
\draw [color=ffqqqq](-9.7,0.4) node[anchor=north west] {$ \ell_{1} $};
\draw [color=ffqqqq](-6.85,0.4) node[anchor=north west] {$ \ell_{2} $};
\draw [color=ffqqqq](-3.8,2) node[anchor=north west] {$ \alpha_{2} $};
\draw [color=ffqqqq](-3.8,0.45) node[anchor=north west] {$ \alpha_{1} $};
\draw [shift={(-3.25,5.37)}] plot[domain=3.76:5.11,variable=\t]({1*5.82*cos(\t r)+0*5.82*sin(\t r)},{0*5.82*cos(\t r)+1*5.82*sin(\t r)});
\draw [shift={(-7.31,-6.84)}] plot[domain=0.95:1.63,variable=\t]({1*10.86*cos(\t r)+0*10.86*sin(\t r)},{0*10.86*cos(\t r)+1*10.86*sin(\t r)});
\draw [->] (-1.28,-0.11) -- (-1,0);
\draw [->] (-1.2,2.14) -- (-1,2);
\begin{scriptsize}
\fill [color=ffqqqq] (-9,0) circle (2.5pt);
\fill [color=qqqqff] (-8,2) circle (2.5pt);
\fill [color=ffqqqq] (-7,0) circle (2.5pt);
\fill [color=qqqqff] (-8,4) circle (2.5pt);
\fill [color=ffqqqq] (-3.02,0) circle (2.5pt);
\fill [color=qqqqff] (-1,0) circle (2.5pt);
\fill [color=ffqqqq] (-3,2) circle (2.5pt);
\fill [color=qqqqff] (-1,2) circle (2.5pt);
\fill [color=qqqqff] (-3,4) circle (2.5pt);
\fill [color=qqqqff] (-1,4) circle (2.5pt);
\end{scriptsize}
\end{tikzpicture}}
\end{center}
\caption{Example of an injective homomorphism from partially labeled graphs.}
\label{fig:figure_hom_label}
\end{figure}

Denote  $ F^{[\ell]}$  the labeled  sub-graph  of $F$   with  vertices
$\{\ell_1, \ldots,  \ell_k\}$ and edges:
\begin{equation}
   \label{eq:def-Eht} 
E(F^{[\ell]})=\{\{i,j\}\in E(F): i,j\in \ell\}. 
\end{equation}
For $\alpha\in \cs_{n,k}$, we set:
 \begin{equation}
\label{eq:def-Ya}
\hat Y^\alpha(F^\ell, G^\alpha)= Y^\alpha (F^{[\ell]}, G)= \prod_{\{i,j\}\in
 E(F^{[\ell]})}\1{\{\alpha_i,\alpha_j\}\in E(G)},
 \end{equation}
 For      $\beta\in      \cs^{\alpha,     \ell}_{n,p}$,      we      set
 $Y^\beta(F^\ell,  G^\alpha)=\hat{Y}^\alpha(F^\ell,  G^\alpha)\,  \tilde
 Y^\beta(F^\ell, G^\alpha)$ with:
 \begin{equation}
\label{eq:def-Yb-2}
\tilde Y^\beta(F^\ell, G^\alpha)=\!\!\!\!\prod_{\{i,j\}\in
  E(F)\backslash E(F^{[\ell]})}\!\!\!\!\ind_{\{\{\beta_i,\beta_j\}\in E(G)\}}.
 \end{equation}
Notice that $Y^\beta(F^\ell, G^\alpha)$ is equal to $Y^\beta$ defined in
\reff{eq:def-Yb}  for $\beta\in
\cs_{n,p}^{\ell,\alpha}$. 
When there is no risk of confusion, we shall write $\hat Y^\alpha,
\tilde Y^\beta$ and $Y^\beta$ for
$\hat Y^\alpha(F^\ell, G^\alpha),
\tilde Y^\beta(F^\ell, G^\alpha)$ and $Y^\beta(F^\ell,
G^\alpha)$. Remark that $\hat Y^\alpha$ is either
$0$ or $1$. By construction, we have:
\begin{equation}
   \label{eq:def-tti}
\ti(F^{\ell},G^{\alpha})=\hat Y^\alpha \, 
\tti(F^{\ell},G^{\alpha})
\quad\text{with}\quad
\tti(F^{\ell},G^{\alpha})=\inv{|\cs_{n,p}^{\ell, \alpha}|}
\, \sum_{\beta\in \cs_{n,p}^{\ell,\alpha}  } \,\, \tilde Y^\beta. 
\end{equation}

Summing \reff{hom3} over $\alpha\in \cs_{n,k}$, we get using
\reff{eq_snk} and  \reff{hom2} that:
\begin{equation}
   \label{eq:integ-tiFW0}
\inv{ |\cs_{n,k}|} \sum_{\alpha\in\mathcal{S}_{n,k}}\ti 
(F^{\ell},G^{\alpha })
 =\ti  (F, G). 
\end{equation} 
We  can generalize  this  formula  as follows.  Let $n\geq p>k>k'\geq
1$, $\ell\in \cs_{p,k}$, $\gamma\in\cs_{k,
  k'}$  and $\alpha'\in
\cs_{n,k'}$. We easily get:
\begin{equation}
   \label{eq:integ-tiFW}
\inv{|\cs_{n,k}^{\gamma, \alpha'}|}
\sum_{\alpha\in\mathcal{S}_{n,k}^{\gamma, \alpha'}} 
\ti (F^{\ell},G^{\alpha })
 =\ti  (F^{\ell_\gamma}, G^{\alpha'}). 
\end{equation}

\begin{rem}
  Let $K_{2}$ (resp. $K_{2}^{\bullet}$) denote  the complete
 graph with two vertices (resp. one of them being labeled). 
 Let $G\in\mathcal{F}$ with $n$ vertices.
 We define the degree sequence $(D_{i}(G):i\in[n])$ of the graph $G$ by,
 for $i\in[n]$:
 \begin{equation}
\label{eq:def_degree_seq}
D_i(G) = \ti\left(K_{2}^{\bullet},G^{i}\right)=\inv{n-1}\sum_{j\in
  [n]\backslash\{i\} }\1{\{i,j\}\in E(G)}.
 \end{equation}
\end{rem}

\begin{rem}
   \label{rem:isolated}
   Let   $F\in   \cf$   be   a  simple finite  graph   with   $V(F)=[p]$.    Let
   $\ell\in \cs_{p,k}$  for some  $k\in [p]$.   Assume $F_0$  is obtained
   from  $F$ by  adding $p'$  isolated vertices  numbered from  $p+1$ to
   $p+p'$, and label $\ell'$ of  those isolated vertices so that $\ell'$
   is a $\{p+1,  \ldots, p+p'\}$-word of length  say $k'=|\ell'|\leq p'$
   and $\ell\ell'\in \cs_{p+p', k+k'}$.  By convention $k'=0$ means none
   of the  added isolated vertices  is labeled. Assume $n\geq  p+p'$ and
   let $\alpha\in  \cs_{n,k}$ and  $\alpha'$ be  a $[n]$-word  such that
   $\alpha\alpha'\in \cs_{n,  k+k'}$.  Then,  it is elementary  to check
   that:
\[
\ti (F_0^{\ell\ell'},G^{\alpha\alpha'})=\ti (F^{\ell},G^{\alpha}).
\]
as well as, with $\delta_x$ the Dirac mass at $x$:
\[
\inv{|\cs_{n, k+k'}|} \sum_{\alpha\alpha'\in \cs_{n,
     k+k'}} \delta _{\ti (F_0^{\ell\ell'},G^{\alpha\alpha'})}
=
\inv{|\cs_{n,k}|} \sum_{\alpha\in \cs_{n,
     k}} \delta _{\ti (F^{\ell},G^{\alpha})}.
\]
In conclusion adding isolated vertices (labeled or non labeled) does not
change the homomorphism densities. 
\end{rem}

Finally,  we recall  an  induced  homomorphism from  $F$  to  $G$ is  an
injective  homomorphism which  preserves non-adjacency,  that is: an injective maps $\ff$ 
from $V(F)$ to $V(G)$ is an induced   homomorphism if  $\{i,j\}\in E(F)$ if and only if
$\{\ff(i),\ff(j)\}\in E(G)$.  See
Figure \ref{fig:example_inj_ind}  for an injective  homomorphism which
is not an induced homomorphism.  Let $\text{Ind}(F,G)$ denote the set of
induced homomorphisms; we denote its density by:
\begin{equation}\label{hom6}
 \tid (F,G)=\frac{|\text{Ind}(F,G)|}{A_{n}^{p}}\cdot
\end{equation}

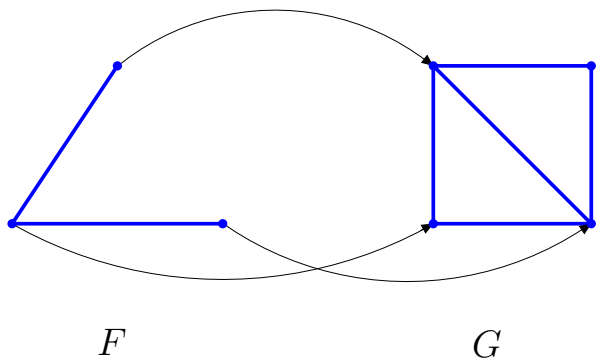
\begin{figure}[!ht]
\begin{center}
\scalebox{0.7}{
\begin{tikzpicture}[line cap=round,line join=round,>=triangle 45,x=1cm,y=1cm]
\clip(-13.7,-2.7) rectangle (0.37,5.43);
\draw [line width=2pt,color=qqqqff] (-10,3)-- (-12,0);
 \draw [line width=2pt,color=qqqqff] (-8,0)-- (-12,0);
 \draw [line width=2pt,color=qqqqff] (-4,3)-- (-4,0);
 \draw [line width=2pt,color=qqqqff] (-1,3)-- (-4,3);
 \draw [line width=2pt,color=qqqqff] (-1,0)-- (-1,3);
 \draw [line width=2pt,color=qqqqff] (-4,0)-- (-1,0);
 \draw (-10.52,-1.87) node[anchor=north west] { \huge $ F $};
 \draw (-3.4,-1.9) node[anchor=north west] { \huge $ G $};
 \draw [line width=2pt,color=qqqqff] (-4,3)-- (-1,0);
 \draw [shift={(-7,-0.69)}] plot[domain=0.89:2.25,variable=\t]({1*4.75*cos(\t r)+0*4.75*sin(\t r)},{0*4.75*cos(\t r)+1*4.75*sin(\t r)});
 \draw [shift={(-4.5,4.99)}] plot[domain=4.1:5.32,variable=\t]({1*6.09*cos(\t r)+0*6.09*sin(\t r)},{0*6.09*cos(\t r)+1*6.09*sin(\t r)});
 \draw [shift={(-8,7.05)}] plot[domain=4.2:5.23,variable=\t]({1*8.11*cos(\t r)+0*8.11*sin(\t r)},{0*8.11*cos(\t r)+1*8.11*sin(\t r)});
 \draw [->] (-4.18,3.14) -- (-4,3);
 \draw [->] (-4.24,-0.13) -- (-4,0);
 \draw [->] (-1.18,-0.12) -- (-1,0);
 \begin{scriptsize}
 \fill [color=qqqqff] (-12,0) circle (2.5pt);  
 \fill [color=qqqqff] (-10,3) circle (2.5pt);
 \fill [color=qqqqff] (-8,0) circle (2.5pt);
 \fill [color=qqqqff] (-4,3) circle (2.5pt);
 \fill [color=qqqqff] (-1,3) circle (2.5pt);
 \fill [color=qqqqff] (-1,0) circle (2.5pt);
 \fill [color=qqqqff] (-4,0) circle (2.5pt);
 \end{scriptsize}
\end{tikzpicture}}
\end{center}
\caption{An example of an  injective homomorphism but not an induced homomorphism.}
\label{fig:example_inj_ind}
\end{figure}

We  recall results  from  \cite{Lovasz_2012_book},  see Section  5.2.3.,
which  gives  relations  between   injective  and  induced  homomorphism
densities.
\begin{prop}
\label{prop:relation_t_inf}
 For $F, G\in \cf$, two simple finite graphs, we have:
\begin{equation}
   \label{eq:ti-tid}
\ti (F,G)=\sum_{F'\geq F} \tid (F',G)
\quad\text{and}\quad
\tid (F,G)=\sum_{F'\geq F} (-1)^{e(F')-e(F)}\ti (F',G),
\end{equation}
where $F'\geq   F$ means that $V(F)=V(F')$ and $E(F)\subset E(F')$,
that is $F'$ ranges over all simple graphs obtained from $F$ by adding
edges. 
\end{prop}

\subsection{Graphons}

A graphon is a symmetric, measurable function $W:[0,1]^{2} \to [0,1]$. Denote the space of all graphons by $\mathcal{W}$. 
Homomorphim densities from graphs can be extented to graphons. For every
simple finite graph $F$ and every graphon $W\in\mathcal{W}$, we define
\begin{equation}\label{hom4}
 t(F,W)=\ti (F,W)=\int_{[0,1]^{V(F)}}\prod_{\{i,j\}\in E(F)} W(x_i,x_j) \prod_{k\in V(F)}dx_k
\end{equation}
and 
\begin{equation}\label{hom7}
 \tid (F,W)=\int_{[0,1]^{V(F)}}\prod_{\{i,j\}\in E(F)} W(x_i,x_j)\prod_{\{i,j\}\notin E(F)} (1-W(x_i,x_j))\prod_{k\in V(F)}dx_k .
\end{equation}

A sequence  of simple finite graphs  $(H_{n}:n\in\Ne)$ is called  convergent if
the sequence  $(t(F,H_n):n\in\Ne)$ has  a limit  for every  simple
finite graph
$F$.       Lov\'{a}sz    and     Szegedy
\cite{Lovasz_Szegedy_2006_article}  proved  that  the limit  of a  convergent
graphs  sequence can  be represented  as  a graphon, up to a measure preserving
bijection.   In particular,  a
sequence of  graphs $(G_{n}:n\in\Ne)$ is  said to converge to  a graphon
$W$ if for every simple finite graph $F$, we have
\begin{equation*}
 \lim_{n\to \infty} t(F,H_{n})=t(F,W).
\end{equation*}

As an extension, we can define homomorphism densities from a $k$-labeled
simple finite  graph $F$  to a graphon $W$ which are  defined by not integrating
the variables corresponding to labeled vertices. Let $F\in \cf$ be a
simple finite 
graph, set $p=v(F)$ and identify $V(F)$  with $[p]$. Let $p\geq k\geq 1$
and  $\ell\in\mathcal{S}_{n,k}$.  Recall  $E(F^{[\ell]})$  defined  in
\reff{eq:def-Eht}.  We set for $y=(y_1, \ldots, y_p)\in [0, 1]^p$:
\begin{equation}
   \label{eq:def-Zy}
\tilde Z(y)=\prod_{\{i,j\}\in E(F)\backslash E(F^{[\ell]})
} W(y_i,y_j)  
\end{equation}
and  for $x=(x_{1},\dots,x_{k})\in [0, 1]^k$ we consider the average of
$\tilde Z(y)$ over $y$ restricted to $y_\ell=x$:
 \begin{equation}
\label{eq:def-tildet}
\tilde t_x(F^\ell, W)= \int_{[0,1]^{p}} \tilde Z(y)  \prod_{m\in
  [p]\setminus \ell }dy_{m} \prod_{m'\in [k]
}\delta_{x_{m'}}(dy_{\ell_{m'}}), 
 \end{equation}
as well as the analogue of $\hat Y^\alpha$  for the graphon:
 \begin{equation}
\label{eq:def-hatt}
\hat t_x(F^\ell, W)= \prod_{\{\ell_i,\ell_j\}\in
   E(F^{[\ell]})}W(x_i, x_j)
\quad\text{and}\quad
\hat t(F^\ell, W)= \int_{[0,1]^k}\hat t_{x}(F^{\ell},W)\, dx=  t(F^{[\ell]}, W).
\end{equation}
Similarly to \reff{eq:def-tti}, we set for $\ell \in \cs_{n,k}$ and $x\in [0, 1]^k$:
\begin{equation}\label{hom5}
 t_{x}(F^{\ell},W)=\hat t_x(F^\ell, W) \, \tilde  t_x(F^\ell, W).
\end{equation}
Let $\beta$ and $\beta'$ be $[k]$-words such that $\beta\beta'\in
\cs_k$, with $k'=|\beta'|$ and $1\leq k'<k$. We easily get:
\begin{equation}
   \label{eq:integ-txFW}
\int_{[0,1]^{k'}} t_x(F^\ell, W)\, dx_{\beta'}=t_{x_\beta}
  (F^{\ell_\beta}, W). 
\end{equation} 
The result also holds for $k'=k$ with the convention $t_{x_\beta}
  (F^{\ell_\beta}, W)= t
  (F, W)$ when $\beta=\emptyset$. 

\begin{rem}
The Erd\"{o}s-R\'{e}nyi case corresponds to $W\equiv \fp$ with
$0<\fp<1$, and in this case we have $t(F,W)=t_{x}(F^{\ell},W)=\fp^{e(F)}$  for all $x\in[0,1]^{k}$.
\end{rem}

\begin{rem}
The normalized  degree function $\Dw$ of the graphon $W$   is
 defined by, for all $x\in[0,1]$:
\begin{equation}
   \label{eq:def-DW}
\Dw(x)=\int_{0}^{1}W(x,y)dy.
\end{equation}
 We have for $W\in\mathcal{W}$ and $x\in [0,1]$:
 \begin{equation*}
 t_{x}(K_{2}^{\bullet},W)=\int_{0}^{1}W(x,y)dy=\Dw(x)
\quad\text{and}\quad  t(K_{2},W)=\int_{0}^{1}\Dw(x) dx.
 \end{equation*}
\end{rem}

\subsection{$W$-random graphs}
\label{sec:Gn-def}
To  complete the  identification  of  graphons as  the  limit object  of
convergent  sequences, it  has  been proved  by  Lov\'{a}sz and  Szegedy
\cite{Lovasz_Szegedy_2006_article} that we can always find a sequence of
graphs,  given by  a sampling  method, whose  limit is  a given  graphon
function.

Let $W\in\mathcal{W}$.  We can  generate a  $W$-random graph  $G_n$ with
vertex  set  $[n]$ from  the  given  graphon  $W$,  by first  taking  an
independent  sequence  $X=(X_i:i\in\Ne)$  with uniform  distribution  on
$[0,1]$,  and   then,  given  this  sequence,   letting  $\{i,j\}$  with
$i,j\in[n]$ be an edge in $G_{n}$ with probability $W(X_i,X_j)$. When we
need to stress the dependence in $W$, we shall write $G_n(W)$ for $G_n$.
For  a given  sequence $X$,  this is  done independently  for all  pairs
$(i,j)\in [n]^2$ with $i<j$.

The  random graphs  $G_{n}(W)$ thus  generalize the  Erd\"{o}s-R\'{e}nyi
random  graphs  $G_{n}(\fp)$  obtained  by  taking  $W\equiv  \fp$  with
$0<\fp<1$  constant.  (We  recall  that  the Erd\"{o}s-R\'{e}nyi  random
graph $G_{n}(\fp)$ is a random graph  defined on the finite set $[n]$ of
vertices  whose  edges occur  independently  with  the same  probability
$\fp$, $0<\fp<1$.)  Moreover, $(G_n:n\in\Ne)$ converges a.s. towards the
graphon  $W$,  see  for  instance  \cite{Lovasz_2012_book},  Proposition
11.32.

\begin{rem}
We provide elementary  computations which motivate the introduction  in
the previous section of
$\hat t_x(F^\ell,W) $ and $\tilde t_x(F^\ell,W)$. 
 Recall that $X_\gamma=(X_{\gamma_1}, \ldots, X_{\gamma_r})$
with $\gamma$ a $\N^*$-word of length $|\gamma|=r$. Let $n\geq p\geq 1$
and $F\in \cf$ with $V(F)=[p]$ and $\ell \in \cs_{p,k}$. We set for
$x=(x_1, \ldots, x_p)\in [0, 1]^p$:
\[
Z(x)=\prod_{\{i,j\}\in E(F)} W(x_i, x_j).
\]
Let $\alpha\in
\cs_{p,k}$ and $\beta\in \cs_{n,p}^{\ell, \alpha}$. 
By construction, we have:
\[
Z(X_\beta)=\E\left[Y^\beta(F, G_n)\,|\,  X \right]
=\E\left[Y^\beta(F^\ell, G_n^\alpha)\,|\,  X \right].
\]
By definition of $\hat t_x(F^\ell, W)$ and $\tilde t_x(F^\ell, W)$, we
get:
\begin{align}
   \label{eq:line-ht}
\hat t_{X_\alpha}(F^\ell, W)
&=\E\left[\hat Y^\alpha(F^\ell,  G_n^\alpha)\,|\,  X \right]
=\E\left[ Y^\alpha(F^{[\ell]},  G_n^\alpha)\,|\,  X \right]\\
   \label{eq:line-ht2}
\tilde t_{X_\alpha}(F^\ell, W)
&=\E\left[ \tilde Z(X_\beta)\,|\,  X_\alpha \right]
=\E\left[\tilde Y^\beta(F^\ell,  G_n^\alpha)\,|\,  X_\alpha \right]\\
   \label{eq:line-ht3}
 t_{X_\alpha}(F^\ell, W)
&=\E\left[  Z(X_\beta)\,|\,  X_\alpha \right]
=\E\left[ Y^\beta(F^\ell,  G_n^\alpha)\,|\,  X_\alpha \right].
\end{align}
By summing \reff{eq:line-ht2} and \reff{eq:line-ht3} over $\beta\in
\cs_{n,p}^{\ell, \alpha}$, we get, using \reff{eq:def-tti}
and \reff{eq:def-Eht}, that
\begin{equation}
   \label{eq:sum-ht2}
\tilde t_{X_\alpha}(F^\ell, W)= 
\E\left[\tti (F^\ell,  G_n^\alpha)\,|\,  X_\alpha \right]
\quad\text{and}\quad
 t_{X_\alpha}(F^\ell, W)= 
\E\left[\ti (F^\ell,  G_n^\alpha)\,|\,  X_\alpha \right].
\end{equation}
Taking the expectation in the second equality of \reff{eq:sum-ht2}, we
deduce that:
\[
t(F,W)=\int_{[0, 1]^k}  t_x(F^\ell, W)\, dx
=\E\left[t_{X_\alpha}(F^\ell, W)\right]
= \E\left[\ti (F^\ell,  G_n^\alpha)\right].
\]
Thanks to \reff{eq:integ-tiFW0}, we recover that
 \[
 t(F,W)= \E\left[\ti (F,G_{n})\right],
 \]
see also \cite{Lovasz_2012_book},  Proposition  11.32  or
  \cite{Maugis_Priebe_Olhede_Wolfe_2017_article} Proposition A.1. We
  also have:
\[
t(F, W)=\E\left[Z(X_\beta)\right]= \E\left[Y^\beta(F, W)\right].
\]
By definition of $\hat t(F^\ell, W)$, we get:
\[
\hat t(F^\ell,W)
=\E\left[\hat t_{X_\alpha}(F^\ell, W)\right]
= \E\left[\hat Y^\alpha(F^\ell, G_n^\alpha)\right]
= \E\left[ Y^\alpha(F^{[\ell]}, G_n)\right]
=t(F^{[\ell]}, W).
\]
Since $\hat Y^\alpha(F^\ell, G_n^\alpha)$ and $\tilde Y^\beta(F^\ell,
G_n^\alpha)$ are, conditionally on $X$ or $X_\beta$ or $X_\alpha$,
independent, we deduce that:
\begin{align*}
   t_{X_\alpha}(F^\ell, W)
&
= \E\left[\hat Y^\alpha(F^\ell, G_n^\alpha)\tilde Y^\beta(F^\ell,
  G_n^\alpha)\,|\,X_\alpha\right] \\
&=\E\left[\hat Y^\alpha(F^\ell, G_n^\alpha)\,|\,X_\alpha\right] \E\left[\tilde Y^\beta(F^\ell,
  G_n^\alpha)\,|\,X_\alpha\right] \\
&=  \hat t_{X_\alpha}(F^\ell, W) \,\, \tilde   t_{X_\alpha}(F^\ell, W).
\end{align*}
This latter equality gives an other interpretation of \reff{hom5}. 
\end{rem}


\section{Asymptotics for homomorphism densities of sampling partially labeled graphs from a graphon}

\subsection{Random measures associated to a graphon}

Let  $d\geq  1$   and  $I=[0,  1]^d$.  We   denote  by  $\mathcal{B}(I)$
(resp.  $\mathcal{B}^{+}(I)$)  the set  of  all  real-valued (resp.  non
negative)   measurable  functions   defined   on  $I$.   We  denote   by
$\mathcal{C}(I)$  (resp. $\mathcal{C}_{b}(I)$)  the  set of  real-valued
(resp.   bounded)    continuous   functions   defined   on    $I$.   For
$f\in\mathcal{B}(I)$  we  denote  by   $\lVert  f  \rVert_{\infty}$  the
supremum norm of $f$ on $I$.   We denote by $\mathcal{C}^{k}(I)$ the set
of real-valued functions  $f$ defined on $I$ with continuous $k$-th  derivative.
For   $f\in\mathcal{C}^{1}(I)$,    its   derivative   is    denoted   by
$\nabla   f   =(\nabla_1   f,   \ldots,   \nabla   _df)$   and   we   set
$\norm{\nabla f}_\infty =\sum_{i=1}^d \norm {\nabla_i f }_\infty $.
\\

Let $F=(F_{m}:1\le  m \le d)\in  \cf^d$ be  a finite sequence  of
simple  finite graphs. Using  Remark \ref{rem:isolated}, if necessary, we can
complete  the graphs  $F_m$ with  isolated  vertices such  that for  all
$m\in [d]$,  we have $v(F_m)=p$ for  some $p\in \N^*$ and  consider that
$V(F_m)=[p]$.  We  shall  write  $p=v(F)$  Let  $\ell\in\mathcal{M}_{p}$
(where  $\mathcal{M}_{p}$  is  the  set  of  all  $[p]$-words  with  all
characters distinct, given by  \reff{[n]_words}) and set $k=|\ell|$. We
denote:
\[
\ti (F^{\ell},G^{\alpha}_n)= \left(\ti (F_m^{\ell},G_{n}^{\alpha})  :
  m\in [d]
\right)\in [0,1]^d,
\]
and similarly for $\tti (F^{\ell},G^{\alpha}_n)$. 
 Let  $W$ be a
graphon and $x\in [0, 1]^k$. Similarly, we define $t_x (F^{\ell},W)$,
and $\tilde t_x (F^{\ell},W)$, so for example:
\[
t_x (F^{\ell},W)= \left(t_x (F_m^{\ell},W)  : m\in [d] \right)\in [0,1]^d.
\]
Notice  that relabeling  $F_m$ if  necessary,  we get  all the  possible
combinations of density  of labeled injective homomorphism  of $F_m$ into
$G$ for all  $ m\in [d]$ (we  could even take $\ell=[k]$).\\

Recall  $ F_m^{[\ell]}$ is  the labeled  sub-graph  of $F_m$   with  vertices
$\{\ell_1, \ldots,  \ell_k\}$ and set of edges $ E
(F_m^{[\ell]})=\{\{i,j\}\in E(F): i,j\in \ell\}$ see  \reff{eq:def-Eht}. For simplicity, we
shall assume the following condition which states that $ F_m^{[\ell]}$ 
does not depend on $m$:
\begin{equation}\label{eq:Hyp-hat-E}
\text{For $m,m'\in [d]$, $i,i'\in \ell$, we have:
$\{i,i'\}\in E(F_m) \Longleftrightarrow \{i,i'\}\in E(F_{m'})$.}
\end{equation}
This condition can be removed when stating the main results from Section
\ref{sec:main-res} at the cost of  very involved notations. Therefore,
we shall leave this extension to the very interested reader. \\

Let  $G_{n}=G_{n}(W)$  be  the  associated $W$-random  graphs  with  $n$
vertices constructed  from $W$  and the sequence  $X=(X_{i}:i\in\Ne)$ of
independent  uniform  random  variables  on  $[0,1]$.   Under  Condition
\reff{eq:Hyp-hat-E}, for  $\alpha\in \cs_{n,k}$ and $x\in  [0, 1]^k$, we
have    that     $\hat    Y^\alpha    (F^\ell_m,     G_n^\alpha)$    and
$\hat  t_x  (F^\ell_m,  W)$  do  not   depend  on  $m\in  [d]$.  We  set
$\hat Y^\alpha (F^\ell, G_n^\alpha)$ and $\hat  t_x (F^\ell, W)$ for the common
values.  When  there is  no  confusion,  we  write $\hat  Y^\alpha$  for
$\hat   Y^\alpha  (F^\ell,   G_n^\alpha)$.  In   particular,  we   deduce  from
\reff{eq:def-tti} that:
\begin{equation}
   \label{eq:def-ttid}
\ti (F^{\ell},G^{\alpha}_n)=\hat Y^\alpha \,\, \tti (F^{\ell},G^{\alpha}_n)
\quad\text{with}\quad
\tti (F^{\ell},G^{\alpha}_n)= \left(\tti (F_m^{\ell},G_{n}^{\alpha})  :
  m\in [d]
\right).
\end{equation}

\begin{rem}
   \label{rem:case-k=1}
If $|\ell|=k=1$, then Condition \reff{eq:Hyp-hat-E} is automatically satisfied
and we have by convention that $\hat Y^\alpha=\hat t_x (F^\ell,
W)=1$ for  $\alpha\in \cs_{n,k}$
and $x\in [0, 1]^k$. 
If $d=1$, then, Condition \reff{eq:Hyp-hat-E} is also automatically
     satisfied. 
\end{rem}

We define the random probability measure $\Gamma_n^{F,\ell}$ on
$([0,1]^d,\mathcal{B}([0, 1]^d))$ by, for $g\in\mathcal{B}^{+}([0,1]^d)$:
\begin{align}
\label{me1}
\Gamma_n^{F,\ell}(g)
&=\inv{|\cs_{n,k}|} \sum_{\alpha\in\mathcal{S}_{n,k}}
g\left(\ti  (F^{\ell},G_{n}^{\alpha})\right)\\
\nonumber
&=\inv{|\cs_{n,k}|} \sum_{\alpha\in\mathcal{S}_{n,k}}
\hat Y^\alpha g\left(\tti (F^{\ell},G_{n}^{\alpha})\right) + (1-\hat Y^\alpha) g(0),
\end{align}
where we used \reff{eq:def-ttid} and the fact that $\hat Y^\alpha$ takes
values  in $\{0,1\}$  for the  second  equality.  For  $k\in\Ne$ and $\alpha$ an
$\Ne$-word    of   length   $k$,  we   recall   the   notation
$X_{\alpha}=\left(X_{\alpha_1},\dots,X_{\alpha_k}\right)$            and
$X_{[k]}=(X_{1},\dots,X_{k})$.     Recall    \reff{eq:def-tildet}    and
\reff{eq:def-hatt}.  We define the  auxiliary random probability measure
$\hat{\Gamma}^{F,\ell}_{n}$       on       $[0,1]^d$       by,       for
$g\in\mathcal{B}^{+}([0,1]^d)$:
\begin{equation}
\label{me2}
\hat{\Gamma}^{F,\ell}_{n}(g)
=\inv{|\cs_{n,k}|}  \sum_{\alpha\in\mathcal{S}_{n,k}}
\hat t_{X_\alpha}(F^\ell, W) \, g\left(  \tilde
    t_{X_{\alpha}}(F^{\ell},W)\right) + \left(1 -\hat
    t_{X_\alpha}(F^\ell, W)  \right) g(0).
\end{equation}
and the deterministic probability measure
$\Gamma^{F,\ell}$,  by, for all
$g\in\mathcal{B}^{+}([0,1]^d)$: 
\begin{align}
\label{me3}
\Gamma^{F,\ell}(g)
&=\E\left[\hat{\Gamma}_n^{F,\ell}(g)\right]\\
\nonumber
&=\int_{[0,1]^k}\hat t_x(F^\ell, W)\, g\left(  \tilde
    t_x(F^{\ell},W)\right)\, dx + \Big(1- \hat t(F^\ell, W)\Big)
  \, g(0).
\end{align}

\begin{rem}\label{rem:CP_Gamma} $ $
\begin{itemize}
   \item[(i)]  If $d=1$ and  $g=\Id$, then we have thanks to
    \reff{eq:integ-tiFW0} that:
 \[
   \Gamma_n^{F,\ell}(\Id)
 =\inv{|\cs_{n,k}|} \sum_{\alpha\in\mathcal{S}_{n,k}}\ti
 (F^{\ell},G_{n}^{\alpha}) 
 =\ti (F,G_n)
\]
and, thanks to \reff{hom5} and \reff{eq:integ-txFW}:
\[
 \Gamma^{F,\ell}(\Id)=\int_{[0,1]^k}t_{x}(F^{\ell},W)dx=t(F,W).
\]
Notice that $\Gamma_n^{F,\ell}(\Id)$ and $\Gamma^{F,\ell}(\Id)$ do not depend on $\ell$.

\item[(ii)] If $|\ell|=1$, then according to Remark \ref{rem:case-k=1}, we
  get:
\[
\Gamma^{F,\ell}(g)
= \int_{[0,1]}  g\left( 
    t_x(F^{\ell},W)\right)\, dx.
\]
\end{itemize}
\end{rem}

\subsection{Invariance principle and its fluctuations}
\label{sec:main-res}

We first state  the  invariance principle for the random probability
meausure $\Gamma_{n}^{F,\ell}$. The  proof of the next  theorem is given
in Section \ref{sec:proof_LLN}.
\begin{theo}
\label{Theo_LLN}
Let $W\in \cw$ be a graphon. Let $F\in \cf^d$ be a sequence of $d\geq 1$
 simple finite  graphs  with  $V(F)=[p]$, $\ell\in\mathcal{M}_{p}$.  Assume
that Condition  \reff{eq:Hyp-hat-E} holds. Then, the  sequence of random
probability           measures          on           $[0,          1]^d$,
$\left(\Gamma_{n}^{F,\ell}:n\in\Ne\right)$ converges  a.s. for  the weak
topology towards $\Gamma^{F,\ell}$.
\end{theo}
The convergence of $\left(\Gamma_{n}^{F,\ell}(\Id) :n\in\Ne\right)$,
with 
 $d=1$,    can also be found in \cite{Lovasz_2012_book}, see
Proposition 11.32.

\begin{rem}
  By  Portmanteau Theorem,  we  have that   a.s.  for  all
  bounded    measurable   function    $g$   on    $[0,1]^d$   such    that
  $\Gamma^{F,\ell}(\cd_{g})=0$ where  $\cd_{g}$ is the set  of discontinuity
  points             of             $g$, 
  $\lim_{n\to\infty}\Gamma_{n}^{F,\ell}(g)=\Gamma^{F,\ell}(g)$.

  For  simplicity, consider  the case  $d=1$ and $W\equiv \fp$  with
  $0<\fp<1$. Let $\hat e(F)$ denote  the cardinal
  of $ E(F^{[\ell]})$. Because  $\Gamma^{F,\ell}=\fp^{\hat e(F)}
  \delta_{\fp^{e(F)- \hat e(F)}} + (1- \fp^{\hat e(F)})
  \delta_0$, with $k=|\ell|$,  then if $g$  is continuous at
  $\fp^{e(F)-\hat e (F)}$
  and at 0, we get that 
  a.s. $\lim_{n\to\infty}\Gamma_{n}^{F,\ell}(g)=\Gamma^{F,\ell}(g)$.
\end{rem}

The next theorem,  whose proof is given  in Section \ref{sec:proof_CLT},
gives  the fluctuations  corresponding to  the invariance  principle of
Theorem  \ref{Theo_LLN}.   Notice  the   speed  of  convergence  in  the
invariance  principle is  of  order  $\sqrt{n}$.  

For $\mu\in \R$ and $\sigma\geq 0$, we denote by $\cn(\mu, \sigma^2)$ the
Gaussian distribution with mean $\mu$ and variance $\sigma^2$. 

\begin{theo}
\label{Theo_CLT}
   Let  $W\in \cw$  be a  graphon.  Let $F\in  \cf^d$ be  a sequence  of
   $d\geq     1$        simple     finite  graphs     with     $V(F)=[p]$,
   $\ell\in\mathcal{M}_{p}$, with $k=|\ell|$. Assume
that Condition  \reff{eq:Hyp-hat-E} holds. Then, for  all
  $g\in\mathcal{C}^{2}([0,1]^d)$,  we have  the  following convergence  in
  distribution:
\[
\sqrt{n}\left(\Gamma^{F,\ell}_{n}(g)-\Gamma^{F,\ell}(g)\right)
  \cvloi{n} \mathcal{N}\left(0, \sigma^{F,\ell}(g)^{2}\right),
\]
  with $\sigma^{F,\ell}(g)^{2}
  =\Var(\cu_g^{F,\ell})$ and 
\begin{multline}
\label{eq:varU}
  \cu_g^{F,\ell}=\sum_{i=1}^{k}\int_{[0,1]^k} \hat t_{R_{i}(x,U)}(F^\ell,
  W)\, \left(g\big(\tilde t_{R_{i}(x,U)}(F^{\ell},W)\big)-g(0)\right)\,dx
\\
  + \sum_{q\in [p]\backslash \ell}\int_{[0,1]^k}dx \, 
 \langle t_{xU}(F^{\ell q},W), \nabla g\left(\tilde t_{x}(F^{\ell},W)\right) \rangle,
\end{multline}
where $U$ is a uniform random variable on $[0,1]$, and
$[p]\backslash
\ell=\{1,\dots,p\}\backslash\{\ell_{1},\dots,\ell_{k}\}$. 
 \end{theo}

\begin{rem}\label{rem:variance_TCL}
Let $U$ be  a uniform random variable on $[0,1]$.
 \begin{itemize}
  \item[(i)] Assume that $|\ell|=k=1$. Using Remark \ref{rem:case-k=1},
    we get for $\ell\in[p]$:
  \begin{equation}\label{var_k=1}
   \sigma^{F,\ell}(g)^{2}=\Var\left(g\left(t_{U}(F^{\ell},W)\right)+
     \sum_{q\in [p]\backslash {\ell}}\int_{[0,1]}\langle t_{xU}(F^{\ell
       q},W), \nabla g\big(t_{x}(F^{\ell},W)\big) \rangle\, dx \right).
  \end{equation}

\item[(ii)] Let $F\in \cf^d$ with $p=v(F)$. Take $\ell=1$
 with $k=|\ell|=1$.  
 Let  $a\in\R^d$ and consider  $g(x)=\langle a,x \rangle$
for $x\in \R^d$. We deduce from \reff{var_k=1} and  \reff{eq:integ-txFW} that:
\begin{equation}
\label{eq:var_quantum_graph}
  \sigma^{F,\ell}(g)^{2}
  =\Var\left( \langle a,  \sum_{q=1}^{p} t_{U}\left(F^{q},W\right)
\rangle\right).
\end{equation}

\item[(iii)] In  the case  $d=1$, $F\in \cf$,  and $g=\Id$,  the central
  limit           theorem          appears           already          in
  \cite{Feray_Meliot_Nikeghbali_2017_article}.   In this  case, we  have
  $\Gamma^{F,\ell}_{n}(\Id)=\ti (F,G_n)$, $\Gamma^{F,\ell}(\Id)=t(F, W)$
  and, thanks to \reff{eq:var_quantum_graph} (with $d=1$ and $a=1$):
\begin{equation}
\label{eq:var_d=1Id}
  \sigma^{F,\ell}(\Id)^{2}
  =\Var\left(\sum_{q=1}^{p} t_{U}\left(F^{q},W\right)
\right).
\end{equation}

Let $F, F'\in \cf$ be two simple finite graphs, let $i\in V(F)$ and $i'\in V(F')$.
We define a new graph $(F \bowtie F')(i,i')=(F\sqcup F')/\{i\sim i'\}$ which
is the  disjoint union of  $F$ and $F'$ followed  by a quotient  where we
identify the vertex  $i$ in $V(F)$ with the vertex  $i'$ in $V(F')$, see
Figure \ref{fig:graphs_connected_vertices}.

\begin{figure}[ht!]
\begin{center}
\scalebox{0.8}{
\begin{tikzpicture}[line cap=round,line join=round,>=triangle 45,x=1.0cm,y=1.0cm]
\clip(-11.2,-1.5) rectangle (3.5,5.38);
\draw (-10,0)-- (-9,2);
\draw (-5,4)-- (-5,2);
\draw (-8,0)-- (-9,2);
\draw (-10,0)-- (-8,0);
\draw (-6,0)-- (-4,0);
\draw (-4,0)-- (-5,2);
\draw (-5,2)-- (-6,0);
\draw (-2,0)-- (-1,2);
\draw (0,0)-- (-1,2);
\draw (0,0)-- (-2,0);
\draw (0,0)-- (1,2);
\draw (2,0)-- (1,2);
\draw (1,4)-- (1,2);
\draw (0,0)-- (2,0);
\draw (-2,0)-- (0,0);
\draw (-9.4,-0.26) node[anchor=north west] {\LARGE $ F $};
\draw (-5.3,-0.34) node[anchor=north west] {\LARGE $ G $};
\draw (-1.8,-0.3) node[anchor=north west] {\LARGE $ (F \bowtie G)(2,4) $};
\draw (-9.2,2.7) node[anchor=north west] {$ 1 $};
\draw (-7.86,0.4) node[anchor=north west] {$ 2 $};
\draw (-10.6,0.4) node[anchor=north west] {$ 3 $};
\draw (-5.2,4.7) node[anchor=north west] {$ 1 $};
\draw (-4.8,2.3) node[anchor=north west] {$ 2 $};
\draw (-3.82,0.4) node[anchor=north west] {$ 3 $};
\draw (-6.6,0.4) node[anchor=north west] {$ 4 $};
\begin{scriptsize}
\fill [color=qqqqff] (-10,0) circle (2.5pt);
\fill [color=qqqqff] (-9,2) circle (2.5pt);
\fill [color=ffqqqq] (-8,0) circle (2.5pt);
\fill [color=qqqqff] (-5,2) circle (2.5pt);
\fill [color=qqqqff] (-4,0) circle (2.5pt);
\fill [color=ffqqqq] (-6,0) circle (2.5pt);
\fill [color=qqqqff] (-5,4) circle (2.5pt);
\fill [color=qqqqff] (-1,2) circle (2.5pt);
\fill [color=ffqqqq] (0,0) circle (2.5pt);
\fill [color=qqqqff] (-2,0) circle (2.5pt);
\fill [color=qqqqff] (1,2) circle (2.5pt);
\fill [color=qqqqff] (2,0) circle (2.5pt);
\fill [color=qqqqff] (1,4) circle (2.5pt);
\end{scriptsize}
\end{tikzpicture}}
\end{center}
\caption{Example of two graphs connected by two vertices.}
\label{fig:graphs_connected_vertices}
\end{figure}
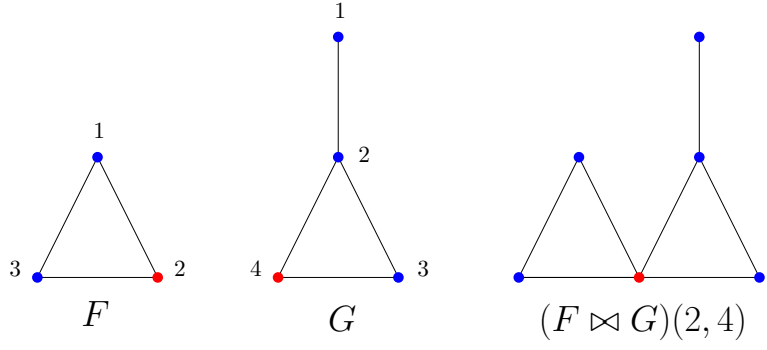

With this notation, we have:
  \begin{align}
\nonumber
\sigma^{F,\ell}(\Id)^{2}
   &=\E\left[\left(\sum_{q=1}^{p}t_{U}(F^{q},W)\right)^2\right]
-\E\left[\sum_{q=1}^{p}t_{U}(F^{q},W)\right]^2\\
\nonumber
   &=\sum_{q,q'=1}^{p}\int_{0}^{1}t_{x}(F^{q},W)t_{x}(F^{q'},W)dx 
-\left(\sum_{q=1}^{p}\int_{0}^{1}t_{x}(F^{q},W)dx\right)^2\\
   &=\sum_{q,q'=1}^{p}t\big((F\bowtie F)(q,q'),W\big) - p^2\,  t(F,W)^{2}.
\label{eq:s2=FF}
  \end{align}
Thus, we recover the  limiting variance given  in \cite{Feray_Meliot_Nikeghbali_2017_article}.

\item[(iv)] Let $d=1$. We consider the two degenerate cases where no vertex
is labelled ($k=0$) or all vertices are labelled ($k=p$):
 \begin{enumerate}
  \item[(a)] for $k=0$, we apply the $\delta$-method to \reff{eq:var_d=1Id}, to get that
  \begin{equation*}
   \sqrt{n}\left[g(\ti(F,G_n))-g(t(F,W))\right]\cvloi{n}\mathcal{N}\left(0,\sigma^{F}(g)^2\right),
  \end{equation*}
  where 
  \begin{equation}\label{eq:var_k=0}
   \sigma^{F}(g)^2=g'(t(F,W))^2\sigma^{F,\ell}(\text{Id})^2.
  \end{equation}

  \item[(b)] for $k=p$, we have $\Gamma_{n}^{F,\ell}(g)=(g(1)-g(0))\ti(F,G_n)+g(0)$,
  $\Gamma^{F,\ell}(g)=(g(1)-g(0))t(F^{\ell},W)dx+g(0)$ and 
  \begin{equation}\label{eq:var_k=p}
   \sigma^{F,\ell}(g)^2=(g(1)-g(0))^2\sigma^{F,\ell}(\text{Id})^2.
  \end{equation}
 \end{enumerate}
   \item[(v)] Let $d=1$ and $F=K_{2}$. We have $\Gamma_n^{K_2,
       \ell}(\text{Id})=t(K_2, G_n)$ and we deduce from \reff{eq:var_d=1Id}
that  $\sigma^{K_2,\ell}(\text{Id})^2=4\Var\left(\Dw(U)\right)$.

\begin{enumerate}
  \item[(a)]    If  $k=1$, then we have $\Gamma_n^{K_2^{\bullet},
       \ell}(g)=\inv{n}\sum_{i=1}^n g(D_i^{(n)})$ with
     $D_i^{(n)}=D_i(G_n)$  the normalized degree of $i$ in $G_n$, see
     \reff{eq:def_degree_seq}. We deduce from \reff{var_k=1} that:
  \begin{equation*}
   \sigma^{K_{2}^{\bullet},\ell}(g)^{2}
=\Var\left(g(\Dw(U))+\int_{0}^{1}W(x,U)g'(\Dw(x))\,dx\right).  
  \end{equation*}

\item[(b)] If $k=2$, using (iv)-(2), we get from \reff{eq:var_k=p} that: 
  \begin{equation*}
     \sigma^{K_{2}^{\bullet\bullet},\ell}(g)^{2}
=4(g(1)-g(0))^2\Var\left(\Dw(U)\right),
  \end{equation*}
 where $K_{2}^{\bullet\bullet}$ denotes the complete graph $K_2$ with two labeled vertices.

\item[(c)]   Finally, if $k=0$, using (iv)-(1),  we get from \reff{eq:var_k=0} that
  \begin{equation*}
   \sigma^{K_{2}}(g)^{2}
=4\,g'\left(\int_{0}^{1}\Dw(x)dx\right)^2\Var\left(\Dw(U)\right).
  \end{equation*}
\end{enumerate}
 \item[(vi)]  Thanks to \reff{eq:t-ti}, we get that Theorems \ref{Theo_LLN} and
  \ref{Theo_CLT}  also hold with $\ti$ replaced by $t$. 
\item[(vii)] It is left to reader to check that theorem \ref{Theo_CLT} is
  degenerated,    that    is    $\sigma^{F,    \ell}(g)=0$,    in    the
  Erd\"{o}s-R\'{e}nyi   case,   that   is   $W\equiv   \fp$   for   some
  $0\leq \fp\leq 1$, or when $\int_{[0, 1]^k} \hat t_x(F,W)\, dx =0$ and
  in particular when $t(F,W)=0$.
 \end{itemize}
\end{rem}

The next 
 corollary gives the limiting Gaussian process for
the fluctuations of $(\ti (F, G_n): F\in \cf)$. 
\begin{cor}\label{cor:Theo_CLT_multi}
 We have the following convergence of finite-dimensional distributions:
 \begin{equation*}
  \left(\sqrt{n}\left(\ti
      (F,G_{n})-t(F,W)\right): F\in\mathcal{F}\right)
  \,\xrightarrow[n\rightarrow+\infty]{(fdd)}\, \Theta_{\rm{inj}} ,
 \end{equation*}
where $\Theta_{\rm{inj}}=(\Theta_{\rm{inj}}(F):  F\in\mathcal{F})$ is a centered Gaussian
process with covariance function $K_{\rm{inj}}$ given, for $F,F'\in\mathcal{F}$, with $V(F)=[p]$ and $V(F')=[p']$, by:
\begin{align}
\label{eq:K1}
    K_{\rm{inj}}(F,F')
&=\Cov\left(\sum_{q=1}^{p}t_{U}(F^{q},W),\, \sum_{q=1}^{p'}t_{U}\big({F'}^{q},W\big)\right)\\
\label{eq:K2}
&=\sum_{q=1}^{p}\sum_{q'=1}^{p'}t\left((F\bowtie F')(q,q'),W\right) -
pp'\, \, t(F,W)t(F',W).
\end{align}
\end{cor}

\begin{proof}
We deduce from \reff{eq:var_quantum_graph} and standard results on
Gaussian vectors, the  convergence,
for the finite-dimensional  distributions towards the Gaussian process with covariance function
given by \reff{eq:K1}. Formula \reff{eq:K2} can be derived similarly to  \reff{eq:s2=FF}. 
\end{proof}

 \begin{rem}
   In  particular, Corollary  \ref{cor:Theo_CLT_multi} proves  a central
   limit     theorem    for     quantum    graphs     (see    Lov\'{a}sz
   \cite{Lovasz_2012_book},  Section 6.1).   A simple  quantum graph  is
   defined as a  formal linear combination of a finite  number of simple
   finite graphs with real coefficients.  This  definition makes it possible to
   study   linear    combination   of   homomorphism    densities.   For
   $F=(F_m:  m\in  [d])\in \cf^d$  and  $a=(a_m:m\in  [d])\in \R^d$,  we
   define the homomorphism  density of $\fF=\sum_{m=1}^d a_m  F_m$ for a
   graph      $G$     and      a     graphon      $W\in     \cw$      as
   $\ti(\fF,    G)=    \langle    a,     \ti(F,    G)    \rangle$    and
   $t(\fF, W)=  \langle a,  t(F, W) \rangle$.  We deduce  from Corollary
   \ref{cor:Theo_CLT_multi} the following convergence in distribution:
  \begin{equation}\label{eq:cv_quantum_graph}
   \sqrt{n}\left(\ti (\fF,G_{n})-t (\fF,W)\right)
   \cvloi{n} \mathcal{N}\left(0, \sigma(\fF)^{2}\right),
 \end{equation}
 where $\sigma(\fF)^{2}$ is given by \reff{eq:var_quantum_graph} or
 equivalently $\sigma(\fF)^{2}=\sum_{m,m'\in [d]} a_ma_{m'}\,   K_{\rm{inj}}(F_m, F_{m'})$. 
\end{rem}

We also have the limiting Gaussian process for the fluctuations of 
$(\tid(F,G_n):F\in\cf)$.
\begin{cor}\label{cor:CLT_tind}
We have the following convergence of finite-dimensional distributions:
 \begin{equation*}
  \left(\sqrt{n}\left(\tid
      (F,G_{n})-\tid(F,W)\right): F\in\mathcal{F}\right)
  \,\xrightarrow[n\rightarrow+\infty]{(fdd)}\, \Theta_{\rm{ind}} ,
 \end{equation*}
 where $\Theta_{\rm{ind}}=(\Theta_{\rm{ind}}(F):  F\in\mathcal{F})$ is a
 centered  Gaussian  process  with  covariance  function  $K_{\rm{ind}}$
 given,    for    $F_1,F_2\in\mathcal{F}$,   with    $V(F_1)=[p_1]$    and
 $ V(F_2)=[p_2]$, by:
\begin{multline}
\label{eq:K_ind}
 K_{\rm{ind}}(F_1,F_2)\\
\begin{aligned}
&=\Cov\left(\sum_{F'_1\geq F_1}(-1)^{e(F'_1)} \sum_{q=1}^{p_1}t_{U}\left((F'_1)^{q},W\right),\, 
\sum_{F'_2\geq F_2}(-1)^{e(F'_2)} \sum_{q=1}^{p_2}t_{U}\left((F'_2)^{q},W\right)\right) \\
&=\sum_{F'_1\geq F_1, \, F'_2\geq F_2}(-1)^{e(F'_1)+e(F'_2)}
\left(\sum_{q=1}^{p_1}\sum_{q'=1}^{p_2}t\left((F'_1\bowtie F'_2)(q,q'),W\right) -p_1p_2\, \, t(F'_1,W)t(F'_2,W)\right),
\end{aligned}
\end{multline}
where $F'\geq  F$ means that $V(F)=V(F')$ and $E(F)\subset E(F')$,
that is $F'$ ranges over all simple graphs obtained from $F$ by adding
edges. 
\end{cor}

\begin{proof}
  Notice  that $\tid  (F,G_{n})$  is a  linear  combination of  subgraph
  counts  by  Proposition  \ref{prop:relation_t_inf}.   We  deduce  from
  \reff{eq:var_quantum_graph} and standard  results on Gaussian vectors,
  the convergence, for the  finite-dimensional distributions towards the
  Gaussian process with covariance function  given by the first equality
  in  \reff{eq:K_ind}  which  is  derived from  the  second  formula  of
  \reff{eq:ti-tid}     and     \reff{eq:var_quantum_graph}.    The second  equality  of
  \reff{eq:K_ind}  
  can  be derived  similarly  to  \reff{eq:K2}. 
\end{proof}

\section{A preliminary result}\label{sec:prem_result}
Let  $F=(F_m: m\in  [d])\in \cf^d$  be a  sequence of  $d\geq 1$  simple
finite  graphs with  $p=v(F)$, $\ell\in\mathcal{M}_{p}$  with $|\ell|=k$
such  that Condition  \reff{eq:Hyp-hat-E} holds.   Let $W\in  \cw$ be  a
graphon  and $X=(X_{i}:i\in\Ne)$  be a  sequence of  independent uniform
random  variables on  $[0,1]$.  Let  $n\in\Ne$  such that  $n> p$.   Let
$G_{n}=G_{n}(W)$  be  the  associated $W$-random  graphs  with  vertices
$[n]$,    see    Section     \ref{sec:Gn-def}.     Recall    definitions
\reff{eq:def-Yb}    of    $Y^\beta(F,     G)$,    \reff{eq:def-Ya}    of
$\hat    Y^\alpha(F^\ell,   G^\alpha)$    and   \reff{eq:def-Yb-2}    of
$\tilde  Y^\beta(F^\ell, G^\alpha)$  for  a simple finite  graph  $F$.  We  set
$Y^\beta=(Y^\beta(F_m,  G_n):m\in [d])$  for $\beta\in\mathcal{S}_{n,p}$
and $\tilde Y^\beta=(\tilde Y^\beta(F_m^\ell, G_n^\alpha):m\in [d])$ as
well as $Y^\alpha=Y^\alpha(F^\ell_m, G^\alpha_n)$ (which does not depend on
$m\in [d]$) for
$\beta\in  \cs_{n,p}^{\ell,\alpha}$ and  $\alpha\in \cs_{p,k}$.   Notice
that for  $\alpha\in \cs_{p,k}$ and  $\beta\in \cs_{n,p}^{\ell,\alpha}$,
we have that,  conditionally to $X$,  $\hat Y^\alpha$
and $\tilde Y^\beta$ are independent, $\hat Y^\alpha$ is  a Bernoulli random variable and:
\[
Y^\beta=\hat Y^\alpha  \, \tilde Y^\beta.
\]
Recall that  $\ti \left(F^{\ell},G_n^{\alpha}\right)=\left(\ti
  \left(F_m^{\ell},G_n^{\alpha}\right): m\in [d]\right)$.
With these notations, we get from equation \reff{eq:def-tti} that for $\ell\in\mathcal{M}_{p}$ with $|\ell|=k$ and  $\alpha\in\mathcal{S}_{n,k}$:
\begin{equation}
\label{t_inj_y_beta}
 \ti \left(F^{\ell},G_n^{\alpha}\right)
=\inv{|\cs_{n,p}^{\ell, \alpha}|} 
   \,\sum_{\beta\in\mathcal{S}_{n,p}^{\ell,\alpha}}
    \, Y^{\beta}
=\hat Y^\alpha  \,\,  \tti \left(F^{\ell},G_n^{\alpha}\right),
\end{equation}
with
\begin{equation}
   \label{eq:def-tilde-tinj}
 \tti \left(F^{\ell},G_n^{\alpha}\right)
=\inv{|\cs_{n,p}^{\ell, \alpha}|} \,
 \,  \sum_{\beta\in\mathcal{S}_{n,p}^{\ell,\alpha}}
    \, \tilde Y^{\beta}.
\end{equation}

 We also set $Z^\beta=\E[Y^\beta|X]$ and $\tilde Z^\beta=\E[\tilde
Y^\beta |X]$. Recall \reff{eq:def-Zy}. We have, for  $Z^\beta=(Z^\beta_m:m\in [d])$ and $\tilde
Z^\beta=(\tilde Z^\beta_m:m\in [d])$ 
that  for $m\in
[d]$:
\[
Z_m^{\beta}=\prod_{\{i,j\}\in
  E(F_m)}W(X_{\beta_i},X_{\beta_{j}})
\quad\text{and}\quad
\tilde Z_m^{\beta}=\prod_{\{i,j\}\in
  \tilde E(F_m^\ell)}W(X_{\beta_i},X_{\beta_{j}})=\tilde Z_m(X_\beta). 
\]
We recall that
$\hat t_{X_\alpha}(F^\ell, W)=\E\left[\hat Y^\alpha |\,
  X\right]=\E\left[\hat Y^\alpha |\, X_\alpha\right]$, see
\reff{eq:line-ht}, to deduce that:
\begin{equation}
   \label{eq:Zb=ttZb}
Z^\beta = \hat t_{X_\alpha}(F^\ell, W)\, \tilde Z^\beta . 
\end{equation}

\begin{lem}\label{lem_1}
Let $F\in \cf^d$ be a sequence  of $d\geq 1$  simple  finite graphs with
$p=v(F)$, $\ell\in\mathcal{M}_{p}$ and $W\in\cw$ be  a graphon.
Let $(M_{\beta}:\beta\in\mathcal{S}_{n,p})$ be a sequence of
$\sigma\left(X\right)$-measurable $\R^d$-valued random variables and $n>
p$. Assume Condition \reff{eq:Hyp-hat-E} holds and  that there exists a 
finite constant $K$ such that for all $\beta\in\mathcal{S}_{n,p}$, we
have $\E\left[| M_{\beta} |^2 \right]\le K$. 
Then we have:
 \[\E\left[\left(\inv{|\cs_{n,p}|}
  \,\sum_{\beta\in\mathcal{S}_{n,p}}\, 
\langle Y^{\beta}-Z^{\beta},\, M_{\beta} \rangle \right)^2\right]\le d K
\frac{p(p-1)}{8n(n-1)}\cdot 
 \]
\end{lem}

 \begin{proof}
We first  assume that  $d=1$.  We denote  by
   $\Cov(\left..\right|X)$  the conditional  covariance  given $X$.   We
   have:
\begin{multline*}
\E\left[\left(\inv{|\cs_{n,p}|}
  \,\sum_{\beta\in\mathcal{S}_{n,p}}\, 
    \left(Y^{\beta}-Z^{\beta}\right)M_{\beta}\right)^2\right]\\
\begin{aligned}
 &=\inv{|\cs_{n,p}|^2}\,\sum_{\beta\in\mathcal{S}_{n,p}}
 \sum_{\gamma\in\mathcal{S}_{n,p}}
 \E\left[\E\left[\big(Y^{\beta}-\E[Y^{\beta}
   |X]\big)\big(Y^{\gamma}-\E[Y^{\gamma}
   |X]\big)M_{\beta}M_{\gamma}\Big|\, X\right]\right]\\
&=\inv{|\cs_{n,p}|^2}\,\sum_{\beta\in\mathcal{S}_{n,p}}
 \sum_{\gamma\in\mathcal{S}_{n,p}}
 \E\left[M_{\beta}M_{\gamma}\Cov(Y^{\beta},Y^{\gamma}|\,X)\right]\\
 &\le \inv{|\cs_{n,p}|^2}\,
 \sum_{\beta\in\mathcal{S}_{n,p}}
 \sum_{\gamma\in\mathcal{S}_{n,p}}
\E\left[\lvert M_{\beta}   M_{\gamma} \rvert
 \, \lvert\Cov(Y^{\beta},Y^{\gamma} |\, X)\rvert \right].
\end{aligned}
\end{multline*}

If the $[n]$-words $\beta$  and $\gamma$ have at  most one character in  common, that is
$|\beta\bigcap \gamma|\leq 1$, then, by
construction,  $Y^{\beta}$    and
$Y^{\gamma}$    are  conditionally  on $X$ independent. This implies
then that $\Cov(Y^{\beta},Y^{\gamma}|\, X)=0$. 
If $|\beta\bigcap \gamma|>1$, then as $Y^{\beta}$ and  $Y^{\gamma}$  are
Bernoulli random  variables and   we have the upper bound
$\lvert\Cov(Y^{\beta},Y^{\gamma}|\, X)\rvert\le
1/4$. The number of possible choices for $\beta,\gamma\in \cs_{n,p} $
such that $|\beta\bigcap \gamma|>1$ is bounded from above by
$A_{n}^{p}\binom{p}{2}A_{n-2}^{p-2}$. We deduce that:
\begin{align*}
 \E\left[\left(\inv{|\cs_{n,p}|}
  \,\sum_{\beta\in\mathcal{S}_{n,p}}\, 
    \left(Y^{\beta}-Z^{\beta}\right)M_{\beta}\right)^2\right]
 &\le\frac{1}{4 (A_{n}^{p})^2}\,A_{n}^{p}\,\binom{p}{2}A_{n-2}^{p-2}\,\E\left[\lvert
   M_{\beta}  M_{\gamma} \rvert\right] \\
 &\le K \frac{p(p-1)}{8n(n-1)}, 
\end{align*}
where we used  the Cauchy-Schwarz inequality for the  last inequality to
get
$\E\left[\lvert M_{\beta} M_{\gamma} \rvert\right] \leq
K$.\\

In the case $d\geq 1$, the term $\E\left[\lvert
   M_{\beta} M_{\gamma} \rvert\right] $ in the above
 inequalities has to be replaced by $\E\left[\lvert
  M_{\beta}\rvert_1\, \lvert M_{\gamma} \rvert_1\right] $, where $\lvert
\cdot\rvert_1$ is the $L^1$ norm in $ \R^d$. Then,  use 
that $|x|_1^2\leq  d |x|^2$ and thus $\E\left[\lvert
  M_{\beta}\rvert_1\, \lvert M_{\gamma} \rvert_1\right] \leq  d K$ to conclude.
 \end{proof}

The proof of the next Lemma is similar and left to the reader (notice
the next lemma is in fact Lemma \ref{lem_1} stated for $d=1$ and the graph
$F^{[\ell]}_m$ of the labeled vertices, see Section
\ref{subsec:graph_hom} for the definition of $F^{[\ell]}_m$,
which thanks to condition \reff{eq:Hyp-hat-E},
does not depend on $m\in [d]$). We recall that 
$\hat t_{X_\alpha}(F^\ell, W)=\E\left[\hat Y^\alpha |\,
  X\right]=\E\left[\hat Y^\alpha |\, X_\alpha\right]$, see
\reff{eq:line-ht}.

\begin{lem}
\label{lem:Delta-Ya}
Let $F\in \cf^d$ be a sequence  of $d\geq 1$ simple   finite  graphs with
$p=v(F)$  and $W\in\cw$ be  a graphon.
Let  $k\in [p]$ and $(M_{\alpha}:\alpha\in\mathcal{S}_{n,k})$ be a sequence of
$\sigma\left(X\right)$-measurable $\R^d$-valued random variables and $n>
p$. Assume  Condition \reff{eq:Hyp-hat-E} holds and that there exists a 
finite constant $K$ such that for all $\alpha\in\mathcal{S}_{n,k}$, we
have $\E\left[| M_{\alpha} |^2 \right]\le K$. 
Then we have:
 \[\E\left[\left(\inv{|\cs_{n,k}|} 
  \,\sum_{\alpha\in\mathcal{S}_{n,k}}\, 
\langle \hat Y^{\alpha}-\hat t_{X_\alpha}(F^\ell, W) , \,M_{\alpha} \rangle \right)^2\right]\le d K
\frac{k(k-1)}{8n(n-1)}\cdot 
 \]
\end{lem}

We  also state  a variant of  Lemma \ref{lem_1}, when working  conditionally on $X_{\alpha}$ for
some  $\alpha\in\mathcal{S}_{n,k}$. 

The next  result is a key ingredient in the proof of  Theorems \ref{Theo_LLN} and \ref{Theo_CLT}.
Recall $\tilde t_x\left(F^{\ell},W\right)= \left(\tilde
  t_x\left(F^{\ell}_m,W\right): m\in [d]\right)$ with $\tilde t_x$
defined in \reff{eq:def-tildet}. Notice that for all $\beta\in
\cs_{n,p}^{\ell,\alpha}$:
\[
\tilde t_{X_{\alpha}}\left(F^{\ell},W\right)=\E\left[\tilde
  Z^\beta|X_\alpha\right]=\E\left[
  \tti (F^\ell, G_n)|X_\alpha\right].
\]

\begin{lem}
\label{lem:tt-tt}
Let $F\in \cf^d$ be a sequence  of $d\geq 1$  simple  finite graphs with
$p=v(F)$, $\ell\in\mathcal{M}_{p}$ with $k=|\ell|$,
$\alpha\in\mathcal{S}_{n,k}$  and $W\in\mathcal{W}$ be a graphon.
Assume  Condition \reff{eq:Hyp-hat-E} holds. Then, we have: 
\[
\E\left[\Big|
\tti \left(F^{\ell},G_n^{\alpha}\right)
- \tilde
t_{X_{\alpha}}\left(F^{\ell},W\right)\Big|^2\, \Big|\, X_\alpha, \hat Y^\alpha \right]
\le 
d \frac{(p-k)}{4(n-k)}\cdot
\]
\end{lem}

\begin{proof}
We consider the case  $d=1$. 
Recall the definition of $\tti \left(F^{\ell},G_n^{\alpha}\right)$ given
in \reff{eq:def-tilde-tinj}. Set:
\[
\ca=\E\left[\inv{|\cs_{n,p}^{\ell, \alpha}|^2}\Bigg(
  \,\sum_{\beta\in\mathcal{S}_{n,p}^{\ell,\alpha}}\, 
    \left(\tilde Y^{\beta}-\tilde
t_{X_{\alpha}}\left(F^{\ell},W\right)\right)\Bigg)^2\,\Big |\, X_\alpha,
\hat Y^\alpha\right].
\]
Following the proof of Lemma \ref{lem_1} with
   $M_\beta=1$, and using also that
$\E\left[\tilde Y^\beta \, \big|\, X_\alpha, \hat Y^\alpha\right]
= \tilde
t_{X_{\alpha}}\left(F^{\ell},W\right)$, and  that $\tilde Y^\beta$ and $\tilde Y^\gamma$ are
   conditionally on $X_\alpha$ independent of $\hat Y^\alpha$ for
   $\beta, \gamma\in \cs_{n,p}^{\ell,\alpha}$, we get:
\[
\ca
\le \inv{|\cs_{n,p}^{\ell, \alpha}|^2}
 \sum_{\beta\in\mathcal{S}_{n,p}^{\ell,\alpha}}
 \sum_{\gamma\in\mathcal{S}_{n,p}^{\ell,\alpha}}
\lvert\Cov(\tilde Y^{\beta},\tilde Y^{\gamma} |\, X_\alpha)\rvert.
\]
If  $\beta$ and
$\gamma$   have   no   more   than   $\alpha$   in   common,   that   is
$\beta\bigcap     \gamma=\alpha$,    then     $\tilde Y^{\beta}$    and
$\tilde Y^{\gamma}$ are  conditionally on $X_\alpha$   independent and thus 
$\Cov( \tilde Y^{\beta},\tilde Y^{\gamma} |\, X)=0$. 
 
If $|\beta\bigcap \gamma|>|\alpha|$, then as $\tilde Y^{\beta}$ and
$\tilde Y^{\gamma}$  are
Bernoulli random  variables, we have the upper bound
$\lvert\Cov(\tilde Y^{\beta},\tilde Y^{\gamma}|\, X)\rvert\le
1/4$. The number of possible choices for $\beta,\gamma\in
\cs_{n,p}^{\ell,\alpha} $
such that $|\beta\bigcap \gamma|>|\alpha|$ is bounded from above by
$A_{n-k}^{p-k}(p-k) A_{n-k-1}^{p-k-1}$. We deduce that:
\[
\ca\le \frac{1}{4(A_{n-k}^{p-k})^2} A_{n-k}^{p-k}(p-k) A_{n-k-1}^{p-k-1}
\leq  \frac{(p-k)}{4(n-k)}\cdot
\]
The extension to $d\geq 1$ is direct. 
\end{proof}

\section{Proof of Theorem \ref{Theo_LLN}}\label{sec:proof_LLN}

We first state a preliminary lemma.
 \begin{lem}\label{lem2}
 Let $F\in \cf^d$ be a sequence  of $d\geq 1$   simple finite graphs with
$p=v(F)$, $\ell\in\mathcal{M}_{p}$ with $k=|\ell|$, and
$W\in\mathcal{W}$ be a graphon.
Assume  Condition \reff{eq:Hyp-hat-E} holds. Then, for all $n> k$ and
   $g\in\cc^{1}([0,1])$, we have:
\[
\E\left[\left|\Gamma^{F,\ell}_{n}(g)-\hat{\Gamma}^{F,\ell}_{n}(g)\right|\right]\le
d\norm{g}_\infty  \sqrt{\frac{ k(k-1)}{2n(n-1)}} +
\frac{1}{2} \norm{\nabla g}_\infty \sqrt{\frac{p-k}{n-k}}\cdot
\] 
 \end{lem}
 
\begin{proof}
We first consider the case $d=1$. Let $g\in\cc^{1}([0,1])$. We first
assume   that $g(0)=0$. 
Then, we deduce from the definition \reff{me1} of $\Gamma_n^{F,\ell}$
and from \reff{t_inj_y_beta} and \reff{eq:def-tilde-tinj}, as $\hat
Y^\alpha\in \{0, 1\}$, that:
\[
\Gamma_n^{F,\ell}(g)=\inv{|\cs_{n,k}|}\sum_{\alpha\in\mathcal{S}_{n,k}}
\hat Y^\alpha \, g\left( \tti \left(F^{\ell},G_n^{\alpha}\right)\right).
\]
And thus, using definition \reff{me2} 
of $\hat \Gamma_n^{F, \ell}$, we get  $\big| \Gamma_n^{F,\ell}(g) - \hat{\Gamma}^{F,\ell}_{n}(g)\big|
\leq  B_1+B_2$ with
\[
B_1=\inv{|\cs_{n,k}|} \Big| \sum_{\alpha\in \cs_{n,k}} \left(\hat Y^\alpha - \hat
t_{X_\alpha}(F^\ell, W)\right)\, 
g\left( \tilde t_{X_\alpha}(F^\ell, W)\right) \Big| 
\]
 and 
\[
B_2= \inv{|\cs_{n,k}|}  \sum_{\alpha\in \cs_{n,k}} 
\hat Y^\alpha \Big |g\left( \tti (F^\ell, G_n^\alpha)\right)-
g\left( \tilde t_{X_\alpha}(F^\ell, W)\right) \Big| .
\]
Thanks       to       Lemma       \ref{lem:Delta-Ya},       we       get 
$\E[B_1^2]\leq  \norm{g}_\infty^2  k(k-1)/8n(n-1)$.    Thanks  to  Lemma
\ref{lem:tt-tt}, we get using Jensen inequality that 
$\E[B_2^2]\leq \norm{g'}_\infty^2 (p-k)/4(n-k)$.
This gives the result when $g(0)=0$, except there is a $1/2$ in front of
$\norm{g}_\infty $  in the upper  bound of  the Lemma.  In  general, use
that $\Gamma_n^{F, \ell}$ and  $\hat \Gamma_n^{F, \ell}$ are probability
measures, so that $
\left(\Gamma_n^{F, \ell}- \hat \Gamma_n^{F, \ell}\right)(g)=
\left(\Gamma_n^{F, \ell}- \hat \Gamma_n^{F, \ell}\right)(\bar g)$,
with $\bar g=g-g(0)$. Then use that 
  and  $\norm{\bar g }_\infty \leq  2  \norm{g}_\infty  $  to
conclude. The case $d\geq 1$ is similar.
\end{proof}

We can now prove Theorem \ref{Theo_LLN}.
\begin{proof}[Proof of Theorem \ref{Theo_LLN}]
  We first  consider the case $d=1$.   Let $g\in\mathcal{C}^{1}([0,1])$.
  Using  Lemma  \ref{lem2}  and   Borel-Cantelli  lemma,  we  get                   that                       a.s.
  $\lim_{n\rightarrow\infty                                           }
  \left(\Gamma_{\phi(n)}^{F,\ell}(g)-\hat{\Gamma}_{\phi(n)}^{F,\ell}(g)\right)=0$,  with
  $\phi(n)=n^{4}$.
  We notice  that $\hat{\Gamma}_{n}^{F,\ell}(g)$ is a  U-statistics with
  kernel $\Phi_1(X_{[k]})$ where for $x\in  [0, 1]^k$:
\[
\Phi_1(x)=\hat t_x\,  g\left( \tilde t_x\right) +
\left(1-\hat t_x  \right) g(0),
\]
with $t_x=t_x(F^\ell, W)$ and the obvious variants for $\tilde t_x$ and
$\hat t_x$. 

Morover,  because $g$  is  uniformly  bounded on  $[0,1]$,  we get  that
$\Var\left(\Phi_1(X_{[k]})\right)<+\infty$  and we  can apply  the law  of
large    numbers    for    U-statistics     to    obtain    that    a.s.
$\lim_{n\rightarrow\infty                                             }
\hat{\Gamma}_{n}^{F,\ell}(g)=\E[\Phi(X_{[k]})]=\Gamma^{F,\ell}(g)$.
We                   deduce                  that                   a.s.
$\lim_{n\rightarrow\infty                                             }
\Gamma_{\phi(n)}^{F,\ell}(g)=\Gamma^{F,\ell}(g)$.\\

Let  $n'\geq n>k$. We have $\cs_{n, k}\subset \cs_{n', k}$ and $\cs_{n,
  p}^{\ell,\alpha}\subset \cs_{n', p}^{\ell,\alpha}$ for
$\alpha\in \cs_{n,k}$. Recall $|\cs_{n,k}^{\ell,
  \alpha}|=A_{n-k}^{p-k}$. We deduce that for $\alpha\in \cs_{n,k}$: 
\begin{align*}
\Big | \ti (F^\ell, G^\alpha_n) -\ti (F^\ell, G^\alpha_{n'}) \Big|
&\leq 
\inv{A_{n'-k}^{p-k}} \big| A_{n'-k}^{p-k} - A_{n-k}^{p-k}\big| \, 
+
 \Big| \inv{A_{n'-k}^{p-k}} -\inv{A_{n-k}^{p-k}}\Big|\,  A_{n-k}^{p-k}\\
& = 2 \left(1- \frac{A_{n-k}^{p-k}}{A_{n'-k}^{p-k}}
\right)\\
&\leq 2 \left(1- \left(\frac{n-p}{n'-p}\right)^{p-k}
\right).
\end{align*}
We deduce that:
\begin{align*}
\big| \Gamma_{n}^{F,\ell}(g)- \Gamma_{n'}^{F,\ell}(g)\big|
&\leq  \inv{A_{n'}^k} \big| A_{n'}^k - A_n^k\big| \, \norm{g}_\infty 
+  \Big| \inv{A_{n'}^k} -\inv{A_{n}^k}\Big|\,  A_n^k \, \norm{g}_\infty
\\
&\hspace{3cm}+ 
 \inv{|\cs_{n,k}|} \sum_{\alpha\in \cs_{n,k}}
\Big| g\left(\ti (F^\ell, G^\alpha_n)\right) -  g\left(\ti (F^\ell,
  G^\alpha_{n'})\right)\Big|\\
&\leq  2\norm{g}_\infty \left(1- \left(\frac{n-k}{n'-k}\right)^{k}
\right) +  2\norm{g'}_\infty \left(1- \left(\frac{n-p}{n'-p}\right)^{p-k}
\right).
\end{align*}
This implies that a.s. $\lim_{n\rightarrow\infty } \sup_{n'\in
  \{\phi(n), \ldots, \phi(n+1)\}} \big| \Gamma_{\phi(n)}^{F,\ell}(g)-
\Gamma_{n'}^{F,\ell}(g)\big|=0$. \\

With  the   first  part   of  the   proof,  we   deduce  that   for  all
$g\in \cc^1([0, 1])$, a.s. $\lim_{n\rightarrow\infty }
\Gamma_{n}^{F,\ell}(g)=\Gamma^{F,\ell}(g)$. Since there exists a
convergence determining countable subset of $\cc^1([0, 1])$, we get that
a.s. $\lim_{n\rightarrow\infty }
\Gamma_{n}^{F,\ell}=\Gamma^{F,\ell}$ for the  weak 
 convergence of the measures on $[0,1]$. 

The proof for $d\geq 1$ is
 straightforward.
\end{proof}

\section{Proof of Theorem \ref{Theo_CLT}}
\label{sec:proof_CLT}
Let $\ell\in \cm_p$ with $k=|\ell|$. 
We assume Condition \reff{eq:Hyp-hat-E} holds.

Recall the random probabilities measures $\Gamma^{F,\ell}_n$,
$\hat{\Gamma}_{n}^{F,\ell}$ and $\Gamma^{F,\ell}$ are defined in
\reff{me1}, \reff{me2} and \reff{me3}.
Let $g\in \cc^2([0, 1]^d)$. We define the U-statistic
\begin{equation}
   \label{eq:U-stat2}
U_n(g)=\inv{|\cs_{n,p}|} \sum_{\beta\in \cs_{n, p}} \Phi_2 (X_\beta),
\end{equation}
 with kernel $\Phi_2(X_{[p]})$ given by,
for $x\in [0, 1]^p$:
\begin{equation}
   \label{eq:def-Phi}
\Phi_2(x)=\hat t_{x_{\ell}} \, g \big(\tilde t_{x_\ell}\big) 
 + \big(1- \hat t_{x_\ell}\big) \, g(0)+ 
\hat t_{x_\ell} \, \langle \nabla g \big(\tilde t_{x_\ell}
\big), \tilde Z(x) - \tilde t_{x_\ell}\rangle,
\end{equation}
with $\hat t_{y}=\hat t_y (F^\ell, W)$, $\tilde t_{y}=\tilde t_y
(F^\ell, W)$ for $y\in [0, 1]^k$ and $ \tilde Z(x)$ defined in
\reff{eq:def-Zy}. Notice that:
\begin{equation}
   \label{eq:meanU}
\E[U_n(g)]=\Gamma^{F,\ell}(g).
\end{equation}
We define the random signed measure
$\Lambda_{n}^{F,\ell}=\sqrt{n}\left[\Gamma^{F,\ell}_{n}
  -\Gamma^{F,\ell}\right]$. 

\begin{lem}
   \label{lem:cv-0P}
   Let  $W\in \cw$  be a  graphon.  Let $F\in  \cf^d$ be  a sequence  of
   $d\geq     1$        simple    finite   graphs     with     $p=v(F)$,
   $\ell\in\mathcal{M}_{p}$, with $k=|\ell|$. Assume Condition \reff{eq:Hyp-hat-E}
   holds.   Let  $g\in  \cc^2([0,   1]^d)$.  Then,  we  have  that
   $\lim_{n\rightarrow\infty   }   \Lambda_{n}^{F,\ell}(g)  -   \sqrt{n}
   \left(U_n(g) - \E[U_n(g)]\right) =0$ in $L^1(\P)$.
\end{lem}

\begin{proof}
Recall \reff{eq:def-tilde-tinj}. We write:
\begin{equation}
   \label{eq:L=U+R}
\Lambda_{n}^{F,\ell}(g)
-\sqrt{n} \left(U_n(g) - \E[U_n(g)]\right) = R_1(n)+R_2(n)+ R_3(n)
\end{equation}
with
\begin{align*}
R_1(n)
&=\frac{\sqrt{n}}{|\cs_{n,k}|} \sum_{\alpha\in \cs_{n,k}}\hat Y^\alpha\,
  H_1(\alpha),    \\ 
R_2(n)
&=\frac{\sqrt{n}}{|\cs_{n,k}|} \sum_{\alpha\in \cs_{n,k}}  \big (\hat
  Y^\alpha- \hat t_{X_\alpha}\big) \, H_2(\alpha),\\
R_3(n)
&=  \frac{\sqrt{n}}{|\cs_{n,p}|} \sum_{\beta\in \cs_{n,p}} \langle  Y^\beta - Z^\beta,\,
\nabla  g(\tilde t _{X_{\beta_\ell}}) \rangle\\
&= \frac{\sqrt{n}}{|\cs_{n,k}|} \sum_{\alpha\in \cs_{n, k}} \hat
  Y^\alpha \, 
\langle \tti  (F^{\ell},G_n^{\alpha}),  \nabla  g(\tilde t
  _{X_{\beta_\ell}}) \rangle
- \frac{\sqrt{n}}{|\cs_{n,p}|} \sum_{\beta\in \cs_{n,p}} \hat
  t_{X_{\beta_\ell}} \langle \tilde Z(X_\beta),  \nabla  g(\tilde t
  _{X_{\beta_\ell}})  \rangle, 
\end{align*}
(where we used \reff{t_inj_y_beta} and \reff{eq:Zb=ttZb} for the last
equality) 
and
\begin{align*}
   H_1(\alpha)
&=
 g\big(\tti (F^{\ell},G_n^{\alpha}) \big)
- g\big(\tilde t_{X_\alpha}\big) - \langle  \tti
(F^{\ell},G_n^{\alpha}) - \tilde t_{X_\alpha},\, \nabla g 
\big(\tilde t_{X_\alpha}\big)
 \rangle, \\
H_2(\alpha)
&=
 g\big(\tilde t_{X_\alpha} \big)  -g(0) 
- \langle \tilde t_{X_\alpha}, \nabla g\big(\tilde t_{X_\alpha}\big) \rangle.
\end{align*}

According to Lemma \ref{lem_1}, we get that $\lim_{n\rightarrow \infty }
R_3(n)=0$ in $L^2(\P)$. 
According to Lemma \ref{lem:Delta-Ya}, and since $|H_2(\alpha)|\leq
2\norm{g}_\infty + \norm{\nabla g}_\infty $, we get that $\lim_{n\rightarrow \infty }
R_2(n)=0$ in $L^2(\P)$. 
Since $g\in \cc^2([0, 1]^d)$, 
by Taylor-Lagrange inequality, we have that for all $x,y\in\R$,
\[
  |g(x)-g(y) - \langle x-y, \nabla g(x) \rangle |\leq \inv{2} \norm{\nabla^2 g}_\infty  \, |x-y|^2.
\]
This gives $|H_1(\alpha)|\leq   \inv{2} \norm{\nabla^2 g}_\infty  \,|
\tti (F^{\ell},G_n^{\alpha}) - \tilde t_{X_\alpha}|^2$. 
According to Lemma \ref{lem:tt-tt}, we get that $\lim_{n\rightarrow \infty }
R_1(n)=0$ in $L^1(\P)$. 
This ends the proof. 
\end{proof}

We give a
central limit theorem for the U-statistic $U_{n}$ defined in
\reff{eq:U-stat2}.

\begin{lem}
\label{lem:TCL_U_stat}
 Under the same hypothesis as in Lemma \ref{lem:cv-0P},  we have  the  following convergence  in
distribution:
\[
\sqrt{n}\left(U_{n}(g)-\Gamma^{F,\ell}(g)\right)\cvloi{n} 
 \cn\left(0,\sigma^{F,\ell}(g)^{2}\right),\]
 with $ \sigma^{F,\ell}(g)^{2}=\Var(\cu)$ and,   $U$ being  a uniform random variable on $[0,1]$:
\begin{multline*}
\cu= \sum_{i=1}^k  \int_{[0,1]^{k}} \hat t_{R_i(x, U)}(F^\ell, W)\,  \Big(
g
\big(\tilde t_{R_i(x, U)}(F^\ell, W) \big) -g(0)\Big) \, dx \\
+ \sum_{q \in [p]\backslash \ell}\int_{[0,1]^k}   \langle \nabla g \big(\tilde
t_{x} (F^\ell, W)
\big),  t_{xU} (F^{\ell q}, W) \rangle\, dx.
\end{multline*}
\end{lem}

\begin{proof}
The random variable $U_{n}(g)$ is a U-statistic with bounded
kernel. Since $\E[U_n(g)]=\Gamma^{F, \ell}(g)$, we deduce from standard
results on U-statistics, see \cite{h:sad}, that
$\sqrt{n}\left(U_{n}(g)-\Gamma^{F,\ell}(g)\right)$ converges in
distribution towards a centered Gaussian random variable with variance
$\Var(\cu')$ and $\cu'=\sum_{q=1}^p
\E\left[\Phi_2(\tau_{1q}(X))|\, X_1\right]$, and $\Phi_2$ given by \reff{eq:def-Phi}. 
We first compute $\E\left[\Phi_2(\tau_{1q}(X))|\, X_1\right]$ for $q\in [p]$. 
We  distinguish according to $q\not \in \ell$ and $q\in\ell$.

\subsection*{The case $q\not \in 
  \left\{\ell_{1},\dots,\ell_{k}\right\}$} 
Noticing that $\tau_{1q}(X)_{\ell}$ does not depend on $X_1$, we deduce that:
\begin{align*}
\E\left[\Phi_2(\tau_{1q}(X))|\, X_1\right]
&=\E\left[\hat t_{\tau_{1q}(X)_{\ell}} \, g \big(\tilde t_{\tau_{1q}(X)_{\ell}}\big) 
 + \big(1- \hat t_{\tau_{1q}(X)_{\ell}}\big) \, g(0)|\, X_1\right]\\
&\hspace{2cm} + \E\left[
\hat t_{\tau_{1q}(X)_{\ell}}  \,\langle \nabla g \big(\tilde t_{\tau_{1q}(X)_{\ell}}
\big), \tilde Z(\tau_{1q}(X)_{[p]}) - \tilde
  t_{\tau_{1q}(X)_{\ell}}\rangle|\, X_1\right]\\
&=C+\int_{[0,1]^k} \hat t_{x}  \langle \nabla g \big(\tilde t_{x}
\big), \tilde t_{xX_1} (F^{\ell q}, W) \rangle\, dx\\
&=C+\int_{[0,1]^k}   \langle \nabla g \big(\tilde t_{x}
\big),  t_{xX_1} (F^{\ell q}, W) \rangle\, dx, 
\end{align*}
where $C$ is a constant not depending on $X_1$ (which therefore will
disappear when computing the variance of $\cu'$). 

\subsection*{The case $q\in
  \left\{\ell_{1},\dots,\ell_{k}\right\}$} 
Let $q=\ell_{i}$ for some  $i\in[k]$. 
Since  $\E\left[\tilde Z(\tau_{1q}(X)_{[p]})| \,
  \tau_{1q}(X)_{\ell}\right]= \tilde
  t_{\tau_{1q}(X)_{\ell}}$, we deduce that:
\begin{align*}
\E\left[\Phi_2(\tau_{1q}(X))|\, X_1\right]
&=\E\left[\hat t_{\tau_{1q}(X)_{\ell}} \, g \big(\tilde t_{\tau_{1q}(X)_{\ell}}\big) 
 + \big(1- \hat t_{\tau_{1q}(X)_{\ell}}\big) \, g(0)|\, X_1\right]\\
&=g(0)+ \int_{[0,1]^{k}} \hat t_{R_i(x, X_1)}\, \Big( g \big(\tilde t_{R_i(x, X_1)}\big) 
- g(0)\Big)\, dx.
\end{align*}
Thus, we obtain that $\cu'=\cu+C'$ for some constant $C'$ and:
\[
\cu
= \sum_{i=1}^k  \int_{[0,1]^{k}} \hat t_{R_i(x, X_1)}\,  \Big( g \big(\tilde t_{R_i(x, X_1)}\big) 
 - g(0)\Big) \, dx
+ \sum_{q\not \in \ell}\int_{[0,1]^k}   \langle \nabla g \big(\tilde t_{x}
\big),  t_{xX_1} (F^{\ell q}, W) \rangle\, dx.
\]
This gives the result. 
\end{proof}

The proof of  Theorem \ref{Theo_CLT} is then a direct consequence of
Lemmas \ref{lem:cv-0P}
and \ref{lem:TCL_U_stat} and \reff{eq:meanU}.

\section{Asymptotics for the empirical degrees cumulative distribution function}
\label{sec:CDF_degrees}

Let $W$ be a graphon on $[0,1]$ and $n\in\Ne$.  Recall the definition of
the  normalized  degree  function  $D$  of  the  graphon  $W$  given  in
\reff{eq:def-DW}, $D(x)=\int_{[0, 1]} W(x,y)  dy = t_x(K_2^\bullet, W)$.
From Section \ref{sec:Gn-def}, recall $G_{n}=G_{n}(W)$ is the associated
$W$-random  graphs  with  $n$  vertices constructed  from  $W$  and  the
sequence $X=(X_{i}: i\in\Ne)$ of independent uniform random variables on
$[0,1]$.  Recall the (normalized) degree sequence of a graph defined in
\reff{eq:def_degree_seq}, and set
\[
D_{i}^{(n)}=D_i(G_n)=\ti(K_{2}^{\bullet},G_{n}^{i})
\]
the normalized degree of the vertex $i\in [n]$ in $G_n$.  By construction of
$G_n$,  we  get that  conditionally  on  $X_i$  , $(n-1)D_{i}^{(n)}$  is  for
$n\geq i$ a binomial random variable with parameters $(n-1, D(X_i))$.
We    define   the    empirical    cumulative   distribution    function
$\cdf_{n}=(\cdf_{n}(y):  y\in  [0, 1])$  of  the  degrees of  the  graph
$G_{n}$ by, for $y\in  [0, 1]$:
\begin{equation}\label{eq:def_cdf}
 \cdf_{n}(y)=\frac{1}{n}\sum_{i=1}^{n}\1{D_{i}^{(n)}\le D(y)}. 
\end{equation}

\begin{rem}
   \label{rem:casD}
If  we   take  $g=\ind_{[0,D(y)]}$  with  $y\in[0,1]$   and  $F=K_2$  in
\reff{eq:intro_def_gamma} and using the expression of $\Gamma^{F,\ell}$
given    in   Remark    \ref{rem:CP_Gamma},   (ii),    we   have    that
$              \cdf_{n}(y)=\Gamma_{n}^{K_2,\bullet}(g)$              and
$\Gamma^{K_2,\bullet}(g)=y$.     If    $D$     is    increasing,    then
$\Gamma^{K_2,\bullet}$, which is the distribution of $D(U)$, with $U$
uniform on  $[0, 1]$, has  no atoms and Theorem  \ref{Theo_LLN} implies
that   a.s.   $\lim_{n\rightarrow\infty   }  \cdf_{n}(y)=y$   for   all
$y\in [0, 1]$.   Using Dini's theorem, we get that  if $D$ is increasing
 on  $[0,1]$, then the function  $\cdf_n$ converges almost
surely towards $\Id$, the  identity map on $[0,1]$,  with respect  to the uniform
norm.
\end{rem}

To get the corresponding
fluctuations, we shall consider the following conditions:
\begin{equation}\label{eq:condi_W}
 W\in\mathcal{C}^{3}([0,1]^2), \,  D'>0,\, W\leq 1-\varepsilon_0
 \text{ and } D\geq  \varepsilon_0 
 \text{ for some } \varepsilon_{0}\in\left(0,1/2\right) .
\end{equation}
If \reff{eq:condi_W} holds, then we have $D\in \cc^1([0, 1])$ and $D([0,
1])\subset [\varepsilon_0, 1-\varepsilon_0]$. Notice that even if
\reff{eq:condi_W} holds, the set $\{W=0\}$ might have positive Lebesgue
measure; but the regularity conditions on $W$ rules out bipartite
graphons (but not tripartite graphons). 

\begin{theo}\label{thm:CLT_indicator_degree_sequence}
Assume that $W$ satisfies condition \reff{eq:condi_W}. Then 
 we have the following convergence of finite-dimensional distributions:
 \begin{equation*}
  \left(\sqrt{n}\left(\cdf_{n}(y)-y\right): y\in(0,1)\right) 
  \,\xrightarrow[n\rightarrow+\infty]{(fdd)}\, \chi,
 \end{equation*}
where $\chi=(\chi_{y}:y\in(0,1))$ is a centered Gaussian process defined, for all $y\in(0,1)$ by:
\begin{equation}\label{eq:def_gaussian_proc_ind}
 \chi_{y}=\int_{0}^{1}(\rho(y,u)-\bar \rho(y))dB_{u},
\end{equation}
with $B=(B_u, u\geq 0)$ a standard Brownian motion, and  $(\rho(y, u):u\in
[0, 1]) $  and
$\bar \rho(y)$ defined for $y\in (0, 1)$ by:
\begin{equation*}
 \rho(y,u)=\ind_{[0,y]}(u)-\frac{W(y,u)}{D'(y)} \quad \text{ and } \quad
 \bar \rho(y)=\int_{0}^{1}\rho(y,u)du.
\end{equation*}
\end{theo}

\begin{rem}
   \label{rem:Sigma}
The covariance kernel  of the
   Gaussian process $\chi$ can also be written as
   $\Sigma=\Sigma_1+\Sigma_2+\Sigma_3$, where for $y,z\in(0,1)$:
\begin{align}
\label{eq:def-s1}
  \Sigma_{1}(y,z)&=y\wedge z - yz,\\
\label{eq:def-s2}
  \Sigma_{2}(y,z)&=\frac{1}{D'(y)D'(z)}\left(\int_{0}^{1}W(y,x)W(z,x)dx
                   - D(y)D(z)\right),\\
\label{eq:def-s3}
  \Sigma_{3}(y,z)&
=\frac{1}{D'(y)}\left(D(y)z-\int_{0}^{z}W(y,x)dx\right)
        + \frac{1}{D'(z)}\left(D(z)y-\int_{0}^{y}W(z,x)dx\right).
\end{align}
Thus, for $y\in(0,1)$ the  variance of $\chi(y)$ is:
 \begin{equation*}
 \Sigma(y,y)=y(1-y)+\frac{1}{D'(y)^2}\left(\int_{0}^{1}W(y,x)^2dx-D(y)^2\right)
                  +\frac{2}{D'(y)}\left(D(y)y-\int_{0}^{y}W(y,x)dx\right).
  \end{equation*}
\end{rem}

\begin{rem}
   \label{rem:conj-fct-cv}
We conjecture that the convergence of Theorem \ref{thm:CLT_indicator_degree_sequence}
holds for the process in the Skorokhod space. However, the techniques used to prove
this theorem are not strong enough to get such result.
\end{rem}

\section{Preliminary results for the empirical cdf of the degrees}
\label{sec:preliminaries_CDF_degrees}

\subsection{Estimates for the first moment of the empirical cdf}
Recall $X=(X_n:n\in\Ne)$ is a sequence of independent random variables
uniformly distributed on $[0,1]$ used to construct the sequence of $W$-random graphs 
$(G_n:n\in\Ne)$. Recall $\cdf_{n}(y)$ is given in \eqref{eq:def_cdf}.

For all $y\in(0,1)$, we set $c_{n}(y)=\E\left[\cdf_{n+1}(y)\right]$ that is
\begin{equation}\label{eq:c_{n}(y)}
c_{n}(y)=\pro\left(D_{1}^{(n+1)}\le D(y)\right),
\end{equation}
where $D_1^{(n+1)}$ is a binomial random variable with parameter $(n,D(X_1))$.  
We set:
\begin{equation}\label{eq:def_sigma}
 \sigma_{(x)}^2=x(1-x) \quad \text{for $ x\in[0,1]$},
\end{equation}
and with $\lceil x \rceil$ the unique integer such that $\lceil x \rceil-1<x\le \lceil x \rceil$,
\begin{equation}\label{eq:def_S}
 S(x)=\lceil x \rceil -x -\frac{1}{2} \quad \text{for $ x\in\R$}.
\end{equation}
The next
proposition gives precise asymptotics of $c_n$. 

\begin{prop}\label{prop:c_{n}(y)}
  Assume  that  $W$  satisfies  condition  \reff{eq:condi_W}.   For  all
  $y\in(0,1)$,  there  exists  a  constant   $C>0$  such  that  for  all
  $n\in\Ne$, we have with $d=D(y)$,
\begin{equation*}
  n\left(c_{n}(y)-y\right)
  =-\frac{D''(y)}{D'(y)^3}\frac{\sigma_{(d)}^2}{2} +\inv{ D'(y)}
  \left(\frac{1-2d}{2} + S(nd)  \right)+R_{n}^{\ref{prop:c_{n}(y)}},
\end{equation*}
with
\begin{equation*}
 \left|R_{n}^{\ref{prop:c_{n}(y)}}\right|\le C\,n^{-\frac{1}{4}}.
\end{equation*}
\end{prop}
In  particular,  because  $\left|S(x)\right|\le  \frac{1}{2}$,  for  all
$x\in\R$, we have that for all $y\in(0,1)$:
\begin{equation}\label{eq:lim_c_n}
 c_{n}(y)-y=O\left(n^{-1} \right).
\end{equation}

\begin{proof}
Let $y\in(0,1)$. Recall the definition of $\ch$ in \reff{eq_FDR_bin}.  We have:
 \begin{equation*}
  c_{n}(y)-y=\int_{0}^{1}\left(\mathcal{H}_{n,d,0}(D(x))-\1{x\le y}\right)dx.
 \end{equation*}
By Proposition \ref{prop:ind8} applied with $G(x)=1$ and $\delta=0$, we obtain that:
\begin{equation*}
 n\int_{0}^{1}\left(\mathcal{H}_{n,d,0}\left(D(x)\right)-\1{x\le
     y}\right)dx
  =-\frac{D''(y)}{D'(y)^3}\frac{\sigma_{(d)}^{2}}{2}
+\inv{D'(y)} \left(\frac{1-2d}{2}+ S(nd)
\right)+R_{n}^{\ref{prop:c_{n}(y)}},
\end{equation*}
with $R_{n}^{\ref{prop:c_{n}(y)}}=R_{n}^{\ref{prop:ind8}}(1)$ and
$\left|R_{n}^{\ref{prop:c_{n}(y)}}\right| \le C n^{-\frac{1}{4}}$. 
\end{proof}

For $y\in(0,1)$ and $u\in[0,1]$, we set, with $d=D(y)$,
\begin{equation}\label{eq:H_n}
 H_{n}(y,u)=n\left(\E\left[\left.\1{D_{1}^{(n+1)}\le d}\right|X_{2}=u\right]-c_{n}(y)\right)
\end{equation}
and
\begin{equation}\label{eq:H_n_star}
 H_{n}^{\star}(y,u)=\E\left[\left.\1{D_{1}^{(n+1)}\le d}\right|X_{1}=u\right]-c_{n}(y).
\end{equation}

\begin{prop}\label{prop:mean_phi_n}
Assume that $W$ satisfies condition \reff{eq:condi_W}.
 For all $y\in(0,1)$, there exists a positive constant $C$ such that 
 for all $n\geq 2$ and $u\in[0,1]$, we have with $d=D(y)$:
 \begin{equation*}
 H_{n}(y,u)
 =\frac{1}{D'(y)}\left(d-W(y,u)\right)+R_{n}^{\ref{prop:mean_phi_n}}(u),
 \end{equation*}
 with
 \begin{equation}\label{eq:H_n_rest}
 \left|R_{n}^{\ref{prop:mean_phi_n}}(u)\right|\le C\,n^{-\frac{1}{4}}.
\end{equation}
For all  $y\in(0,1)$ and $u\in[0,1]$ such that $u\neq y$, we have
\begin{equation}
\label{eq:H_n_star_limit}
 \left|H_{n}^{\star}(y,u)\right|\le 1 
\quad \text{for all $n\geq 2$,   and }\quad
 \lim_{n\to\infty} H_{n}^{\star}(y,u)=\1{u\le y}-y.
\end{equation}
\end{prop}

\begin{proof}
  In  what follows,  $C$ denotes  a positive  constant which  depends on
  $\varepsilon_0$, $W$  and  $y\in(0,1)$, and it may vary from line to line. 
Recall that $X_{[2]}=(X_1,X_2)$. 
We define the function  $\varphi_{n}$ by:
\begin{equation}\label{eq:phi_n}
 \varphi_{n}(x, u)=\pro\left(D_{1}^{(n+1)}\le
     d\,\big|\,X_{[2]}=(x,u)\right)-\1{x\le y}\quad \text{for $x,u\in
     [0, 1 ]$}. 
\end{equation}
Then we have for $u\in [0, 1]$:
\begin{equation}\label{eq:relation_H_n_varphi}
 H_{n}(y,u)=n\E\left[\varphi_{n}(X_1, u)\right]-n(c_{n}(y)-y).
\end{equation}

Conditionally on $\{X_{[2]}=(x,u)\}$, $D_1^{(n+1)}$ is distributed as
$Y_{12} + \tilde B^{(n)}$, where $Y_{12}$ and $\tilde B$ are independent,
$Y_{12}$ is Bernoulli $W(x, u)$ and $\tilde B$ is binomial with
parameter $(n-1, D(x))$. Thus, we have:
\begin{align}
 \varphi_{n}(x,u)
 &=\P\left(Y_{12}+\tilde{B}\le nd \right)-\1{x\le y}\nonumber\\
 &=W(x,u)\left[\P\left(\tilde{B}\le nd-1 \right)-\1{x\le
   y} \right]+(1-W(x,u))\left[\P\left(\tilde{B}\le nd
   \right)-\1{x\le y} \right]\nonumber\\ 
 &=W(x,u)\left[\mathcal{H}_{n-1,d,d-1}(D(x))-\1{x\le y}\right]
+   (1-W(x,u))\left[\mathcal{H}_{n-1,d,d}(D(x))-\1{x\le
   y}\right]. \label{eq:ind1} 
\end{align}
Let $W_1(x,u)$ denote $\partial W(x,u)/\partial x$. We apply Proposition
\ref{prop:ind8} with $G(x)=W(x,u)$, $\delta=d-1$ and $n$ replaced by
$n-1$ to get that: 
\begin{multline}
\label{eq:ind2}
(n-1)\E\left[W(X_{1},u)\left[\mathcal{H}_{n-1,d,d-1}(D(X_{1}))-\1{X_{1}\le
      y}\right]\right] \\
=\frac{\sigma_{(d)}^2}{2D'(y)^2}\left[W_1(y,u) 
 - \frac{W(y,u)D''(y)}{D'(y)}\right]\\
 +\frac{W(y,u)}{D'(y)}\left(-\frac{1}{2}+S(nd-1)\right)
 +R_{n-1}^{\ref{prop:ind8}}(W(.,u)), 
\end{multline}
and with $G(x)=1-W(x,u)$, $\delta=d$ and $n$ replaced by $n-1$, to get that:
\begin{multline}
\label{eq:ind3}
(n-1)\E\left[(1-W(X_{1},u))\left[\mathcal{H}_{n-1,d,d}(D(X_{1}))-\1{X_{1}\le
      y}\right]\right]\\ 
=\frac{\sigma_{(d)}^2}{2D'(y)^2}\left[-W_1(y,u) - \frac{(1-W(y,u))D''(y)}{D'(y)}\right]\\
 +\frac{1-W(y,u)}{D'(y)}\left(\frac{1}{2}+S(nd)\right) 
+R_{n-1}^{\ref{prop:ind8}}(1-W(.,u)).
\end{multline}
By equations \reff{eq:ind1}, \reff{eq:ind2} and \reff{eq:ind3} and since
$S(nd-1)=S(nd)$, we get that: 
\begin{equation*}
 (n-1)\E\left[\varphi_{n}(X_1, u)\right]
 =-\frac{\sigma_{(d)}^2}{2}\frac{D''(y)}{D'(y)^3}+\frac{1}{D'(y)}\left(\inv{2}-
   W(y,u)+S(nd)\right)+R_{n}^{(1)}(u),
\end{equation*}
where
$R_{n}^{(1)}(u)=R_{n-1}^{\ref{prop:ind8}}
(W(.,u))+R_{n-1}^{\ref{prop:ind8}}(1-W(.,u))$.
Because  $W$  satifies  condition   \reff{eq:condi_W},  we  deduce  from
\reff{eq:prop:ind8}  that   $\left|R_{n}^{(1)}(u)\right|\le  Cn^{-1/4}$
for some finite constant $C$ which does not depend on $n$ and $u\in [0, 1]$.
Using Proposition \ref{prop:c_{n}(y)}, we get that
\begin{equation}\label{eq:varphi_egality}
 (n-1)\E\left[\varphi_{n}(X_1, u)\right]
 -(n-1)(c_{n}(y)-y)=\frac{d-W(y,u)}{D'(y)}+R_{n}^{(2)}(u), 
\end{equation}
where $R_{n}^{(2)}(u)=R_{n}^{(1)}(u)+R_{n}^{\ref{prop:c_{n}(y)}}+(c_n(y)-y)$ and $\left|R_{n}^{(2)}(u)\right|\le Cn^{-1/4}$ because of 
\reff{eq:lim_c_n}.
By equations \reff{eq:relation_H_n_varphi} and \reff{eq:varphi_egality},
we deduce that: 
\[
  H_{n}(y,u)
 =\frac{n}{n-1} \frac{d-W(y,u)}{D'(y)}+\frac{n}{n-1}
 R_{n}^{(2)}(u)
=
\frac{d-W(y,u)}{D'(y)}+R_{n}^{\ref{prop:mean_phi_n}}(u),
\]
with  $\left|R_{n}^{\ref{prop:mean_phi_n}}(y,u)\right|\le
Cn^{-\frac{1}{4}}$. This gives \reff{eq:H_n_rest}.\\

For the second assertion \reff{eq:H_n_star_limit}, we notice that
\begin{equation*}
  H_{n}^{\star}(y,u)=\mathcal{H}_{n,d,0}(D(u))-c_{n}(y),
\end{equation*}
with $\mathcal{H}_{n,d,0}(D(u))\in[0,1]$ and $c_{n}(y)\in[0,1]$.
By the strong law of large numbers, we have for $u\neq y$:
\begin{equation*}
 \lim_{n\rightarrow\infty } \mathcal{H}_{n,d,0}(D(u)) =\ind_{\{u\le y\}}.
\end{equation*}
Using \reff{eq:lim_c_n}, we get the expected result.
\end{proof}

\subsection{Estimates for the second moment of the empirical cdf}

For $\ry=(y_1, y_2)\in [0, 1]^2$, let  $M(\ry)$
 be the  covariance matrix of a couple $(Y_1, Y_2)$ of Bernoulli random
 variables such that 
 $\P(Y_i=1)=D(y_i)$ for $i\in \{1,2\}$ and $\P(Y_1=Y_2=1)=\int_{[0, 1]} W(y_1,z)
 W(y_2, z)\, dz$.

Let  $\convset$ be the set of all convex sets in $\R^2$.
For $\conv\in\convset$, we define its sum with a vector $\rx$ in $\R^2$ as 
\[
 \conv+\rx=\{k+\rx: k\in \conv\}
\]
and its product with a real matrix $M$ of size $2\times 2$ as
\[
M\conv=\{Mk : k\in \conv\}. 
\]

Recall  that for $\rx\in\R^2$, $|\rx|$ is the Euclidian norm of $\rx$ in
$\R^2$. 
Recall          $X_{[2]}=(X_1,         X_2)$.          We         define
$\hat   D^{(n+1)}=(\hat    D_1^{(n+1)},   \hat   D_2^{(n+1)})$,    where   for
$i\in \{1, 2\}$, $\hat D_i^{(n+1)}$ is the  number of edges from the
vertex $i$
to the  vertices $\{k, 3\leq k\leq  n+1\}$ of $G_{n+1}$; it  is equal to
$n D_i^{(n+1)}$ if the edge $\{1, 2\}$ does not belong to $G_{n+1}$ and to
$n D_i^{(n+1)} -1$ otherwise. The proof of the next proposition is
postponed to section \ref{sec:proof-prop-stein}. 

\begin{prop}
   \label{prop:stein}
Assume that $W$ satisfies condition \reff{eq:condi_W}. There exists a
finite constant $C_0$ such that for all $\rx=(x_1, x_2)\in [0, 1]^2$ with
$x_1\neq x_2$, we get for all $n\geq 2$: 
\[
 \sup_{\conv \in\convset}\left|\P\left(\hat D^{n+1}
     \in \conv |\, X_{[2]}=\rx \right)
 - \P\left(Z\in
   \frac{M(\rx)^{-\frac{1}{2}}}{\sqrt{n-1}}(\conv-\mu(\rx))
 \right)\right|\le \frac{C_0}{\sqrt{n}}, 
 \]
 where $\mu(\rx)=(n-1)(D(x_1),D(x_{2}))$ 
and  $Z$ is a standard $2$-dimensional Gaussian vector.
\end{prop}

For $y_{1},y_{2}\in(0,1)$, with $d_1=D(y_1)$ and $d_{2}=D(y_2)$, we set
with $\delta\in \R$ and $\rx=(x_1, x_2)\in [0, 1]^2$ such that $x_1\neq x_2$:
\[
 \Psi_{n,\delta}\left(\rx\right)
 =\E\left[\prod_{i\in \{1, 2\}} \left(\ind_{\{\hat D_i^{(n+1)} \leq
       nd_i+\delta\}}-\ind_{\{X_i\le y_i\}}\right) 
        \Big|\,X_{[2]}=\rx\right]. 
\]
Recall $\Sigma_2$ defined in \reff{eq:def-s2} and that $X_{[2]}=(X_1,
X_2)$. 
\begin{lem}\label{lem:ind6} 
Assume that $W$ satisfies condition \reff{eq:condi_W}. For all
$\ry=(y_{1},y_{2})\in(0,1)^2$, $\delta\in [-1,0]$ and $G\in\mathcal{C}^{1}([0,1]^2)$,
we have:
 \begin{equation*}
  \lim_{n\to \infty}n\E\left[G\left(X_{[2]}\right)\Psi_{n,\delta}\left(X_{[2]}\right)\right]
  =G(\ry)\Sigma_{2}(\ry).
 \end{equation*}
\end{lem}

\begin{proof}
Let $A=4\sqrt{\log(n-1)}$. For $n\geq 2$ and $\delta\in[-1,0]$,  
we set:
\[
 \Psi_{n,\delta}^{(1)}(\rx)
 =\Psi_{n,\delta}(\rx)\prod_{i=1}^{2}\ind_{\{\sqrt{n-1}|D(x_i)-d_i|\leq
   A\}}
\quad\text{and}\quad \Psi_{n,\delta}^{(2)}(\rx)=\Psi_{n,\delta}(\rx)-\Psi_{n,\delta}^{(1)}(\rx).
\]
Then we have
\begin{equation}
\label{eq:mean_Psi_kappa}
 \E\left[G\left(X_{[2]}\right)\Psi_{n,\delta}\left(X_{[2]}\right)\right]
 =\sum_{i\in \{1, 2\}}\E\left[G\left(X_{[2]}\right)\Psi_{n,\delta}^{(i)}\left(X_{[2]}\right)\right].
\end{equation}

\subsection*{Study of $\E\left[G\left(X_{[2]}\right)
    \Psi_{n,\delta}^{(2)}\left(X_{[2]}\right)\right]$}
Recall   that  for   $i\in  \{1,   2\}$,  conditionally   on  $X_i=x_i$,
$\hat D_i^{(n+1)} $  is distributed as a Bernoulli  random variable with
parameter $(n-1, d_i)$. We get that:
\begin{align*}
  \left|\Psi_{n,\delta}^{(2)}(\rx) \right|
&\leq  2 \sum_{i\in \{1,2 \}}
\E\left[\left|\ind_{\{\hat D_i^{(n+1)} \leq
       nd_i+\delta\}}-\ind_{\{X_i\le y_i\}}\right|\ind_{\{\sqrt{n-1}\, |D(x_i)- d_i|\geq A\}}
        \Big|\,X_{[2]}=\rx\right]\\
&  =  2 \sum_{i\in \{1,2 \}}
\left|\mathcal{H}_{n-1,d_i,\delta+d_i}(x_i)-\ind_{\{x_i\le
  y_{i}\}}\right|\ind_{\{\sqrt{n-1}\,\left|D(x_i)-d_i\right|\ge A\}} . 
\end{align*}
By Lemma  \ref{lem:ind40} (with $n$ replaced by $n-1$), we deduce that:
\begin{equation}
\label{eq:mean_G_Psi_2}
 \lim_{n\to\infty} n\E\left[G\left(X_{[2]}\right)
\Psi_{n,\delta}^{(2)}\left(X_{[2]}\right)\right]=0. 
\end{equation}

\subsection*{Study of $\E\left[G\left(X_{[2]}\right)
    \Psi_{n,\delta}^{(1)}\left(X_{[2]}\right)\right]$}
This part is more delicate. 
For $\rz=(z_1,z_2)\in[0,1]^2$, set
$H(\rz)=\frac{G(\rz)}{D'(z_1)D'(z_2)}$  and 
 $\mathfrak{t}_{n}(\rz)=(t_{n}(z_1),t_{n}(z_2))$:
\begin{equation*}
 t_{n}(z_i)=D^{-1}\left(d_i+\frac{z_i}{\sqrt{n-1}}\right) \quad\text{for $i\in \{1, 2\}$}. 
\end{equation*}
Using  the change  of  variable  $z_{i}=\sqrt{n-1}(D(x_i)-d_i)$ for  
$i\in\{1,2\}$ with $\rx=(x_1, x_2)$, we get:
\begin{align}
\nonumber
(n-1) \E\left[G\left(X_{[2]}\right)\Psi_{n,\delta}^{(1)}\left(X_{[2]}\right)\right]
&=(n-1) \int_{[0,1]^2}G(\rx)\Psi_{n,\delta}(\rx)\prod_{i\in \{1, 2\}}
  \ind_{\{\sqrt{n-1}\left|D(x_i)-d_i\right|\le A\}}d\rx\\
\label{eq:ind7}
 &=\int_{[-A,A]^2}H\left(\mathfrak{t}_{n}(\rz)\right)
                 \Psi_{n,\delta}\left(\mathfrak{t}_{n}(\rz)\right)d\rz.
\end{align}
Notice that:
\[
\Psi_{n,\delta}\left(\mathfrak{t}_{n}(\rz)\right)
=\E\left[\prod_{i\in \{1, 2\}} \left(\ind_{\{\hat D_i^{(n+1)} \leq
       nd_i+\delta\}}-\ind_{\{z_i\le 0\}}\right) 
        \Big|\,X_{[2]}=\mathfrak{t}_{n}(\rz)\right]. 
\]

Set $\tilde \delta=(\delta, \delta)$ and $\hat{D}^{(n+1)}= (\hat
D_1^{(n+1)}, \hat D_2^{(n+1)})$. We define the  sets for
$\rz$ and  $\hat D^{(n+1)}$:
\begin{align*}
 I^{(1)}=[0, A]^2 \quad&\text{and}\quad
 \tilde{C}_{n}^{(1)}=\tilde \delta+n\, (-\infty ,
d_1]\times (-\infty ,d_2],\\
 I^{(2)}=[0, A]\times  [-A, 0)  \quad&\text{and}\quad
 \tilde{C}_{n}^{(2)}=\tilde \delta+n\, (-\infty ,
d_1]\times (d_2,+\infty),\\
 I^{(3)}=[-A, 0)\times  [0,A]  \quad&\text{and}\quad
 \tilde{C}_{n}^{(3)}=\tilde \delta+n\, (
d_1,+\infty )\times (-\infty ,d_2],\\
 I^{(4)}=[-A,0)^2  \quad&\text{and}\quad
 \tilde{C}_{n}^{(4)}=\tilde \delta+n\, (
d_1,+\infty )\times (d_2,+\infty ).
\end{align*}
For $1\leq i\leq 4$, we set:
\begin{equation*}
 Q^{(i)}_{n}(\rz)
 =\P\left(\hat D^{(n+1)} \in
   \tilde{C}_{n}^{(i)}\big|X_{[2]}=\mathfrak{t}_{n}(\rz)\right)
\quad\text{and}\quad \Delta_n^{(i)} = \int_{I^{(i)}}
H\left(\mathfrak{t}_{n}(\rz)\right) Q^{(i)}_{n}(\rz)\,  d\rz. 
\end{equation*}
By construction, we have:
\begin{equation}
\label{eq:sum-Di}
  (n-1)\E\left[G\left(X_{[2]}\right)\Psi_{n,\delta}^{(1)}\left(X_{[2]}\right)\right]
  =\sum_{i=1}^4\Delta_n^{(i)}. 
\end{equation}

We now study $\Delta_n^{(1)}$. 
By Proposition \ref{prop:stein}, we get that
\[
\Delta_n^{(1)}
 =\int_{[0,A]^2}H\left(\mathfrak{t}_{n}(\rz)\right)\, 
  \P\left(Z\in
    \frac{M\left(\mathfrak{t}_{n}(\rz)\right)^{-1/2}}{\sqrt{n-1}}\,
    \left(\tilde C_{n}^{(1)}- \mu(\mathfrak{t}_{n}(\rz))\right)\right)d\rz+R_{n}^{(1)}\\
\]
where 
$\mu(\rx)=(n-1) (D(x_1), D(x_2))$
and $|R_{n}^{(1)}|\leq  \norm{H}_\infty  8 C_0 \sqrt{\log(n)/n}$ so that 
 $\lim_{n\rightarrow\infty } R_{n}^{(1)}=0$. 
Set $\tilde d=(d_1,d_2)$.
Since $\mathfrak{t}_{n}(\rz))$ converges towards $\ry$, we get:
\[
\lim_{n\rightarrow\infty } H(\mathfrak{t}_{n}(\rz))= H(\ry)
\quad\text{and}\quad
\lim_{n\rightarrow\infty } M\left(\mathfrak{t}_{n}(\rz)\right)= M(\ry)
\]
and, with $J(\rz)=(-\infty  , -z_1]\times  (-\infty ,  -z_2]$, 
\begin{equation}
   \label{eq:=cnz}
\inv{\sqrt{n-1}}\,
    \left(\tilde C_{n}^{(1)}- \mu(\mathfrak{t}_{n}(\rz))\right)
= (n-1)^{-1/2} (\tilde \delta+\tilde d) + J(\rz).
\end{equation}

Since $\lim_{n\rightarrow\infty }
M\left(\mathfrak{t}_{n}(\rz)\right)=M(\ry)$ and   $M(\ry)$ is positive definite, we  deduce
that $d\rx$-a.e.:
\[
\lim_{n\rightarrow\infty }
\ind_{\frac{M\left(\mathfrak{t}_{n}(\rz)\right)^{-1/2}}{\sqrt{n-1}}\,
  \left(\tilde  C_{n}^{(1)}-  \mu(\mathfrak{t}_{n}(\rz))\right)}  (\rx)
= 
\ind_{M(\ry)^{-1/2} J(\rz)}(\rx)
\]
and thus (by dominated convergence):
\[
\lim_{n\rightarrow\infty }
\P\left(Z\in {\frac{M\left(\mathfrak{t}_{n}(\rz)\right)^{-1/2}}{\sqrt{n-1}}\,
  \left(\tilde  C_{n}^{(1)}-  \mu(\mathfrak{t}_{n}(\rz))\right)} \right)
= \P\left(M(\ry)^{1/2}Z\in  J(\rz)\right).
\]

For $\rz\in [2, +\infty )^2$ and $n\geq 2$, we have:
\begin{align*}
  \left\{Z\in {\frac{M\left(\mathfrak{t}_{n}(\rz)\right)^{-1/2}}{\sqrt{n-1}}\,
  \left(\tilde  C_{n}^{(1)}-  \mu(\mathfrak{t}_{n}(\rz))\right)} 
\right\}
& \subset
  \left\{2  M\left(\mathfrak{t}_{n}(\rz)\right)^{1/2} Z\in 
J(\rz)
\right\}\\
& \subset
  \left\{2^{3/2} |Z|\geq  |\rz|\right\},
\end{align*}
where we used \reff{eq:=cnz} and that for all $i\in\{1,2\}$, $|(n-1)^{-1/2} (\delta+d_i)|\leq z_i/2$
 for the first inclusion and for the second that  $|M\rx|\leq
 \sqrt{2}\, \norm{M}_\infty \, |\rx|$,  $\norm{M ^{1/2}}_\infty \leq
\sqrt{2}\norm{M}_\infty ^{1/2}$ and $ \norm
  {M\left(\rx\right)}_\infty \leq 1/2$ for all $\rx\in [0, 1]^2$ so that
  $|M\left(\mathfrak{t}_{n}(\rz)\right)^{1/2} Z|\leq \sqrt{2}\,  |Z|$.  
 Since $\int_{\R} \P(2^{3/2} |Z|\geq |\rz|)\, d\rz$ is finite  and $H$ is bounded, we deduce
 from dominated convergence that:
\begin{multline*}
\lim_{n\rightarrow\infty } \int_{[0,A]^2}H\left(\mathfrak{t}_{n}(\rz)\right)\, 
  \P\left(Z\in
    \frac{M\left(\mathfrak{t}_{n}(\rz)\right)^{-1/2}}{\sqrt{n-1}}\,
    \left(\tilde C_{n}^{(1)}-
      \mu(\mathfrak{t}_{n}(\rz))\right)\right)d\rz \\
= H(\ry)\int_{[0,+\infty )^2} \P\left(M(\ry)^{1/2}Z\in  J(\rz)\right)\, d\rz.
\end{multline*}
Recall $x^+=\max(0, x)$ and $x^-=\max(0, -x)$ denote the positive and
negative part of $x\in \R$.  With $\tilde Z=M(\ry)^{1/2}Z=(\tilde Z_1, \tilde Z_2)$, we get:
\[
\int_{[0,+\infty )^2} \P\left(M(\ry)^{1/2}Z\in  J(\rz)\right)\, d\rz
= \E\left[\tilde Z_1^- \tilde Z_2^-\right]
\]
and thus
\[
\lim_{n\rightarrow\infty } \Delta_n^{(1)}= H(\ry) \E\left[\tilde Z_1^-
  \tilde Z_2^-\right].
\]

Similarly, we obtain:
\begin{align*}
   \lim_{n\rightarrow\infty } \Delta^{(2)}_n 
&=-H(\ry) \E\left[\tilde Z_1^-  \tilde Z_2^+\right]\\
   \lim_{n\rightarrow\infty } \Delta^{(3)}_n 
&=  -H(\ry) \E\left[\tilde Z_1^+  \tilde Z_2^-\right]\\
   \lim_{n\rightarrow\infty } \Delta^{(4)}_n 
&= H(\ry) \E\left[\tilde Z_1^+  \tilde Z_2^+\right].
\end{align*}
Using the definition of $\Sigma_2(\ry)$ and $M(\ry)$, notice that
$D'(y_1)D'(y_2)\Sigma_2(\ry)$ is the covariance of $\tilde Z_1$ and $\tilde
Z_2$. Thus, we obtain:
\begin{equation}
   \label{eq:sumDi}
 \lim_{n\rightarrow \infty}\sum_{i=1}^4 \Delta_n^{(i)}
 =H(\ry)\E\left[\left(\tilde Z_{1}^{+}-\tilde
   Z_{1}^{-}\right)\left(\tilde Z_{2}^{+}-\tilde Z_{2}^{-}\right)\right] 
=H(\ry) \E\left[\tilde Z_{1}\tilde Z_{2}\right]
=G(\ry)\Sigma_2(\ry).
\end{equation}

\subsection*{Conclusion}
Use
\reff{eq:mean_Psi_kappa}, \reff{eq:mean_G_Psi_2}, 
 \reff{eq:sum-Di} and  \reff{eq:sumDi} to get the result. 
\end{proof}

The next proposition, main result of this section, is a consequence of
Lemma \ref{lem:ind6}. 
\begin{prop}\label{prop:ind10}
 Assume that $W$ satisfies condition \reff{eq:condi_W}. For all $y_1,y_2\in(0,1)$, we have with $d_1=D(y_1)$ and $d_2=D(y_2)$:
 \begin{equation*}
  \lim_{n\to\infty}n\,\E\left[\left(\1{D_{1}^{(n+1)}\le d_1}
      -\1{X_{1}\le y_1}\right)\left(\1{D_{2}^{(n+1)}\le d_2}-\1{X_{2}\le
        y_2}\right)\right] 
  =\Sigma_{2}(y_1,y_2).
 \end{equation*}
\end{prop}

\begin{proof}
Using the comment before Proposition \ref{prop:stein}, we get:
 \begin{multline*}
\E\Bigg[\prod_{i\in \{1, 2\}} \left(\ind_{\{D_{i}^{(n+1)}\le
    d_i}-\1{X_{i}\le y_i}\right)\Bigg]\\
=\E\left[W(X_{[2]})\Psi_{n,-1}(X_{[2]})\right]
   + \E\left[(1-W(X_{[2]}))\Psi_{n,0}(X_{[2]})\right]. 
 \end{multline*}
We apply Lemma  \ref{lem:ind6} twice with $G=W$ and $G=1-W$ to get the result.
\end{proof}

\section{Proof of Theorem \ref{thm:CLT_indicator_degree_sequence}}
\label{sec:proof_CLT_CDF_degrees}

Recall the definitions of $\cdf_{n+1}$ and $c_{n}(y)$ given in \reff{eq:def_cdf} and \reff{eq:c_{n}(y)}.
We define the normalized and  centered random process $\hat{\cdf}_{n+1}=
(\hat{\cdf}_{n+1}(y):\, y\in(0,1)) $ by:
\begin{equation}\label{eq:def_hat_cdf}
 \hat{\cdf}_{n+1}(y)=\sqrt{n+1}\left[\cdf_{n+1}(y)-c_{n}(y)\right].
\end{equation}
Let $U_{n+1}=(U_{n+1}(y):\, y\in(0,1)) $ be the H\`ajek projection  of $\hat{\cdf}_{n+1}$:
\begin{equation}\label{eq:haj_proj}
 U_{n+1}(y)=\sum_{i=1}^{n+1}\E\left[\left.\hat{\cdf}_{n+1}(y)\right|X_{i}\right].
\end{equation}

Recall $\Sigma$ defined in Remark \ref{rem:Sigma}.
\begin{lem}\label{lem:mean_theo_2}
For all $y,z\in(0,1)$, we have:
\begin{equation*}
\lim_{n\rightarrow\infty }\E\left[U_{n+1}(y)U_{n+1}(z)\right]=\Sigma(y,z).
\end{equation*}
\end{lem}

\begin{proof}
Recall \reff{eq:H_n} and \reff{eq:H_n_star}. With $d=D(y)$, we notice that for $y\in(0,1)$:
\begin{equation*}
 \pro\left(\left.D_{i}^{(n+1)}\le d\right|X_{j}\right)-c_{n}(y)
 =\begin{cases}
   \frac{1}{n}H_{n}(y,X_j) &\mbox{ if } i\neq j, \\
   H_{n}^{\star}(y,X_j)    &\mbox{ if } i=j.
  \end{cases}
\end{equation*}
We have:
\begin{align}
U_{n+1}(y)&=(n+1)^{-\frac{1}{2}}\sum_{i=1}^{n+1}\sum_{j=1}^{n+1}
            \left[\pro\left(\left.D_{i}^{(n+1)}\le
            d\right|X_{j}\right)-c_{n}(y)\right]
\nonumber\\ 
           &=(n+1)^{-\frac{1}{2}}\sum_{j=1}^{n+1}
             \left[H_{n}^{\star}(y,X_j)+H_{n}(y,X_j)\right] \label{eq:U_n_H_n}.  
\end{align}

Let $y, z\in (0, 1)$. 
Since $H_{n}$ and $H_{n}^{\star}$ are centered  and $(X_i:i\in\Ne)$
are independent, using \reff{eq:U_n_H_n}, we obtain that:
\begin{multline*}
 \E\left[U_{n+1}(y)U_{n+1}(z)\right]
=\E\left[H_{n}^{\star}(y,X_1)H_{n}^{\star}(z,X_1)\right]
      +\E\left[H_{n}(y,X_1)H_{n}(z,X_1)\right]\\
+ \E\left[H_{n}^{\star}(y,X_1)H_{n}(z,X_1)\right]
+\E\left[H_{n}^{\star}(z,X_1)H_{n}(y,X_1)\right].
\end{multline*}
Recall $\Sigma=\Sigma_{1}+\Sigma_{2}+\Sigma_{3}$  defined in Remark 
\ref{rem:Sigma}.

\subsection*{Study of $\E\left[H_{n}^{\star}(y,X_1)H_{n}^{\star}(z,X_1)\right]$}
By Proposition \ref{prop:mean_phi_n}, see \reff{eq:H_n_star_limit},  and
by dominated convergence, we get that: 
\begin{equation}\label{eq:term_1}
 \lim_{n\to\infty}\E\left[H_{n}^{\star}(y,X_1)H_{n}^{\star}(z,X_1)\right]
 =\E\left[\left(\1{X_1\le y}-y\right)\left(\1{X_2\le z}-z\right)\right]
 =\Sigma_{1}(y,z).
\end{equation}

\subsection*{Study of $\E\left[H_{n}(y,X_1)H_{n}(z,X_1)\right]$}
By Proposition \ref{prop:mean_phi_n}, we have:
\begin{multline*}
 \E\left[H_{n}(y,X_1)H_{n}(z,X_1)\right]\\
\begin{aligned}
 &=\E\left[\left(\frac{D(y)-W(y,X_1)}{D'(y)}+R_{n}^{\ref{prop:mean_phi_n}}(y,X_1)\right)
           \left(\frac{D(z)-W(z,X_1)}{D'(z)}+R_{n}^{\ref{prop:mean_phi_n}}(z,X_1)\right)\right]\\
 &=\frac{1}{D'(y)}\frac{1}{D'(z)}\E\left[(D(y)-W(y,X_1))(D(z)-W(z,X_1))\right]+R_{n}^{(1)}\\
 &=\Sigma_{2}(y_1,y_2)+R_{n}^{(1)}
\end{aligned}
\end{multline*}
where, because of \reff{eq:H_n_rest},  $\left|R_{n}^{(1)}\right|\le
Cn^{-\frac{1}{4}}$ for some finite constant $C$.
We obtain that
\begin{equation}\label{eq:term_2}
\lim_{n\rightarrow\infty } \E\left[H_{n}(y,X_1)H_{n}(z,X_1)\right] = \Sigma_{2}(y,z).
\end{equation}

\subsection*{Study of $\E\left[H_{n}^{\star}(y,X_1)H_{n}(z,X_1)\right]+\E\left[H_{n}^{\star}(z,X_1)H_{n}(y,X_1)\right]$}
By Proposition \ref{prop:mean_phi_n}, we have that:
\begin{equation*}
 \E\left[H_{n}^{\star}(y,X_1)H_{n}(z,X_1)\right]
 =\E\left[H_{n}^{\star}(y,X_1)\frac{1}{D'(z)}(D(z)-W(z,X_1))
 +H_{n}^{\star}(y,X_1)R_{n}^{\ref{prop:mean_phi_n}}(z,X_1)\right].
\end{equation*}

Thanks to \reff{eq:H_n_rest} and \reff{eq:H_n_star_limit}, we have 
$\left|H_{n}^{\star}(y,X_1)\right|\le 1$ and $\E\left[\left|R_{n}^{\ref{prop:mean_phi_n}}(z,X_1)\right|\right]=O\left(n^{-1/4}\right)$. 
We deduce from Proposition \ref{prop:mean_phi_n} and dominated convergence, that:
\begin{align*}
 \lim_{n\to\infty}\E\left[H_{n}^{\star}(y,X_1)H_{n}(z,X_1)\right]
  &=\E\left[\left(\1{X_1 \le y}-y\right)\frac{1}{D'(z)}(D(z)-W(z,X_1))\right]\\
  &=\frac{1}{D'(z)}\left(yD(z)-\int_{0}^{y}W(z,x)dx\right).
\end{align*}
By symmetry, we finally obtain that
\begin{equation}\label{eq:term_3}
 \lim_{n\rightarrow\infty }\E\left[H_{n}^{\star}(y,X_1)H_{n}(z,X_1)\right]+\E\left[H_{n}^{\star}(z,X_1)H_{n}(y,X_1)\right] = \Sigma_{3}(y,z).
\end{equation}

\subsection*{Conclusion}
Combining \reff{eq:term_1}, \reff{eq:term_2} and \reff{eq:term_3}, we get that
\begin{equation*}
 \lim_{n\rightarrow\infty } \E\left[U_{n+1}(y)U_{n+1}(z)\right] = \Sigma(y,z).
\end{equation*}
\end{proof}

\begin{lem}\label{mlem:mean_theo_U_stat}
We have the following convergence of finite-dimensional distributions:
 \begin{equation*}
\left(U_{n+1}(y):y\in(0,1)\right)
  \,\xrightarrow[n\rightarrow+\infty]{(fdd)}\, \chi ,
 \end{equation*}
where $\chi=(\chi(y): y\in(0,1))$ is a centered Gaussian
process with covariance function $\Sigma$ given in Remark
\ref{rem:Sigma}. 
\end{lem}

\begin{proof}
  Let  $k\in\Ne$  and  $(y_{1},\dots,y_{k})\in(0,1)^k$.  We  define  the
  random  vector $U_{n+1}^{(k)}=(U_{n+1}(y_i):  \, i\in  [k])$. For  all
  $y\in(0,1)$          and          $j\in[n+1]$,         we          set
  $ g_{n}(y,X_j)=\left[H_{n}^{\star}(y,X_j)+H_{n}(y,X_j)\right]$.  Using
  \reff{eq:U_n_H_n}, we have:
\begin{equation*}
U_{n+1}^{(k)}=(n+1)^{-\frac{1}{2}}\sum_{j=1}^{n+1}Z_{j}^{(n+1)},
\end{equation*}
where   $Z_{j}^{(n+1)}=(g_{n}(y_i,X_j):\, i\in [k])$.    Notice
$(Z_{j}^{(n+1)}:j\in[n+1])$  is  a  sequence of  independent,  uniformly
bounded   (see   Proposition  \ref{prop:mean_phi_n})   and   identically
distributed   random  vectors   with     mean   zero  and   common
positive-definite                    covariance                   matrix
$V_{n+1}=\Cov\left(Z_{1}^{(n+1)}\right)$. According to Lemma
\ref{lem:mean_theo_2}, we have that $\lim_{n\rightarrow\infty } V_{n+1}
=\Sigma^{(k)}$, with $\Sigma^{(k)}=(\Sigma(y_i, y_j):\,  i, j\in [k])$. The      multidimensional
Lindeberg-Feller  condition is  trivially  satisfied as
$(Z_{j}^{(n+1)}:j\in[n+1])$  are bounded (uniformly in $n$) with the
same distribution. We deduce from the 
multidimensional central  limit theorem for triangular  arrays of random
variables, see
\cite{Bhattacharya_Rao_2010_book} Corollary 18.2,  that 
$(U_{n+1}^{(k)}: n\geq 0)$ converges in distribution towards the
Gaussian random vector with distribution 
$\mathcal{N}(0,\Sigma^{(k)})$. This gives the result. 
\end{proof}

Recall $\hat{\cdf}_{n}(y)$ defined in \eqref{eq:def_hat_cdf}.
In view of Lemma \ref{mlem:mean_theo_U_stat} and since $c_n(y)=y+
O(1/n)$, in order to prove Theorem \ref{thm:CLT_indicator_degree_sequence}, 
it is enough to prove  that for all $y\in(0,1)$:
\begin{equation}\label{eq:cv_L2_cdf_U_n}
 \hat{\cdf}_{n+1}(y)-U_{n+1}(y)\,\xrightarrow[n\rightarrow+\infty]{L^{2}}\, 0.
\end{equation}
Because $\hat{\cdf}_{n+1}$ and $U_{n+1}$ are centered, we deduce from
\reff{eq:haj_proj} that:
\[
 \E\left[\left(\hat{\cdf}_{n+1}(y)-U_{n+1}(y)\right)^2\right]
=\E\left[\hat{\cdf}_{n+1}(y)^2\right]-\E\left[U_{n+1}(y)^2\right].
\]
By Lemma \ref{lem:mean_theo_2}, we have
$\E\left[U_{n+1}(y)^2\right]\cv{n} \Sigma(y,y)$. So we deduce that the
proof of Theorem \ref{thm:CLT_indicator_degree_sequence} is a complete
as soon as the next lemma is proved.

\begin{lem}\label{lem:mean_theo_1}
 For all $y\in(0,1)$, we have 
\begin{equation*}
 \lim_{n\rightarrow\infty }\E\left[\hat{\cdf}_{n+1}(y)^2\right]=\Sigma(y,y).
\end{equation*}
\end{lem}

\begin{proof}
Let $y\in(0,1)$ and $d=D(y)$.
We have
\begin{align*}
 \E\left[\hat{\cdf}_{n+1}(y)^2\right]
 &=\frac{1}{n+1}\sum_{i,j=1}^{n+1}\E\left[\left(\1{D_{i}^{(n+1)}\le d}-c_{n}(y)\right)\left(\1{D_{j}^{(n+1)}\le d}-c_{n}(y)\right)\right]\\
 &=\E\left[\1{D_{1}^{(n+1)}\le d}\right]-c_{n}(y)^2
  +n\left\{\E\left[\1{D_{1}^{(n+1)}\le d}\1{D_{2}^{(n+1)}\le d}\right]-c_{n}(y)^2\right\}\\
 &=c_{n}(y)-(n+1)c_{n}(y)^2+n\,\E\left[\1{D_{1}^{(n+1)}\le d}\1{D_{2}^{(n+1)}\le d}\right].
\end{align*}
So we get that
\begin{equation*}
 \E\left[\hat{\cdf}_{n+1}(y)^2\right]
 =B_{n}^{(1)}+B_{n}^{(2)}+B_{n}^{(3)}+B_{n}^{(4)},
\end{equation*}
where
\begin{align*}
 B_{n}^{(1)}
&=c_{n}(y)-c_{n}(y)^2,\\
 B_{n}^{(2)}
&=-n(c_{n}(y)-y)^2,\\
 B_{n}^{(3)}
&=n\,\E\left[\left(\1{D_{1}^{(n+1)}\le d}-\1{X_{1}\le y}\right)
  \left(\1{D_{2}^{(n+1)}\le d}-\1{X_{2}\le y}\right)\right],\\ 
B_{n}^{(4)}
&=2n\,\E\left[\1{X_1\le y}\left(\1{D_{2}^{(n+1)}\le d}-c_{n}(y)\right)\right].
\end{align*}

By Equation \reff{eq:lim_c_n}, we get  $ \lim_{n\rightarrow\infty
}B_{n}^{(1)}=  \Sigma_{1}(y,y)$
and $\lim_{n\rightarrow \infty }  B_{n}^{(2)}=0$. 
By Proposition \ref{prop:ind10}, we get $\lim_{n\rightarrow\infty
}  B_{n}^{(3)}=  \Sigma_{2}(y,y)$. 
Using \reff{eq:H_n}, we get $ B_{n}^{(4)}=2\,\E\left[\1{X_{1}\le y}H_{n}(y,X_1)\right]$.
By Proposition \ref{prop:mean_phi_n} and dominated convergence, we get that:
\begin{align*}
 \E\left[\1{X_{1}\le y}H_{n}(y,X_1)\right]
 &= \E\left[\1{X_{1}\le y}\left(\frac{1}{D'(y)}\left(D(y)-W(y,X_1)\right)+R_{n}^{\ref{prop:mean_phi_n}}(y,X_{1})\right)\right]\\
&\xrightarrow[n\rightarrow\infty]{}\,
  \frac{1}{D'(y)}\left(yD(y)-\int_{0}^{y}W(y,x)dx\right).
\end{align*}
This gives $\lim_{n\rightarrow\infty }
 B_{n}^{(4)}= \Sigma_{3}(y,y)$. 
Then, we get that $\lim_{n\rightarrow\infty }
\E\left[\hat{\cdf}_{n+1}(y)^2\right]=\Sigma(y,y)$. 
\end{proof}

 \bibliographystyle{abbrv}
 \bibliography{biblio_article_graphon}

\begin{thebibliography}{10}

\bibitem{Barbour_Karonski_Rucinski_1989_article}
A.~D. Barbour, M.~Karo\'nski, and A.~Ruci\'nski.
\newblock A central limit theorem for decomposable random variables with
  applications to random graphs.
\newblock {\em J. Combin. Theory Ser. B}, 47(2):125--145, 1989.

\bibitem{Bentkus_2003_article}
V.~Bentkus.
\newblock On the dependence of the {B}erry-{E}ss\`een bound on dimension.
\newblock {\em J. Statist. Plann. Inference}, 113(2):385--402, 2003.

\bibitem{Bhattacharya_Rao_2010_book}
R.~N. Bhattacharya and R.~R. Rao.
\newblock {\em Normal approximation and asymptotic expansions}, volume~64 of
  {\em Classics in Applied Mathematics}.
\newblock Society for Industrial and Applied Mathematics (SIAM), Philadelphia,
  2010.

\bibitem{Bickel_Chen_Levina_2011_article}
P.~J. Bickel, A.~Chen, and E.~Levina.
\newblock The method of moments and degree distributions for network models.
\newblock {\em Ann. Statist.}, 39(5):2280--2301, 2011.

\bibitem{Diaconis_Blitzstein_2010_article}
J.~Blitzstein and P.~Diaconis.
\newblock A sequential importance sampling algorithm for generating random
  graphs with prescribed degrees.
\newblock {\em Internet Math.}, 6(4):489--522, 2010.

\bibitem{Bollobas_Riordan_2009_article}
B.~Bollob\'as and O.~Riordan.
\newblock Metrics for sparse graphs.
\newblock In {\em Surveys in combinatorics 2009}, volume 365 of {\em London
  Math. Soc. Lecture Note Ser.}, pages 211--287. Cambridge Univ. Press, 2009.

\bibitem{Borgs_Chayes_Lovasz_Sos_Vesztergombi}
C.~Borgs, J.~T. Chayes, L.~Lov\'asz, V.~T. S\'os, and K.~Vesztergombi.
\newblock Convergent sequences of dense graphs. {I}. {S}ubgraph frequencies,
  metric properties and testing.
\newblock {\em Adv. Math.}, 219(6):1801--1851, 2008.

\bibitem{Chatterjee_Diaconis_Sly_2011_article}
S.~Chatterjee, P.~Diaconis, and A.~Sly.
\newblock Random graphs with a given degree sequence.
\newblock {\em Ann. Appl. Probab.}, 21(4):1400--1435, 2011.

\bibitem{Chen_Fang_2015_article}
L.~H. Chen and X.~Fang.
\newblock Multivariate normal approximation by {S}tein's method: The
  concentration inequality approach.
\newblock {\em arXiv preprint arXiv:1111.4073v2}, 2015.

\bibitem{Chung_2006_book}
F.~Chung, F.~R. Chung, F.~C. Graham, L.~Lu, K.~F. Chung, et~al.
\newblock {\em Complex graphs and networks}.
\newblock Number 107. American Mathematical Soc., 2006.

\bibitem{Coulson_Gaunt_Reinert_2016_article}
M.~Coulson, R.~E. Gaunt, and G.~Reinert.
\newblock Poisson approximation of subgraph counts in stochastic block models
  and a graphon model.
\newblock {\em ESAIM Probab. Stat.}, 20:131--142, 2016.

\bibitem{Fang_Rollin_2015_article}
X.~Fang and A.~R\"ollin.
\newblock Rates of convergence for multivariate normal approximation with
  applications to dense graphs and doubly indexed permutation statistics.
\newblock {\em Bernoulli}, 21(4):2157--2189, 2015.

\bibitem{Feray_Meliot_Nikeghbali_2016_book}
V.~F\'eray, P.-L. M\'eliot, and A.~Nikeghbali.
\newblock {\em Mod-{$\phi$} convergence}.
\newblock SpringerBriefs in Probability and Mathematical Statistics. Springer,
  2016.

\bibitem{Feray_Meliot_Nikeghbali_2017_article}
V.~F{\'e}ray, P.-L. M{\'e}liot, and A.~Nikeghbali.
\newblock Graphons, permutons and the {T}homa simplex: three mod-{G}aussian
  moduli spaces.
\newblock {\em arXiv preprint arXiv:1712.06841}, 2017.

\bibitem{Gilmer_Kopparty_2016_article}
J.~Gilmer and S.~Kopparty.
\newblock A local central limit theorem for triangles in a random graph.
\newblock {\em Random Structures Algorithms}, 48(4):732--750, 2016.

\bibitem{h:sad}
W.~Hoeffding.
\newblock A class of statistics with asymptotically normal distribution.
\newblock {\em Ann. Math. Statistics}, 19:293--325, 1948.

\bibitem{Janson_Nowicki_1991_article}
S.~Janson and K.~Nowicki.
\newblock The asymptotic distributions of generalized {$U$}-statistics with
  applications to random graphs.
\newblock {\em Probab. Theory Related Fields}, 90(3):341--375, 1991.

\bibitem{kt}
K.~Krokowski and C.~Thaele.
\newblock Multivariate central limit theorems for {R}ademacher functionals with
  applications.
\newblock {\em arXiv preprint arXiv:1701.07365}, 2017.

\bibitem{Lovasz_2012_book}
L.~Lov{\'a}sz.
\newblock {\em Large networks and graph limits}, volume~60 of {\em American
  Mathematical Society Colloquium Publications}.
\newblock American Mathematical Society, Providence, RI, 2012.

\bibitem{Lovasz_Szegedy_2006_article}
L.~Lov{\'a}sz and B.~Szegedy.
\newblock Limits of dense graph sequences.
\newblock {\em J. Combin. Theory Ser. B}, 96(6):933--957, 2006.

\bibitem{Maugis_Priebe_Olhede_Wolfe_2017_article}
P.-A.~G. Maugis, C.~E. Priebe, S.~C. Olhede, and P.~J. Wolfe.
\newblock Statistical inference for network samples using subgraph counts.
\newblock {\em arXiv preprint arXiv:1701.00505}, 2017.

\bibitem{Molloy_Reed_1995_article}
M.~Molloy and B.~Reed.
\newblock A critical point for random graphs with a given degree sequence.
\newblock {\em Random structures \& algorithms}, 6(2-3):161--180, 1995.

\bibitem{Molloy_Reed_1998_article}
M.~Molloy and B.~Reed.
\newblock The size of the giant component of a random graph with a given degree
  sequence.
\newblock {\em Combinatorics, probability and computing}, 7(3):295--305, 1998.

\bibitem{Nagaev_Chebotarev_Zolotukhin_2016_article}
S.~V. Nagaev, V.~I. Chebotarev, and A.~Y. Zolotukhin.
\newblock A non-uniform bound of the remainder term in the central limit
  theorem for {B}ernoulli random variables.
\newblock {\em J. Math. Sci. (N.Y.)}, 214(1):83--100, 2016.

\bibitem{Newman_Barabasi_Watts_2011_book}
M.~Newman, A.-L. Barabasi, and D.~J. Watts.
\newblock {\em The structure and dynamics of networks}, volume~19.
\newblock Princeton University Press, 2011.

\bibitem{Newmann_2003_book}
M.~E. Newman.
\newblock The structure and function of complex networks.
\newblock {\em SIAM review}, 45(2):167--256, 2003.

\bibitem{Newman_Strogatz_Watts_2001_article}
M.~E. Newman, S.~H. Strogatz, and D.~J. Watts.
\newblock Random graphs with arbitrary degree distributions and their
  applications.
\newblock {\em Physical review E}, 64(2):026118, 2001.

\bibitem{Nowicki_1989_article}
K.~Nowicki.
\newblock Asymptotic normality of graph statistics.
\newblock {\em J. Statist. Plann. Inference}, 21(2):209--222, 1989.

\bibitem{Nowicki_Wierman_1988_article}
K.~Nowicki and J.~C. Wierman.
\newblock Subgraph counts in random graphs using incomplete {$U$}-statistics
  methods.
\newblock In {\em Proceedings of the {F}irst {J}apan {C}onference on {G}raph
  {T}heory and {A}pplications ({H}akone, 1986)}, volume~72, pages 299--310,
  1988.

\bibitem{Reinert_Rollin_2010_article}
G.~Reinert and A.~R\"ollin.
\newblock Random subgraph counts and {$U$}-statistics: multivariate normal
  approximation via exchangeable pairs and embedding.
\newblock {\em J. Appl. Probab.}, 47(2):378--393, 2010.

\bibitem{Rucinski_1988_article}
A.~Ruci\'nski.
\newblock When are small subgraphs of a random graph normally distributed?
\newblock {\em Probab. Theory Related Fields}, 78(1):1--10, 1988.

\bibitem{Uspensky_1937_article}
J.~V. Uspensky.
\newblock {\em Introduction to Mathematical Probability}.
\newblock McGraw-Hill Book Company, New York, 1937.

\bibitem{Van_der_Vaart_1998_article}
A.~W. van~der Vaart.
\newblock {\em Asymptotic statistics}, volume~3 of {\em Cambridge Series in
  Statistical and Probabilistic Mathematics}.
\newblock Cambridge University Press, 1998.

\end{thebibliography}

\cuthere \\
\begin{center}
 \textbf{Notation Index}
\end{center}
\vspace{1em}

\setlength{\columnseprule}{1pt} 
\begin{multicols}{2}
\noindent
\hrulefill
\begin{itemize}[leftmargin=*,itemsep=0.5em]
\item[] $|A|$ cardinal of set $A$
  \item[] $[n]=\{1,\dots,n\}$ 
\item[] $|\beta|$ length of $[n]$-word $\beta$
  \item[] $\mathcal{M}_{n}$ set of $[n]$-words with all characters distinct  
  \item[] $\mathcal{S}_{n,p}=\{\beta\in \cm_n: |\beta|=p\}$
  \item[] $|\mathcal{S}_{n,p}|=A_{n}^{p}=n!/(n-p)!$ 
\item[] $\beta_{\ell}=\beta_{\ell_1}\dots\beta_{\ell_k}$ for $\ell\in
\cs_{p,k}$ and $\beta\in \cs_{n,p}$
  \item[]
    $\mathcal{S}_{n,k}^{\ell,\alpha}=\left\{\beta\in\mathcal{S}_{n,k}:
      \beta_{\ell}=\alpha\right\}$ for $\alpha\in \cs_{n,k}$
\item[] $|\mathcal{S}_{n,k}^{\ell,\alpha}|=A_{n-k}^{p-k}=(n-k)!/(n-p)!$
\end{itemize}
\hrulefill
\begin{itemize}[leftmargin=*,itemsep=0.5em]
\item[] $\cf$ set of  simple  finite graphs
\item[] $F$  a  simple finite graph (and a finite sequence of simple
  graphs in Sections 3 to 6 satisfying condition \reff{eq:Hyp-hat-E})
  \item[] $E(F)$ set of edges of $F$
  \item[] $V(F)$ set of vertices of $F$
\item[] $v(F)=|V(F)|$ number of vertices of $F$
  \item[]  $G_n=G_{n}(W)$ W-random graph with $n$ vertices associated to
    the sequence 
    $X=(X_k, k\in \N^*)$ of i.i.d. uniform random variables on $[0, 1]$
  \end{itemize}
  \hrulefill
\begin{itemize}[leftmargin=*,itemsep=0.5em]
  \item[] $t(F,G)$ density of hom. from $F$ to $G$
  \item[] $\ti(F,G)$ density of injective hom.
  \item[] $\tid(F,G)$ density  of embeddings   
  \item[] $Y^\beta(F,G)= \prod_{\{i,j\}\in E(F)}\1{\{\beta_i,\beta_j\}\in E(G)}$
  \item[] $\ti (F,G)={|\mathcal{S}_{n,p}|}^{-1}\, \sum_{\beta\in
      \cs_{n,p} }Y^\beta(F, G)$
  \end{itemize}
  \hrulefill
\begin{itemize}[leftmargin=*,itemsep=0.5em]
\item[] \begin{center}
\vspace{.2cm}
$\boxed{\text{$\ell \in \cm_p$ and  $\alpha\in \cs_{p,k}$ with $k=|\ell|$}}$
\vspace{.2cm}
\end{center}
  \item[] $\ti(F^{\ell},G^{\alpha})$ density of injective hom. such
    that the  labelled vertices $\ell$ of $F$, with $V(F)=[p]$,  are sent on the  labelled
    vertices $\alpha$ of $G$, with $V(G)=[n]$
\item[] $F^{[\ell]}$ sub-graph of the labeled vertices $\ell$ of $F$
  \item[] $Y^\beta(F^\ell, G^\alpha)=Y^\beta(F,G)$ for $\beta\in \cs_{n,p}^{\ell,\alpha}$
  \item[] $\hat Y^\alpha(F^\ell, G^\alpha)= \prod_{\{i,j\}\in
      E(F^{[\ell]})}\1{\{\alpha_i,\alpha_j\}\in E(G)}$ 
  \item[] $Y^\beta(F^\ell, G^\alpha)=\hat{Y}^\alpha(F^\ell, G^\alpha)\,
    \tilde Y^\beta(F^\ell, G^\alpha)$ i.e.:
\begin{center}
\vspace{.2cm}
$\boxed{Y^\beta=\hat{Y}^\alpha\,
    \tilde Y^\beta}$
\vspace{.2cm}
\end{center}
  \item[] $\tti(F^{\ell},G^{\alpha})={|\cs_{n,p}^{\ell,\alpha}|}^{-1} \sum_{\beta\in \cs_{n,p}^{\ell,\alpha}  } \,\, \tilde Y^\beta$
  \item[] \begin{center}
   $\boxed{\ti(F^{\ell},G^{\alpha})=\hat Y^\alpha \,\,\tti(F^{\ell},G^{\alpha})}$
  \end{center}
  \end{itemize}
 \hrulefill
\begin{itemize}[leftmargin=*,itemsep=0.5em]
  \item[] $t(F,W)$ hom. densities for graphon  $W$
\item[]  $\tid(F,W)$  density  of embeddings 
\item[]   $X_\alpha=(X_{\alpha_1}, \ldots, X_{\alpha_k})$ and simil. for $X_\beta$
\item[] $Z^\beta=Z(X_\beta)=\E[Y^\beta(F^\ell, G_n^\alpha)|X]$ 
\item[] $\tilde Z^\beta=\E[\tilde Y^\beta(F^\ell, G^\alpha_n)|X]$ 
\item[] $t_{x}=t_{x}(F^\ell, W)=\E[Z^\beta| X_\alpha=x]$ and
\item[] $t_{x}=\E[\ti(F^\ell, G_n^\alpha)|X_\alpha=x]$
\item[] $\tilde t_{x}=\tilde t_{x}(F^\ell, W)=\E[\tilde Z^\beta|
  X_\alpha=x]$ and
\item[] $\tilde t_{x}=\E[\tti(F^\ell, G_n^\alpha)|X_\alpha=x]$
\item[] $\hat t_{x}=\hat t_{x}(F^\ell, W)=\E[\hat Y^\alpha(F^\ell, G_n^\alpha)| X_\alpha=x]$
\begin{center}
\vspace{.2cm}
$\boxed{t_{x}=\hat t_{x} \, \tilde t_{x}}$ $ $ for $x\in [0, 1]^k$
\vspace{.2cm}
\end{center}
   \item[] $ t(F^\ell, W)= \int_{[0,1]^k} t_{x}\, dx=\E[\ti(F, G_n)]$
   \item[] $\hat t(F^\ell, W)= \int_{[0,1]^k}\hat t_{x}\,
     dx=\E[\ti(F^{[\ell]}, G_n)]$
and $\hat t(F^\ell, W)= t(F^{[\ell]}, W)$

  \end{itemize}
  \hrulefill
\begin{itemize}[leftmargin=*,itemsep=0.5em]
  \item[] $\Gamma_{n}^{F,\ell}$ random probability measure: 
\[\Gamma_{n}^{F,\ell}(g)={|\cs_{n,k}|}^{-1}
        \sum_{\alpha\in\mathcal{S}_{n,k}}g\left(\ti
          (F^{\ell},G_{n}^{\alpha})\right)\]
  \item[] $\Gamma^{F,\ell}(dx)=\E\left[\Gamma_{n}^{F,\ell}(dx)\right]$ 
  \item[] $\sigma^{F,\ell}(g)^2$ asymptotic variance of
\[\sqrt{n}\left(\Gamma_{n}^{F,\ell}(g) - \Gamma^{F,\ell}(g)\right)\]
  \end{itemize}
  \hrulefill
\begin{itemize}[leftmargin=*,itemsep=0.5em]
\setlength\itemsep{0.5em}
  \item[] $nD_{i}^{(n)}$  degree of  $i$ in $G_n$
  \item[] $\Pi_{n}$ empirical CDF of the degrees of $G_{n}$
  \item[] $D(x)=\int_{[0, 1]} W(x,y) dy$ degree funct. of $W$
\end{itemize}
\hrulefill
\begin{itemize}[leftmargin=*,itemsep=0.5em]
  \item[] $\mathcal{H}_{n,d,\delta}(\rp)=\pro(X\le nd+\delta)$ for $X\sim\mathcal{B}(n,\rp)$
  \item[] $\sigma_{(x)}^2=x(1-x)$
  \item[] $S(x)=\lceil x \rceil -x -\frac{1}{2}$
  \item[] $\Phi$  the CDF  of $\cn(0, 1)$
\item[]  $\varphi$ probability distribution density of $\cn(01)$
            
\end{itemize}
\hrulefill
\begin{itemize}[leftmargin=*,itemsep=0.5em]
   \item[] \begin{center}
   $\boxed{d=D(y)}$\end{center}
 \item[] \begin{center}
   $\boxed{c_{n}(y)=\pro(D_{1}^{(n+1)}\le d)}$
 \end{center}
  \item[] $ H_{n}^{\star}(y,u)=\P(D_{1}^{(n+1)}\le
          d\,|\,X_{1}=u)- c_{n}(y)$
  \item[] $ \frac{H_{n} (y,u)}{n}=\P(D_{1}^{(n+1)}\le
          d\,|\,X_{2}=u)- c_{n}(y)$
\end{itemize}
\hrulefill
\end{multicols}
\cuthere


\section{Appendix A: Preliminary results for the CDF of binomial
  distributions} 
\label{sec:preliminaries_CDF_bin}

In this section, we study uniform asymptotics for the CDF of binomial
distributions.
Let $n\in\Ne$, $d\in[0,1]$, $\delta\in \R$ and $\rp\in(0,1)$.
We  consider the CDF:
\begin{equation}\label{eq_FDR_bin}
 \mathcal{H}_{n,d,\delta}(\rp)=\pro\left(X\le nd+\delta\right),
\end{equation}
where $X$ a binomial random variable with paramaters $(n,\rp)$.
We denote by $\Phi$ the cumulative distribution function of the standard Gaussian distribution
and by $\varphi$ the probability distribution density of the standard Gaussian
distribution.
We recall \reff{eq:def_sigma} and \reff{eq:def_S}: 
$ \sigma_{(x)}^2=x(1-x)$ for  $ x\in[0,1]$, and $S(x)=\lceil x \rceil
-x -\frac{1}{2}$ for  $ x\in\R$.
\\

We recall a result from
\cite{Nagaev_Chebotarev_Zolotukhin_2016_article}, see also 
\cite{Uspensky_1937_article}, Chapter VII: for all $x\in\R$, for all $\rp\in(0,1)$ and 
$n\in\Ne$ such that $n\sigma_{(\rp)}^2\ge 25$, we have:
\begin{equation}\label{eq:thm_NCZ}
 \pro\left(X\le n\rp+\sqrt{n}\,\sigma_{(\rp)}x\right)
 = \Phi(x)+\frac{1}{\sqrt{n}}\mathcal{Q}(\rp,x)+\frac{1}{\sqrt{n}\,\sigma_{(\rp)}}S(n\rp+x\sqrt{n}\,\sigma_{(\rp)})\varphi(x)
  +U_{n}(\rp,x),
\end{equation}
where
\begin{equation*}
 \mathcal{Q}(\rp,x)=\frac{2\rp-1}{6\sigma_{(\rp)}}\varphi''(x)=\frac{2\rp-1}{6\sigma_{(\rp)}}(x^{2}-1)\varphi(x),
\end{equation*}
and
\begin{equation}
\label{eq:bd-Un}
 \left|U_{n}(\rp,x)\right|\le \frac{0.2+0.3\left|2\rp-1\right|}{n\sigma_{(\rp)}^2}+\exp\left(-\frac{3\sqrt{n}\,\sigma_{(\rp)}}{2}\right).
\end{equation}
We use this result to give an approximation of
$\mathcal{H}_{n,d,\delta}\left(d+\frac{s}{\sqrt{n}}\right)$. 

\begin{prop}\label{prop:CFD_approxi}
  Let               $\varepsilon_{0}\in(0,\frac{1}{2})$              and
  $K_{0}=[\varepsilon_{0},1-\varepsilon_{0}]$.  Let  $\alpha>0$.   There
  exists  a  positive  constant  $C$   such  that  for  all  $n\geq  2$,
  $s\in[-\alpha\sqrt{\log ( n)},\alpha\sqrt{\log (n)}]$,  $\delta\in[-1,1]$
  and $d\in K_0$ such that $d+\frac{s}{\sqrt{n}}\in K_{0}$, we have:
 \begin{equation*}
  \mathcal{H}_{n,d,\delta}\left(d+\frac{s}{\sqrt{n}}\right)
  =\Phi(y_{s})+\frac{1}{\sqrt{n}}\frac{\varphi(y_{s})}{\sigma_{(d)}}\pi(s,n,d,\delta)+R^{\ref{prop:CFD_approxi}}(s,n,d,\delta),
 \end{equation*}
 where 
 \begin{equation}\label{eq:def_y_s}
 y_{s}=\frac{-s}{\sigma_{(d)}} \quad\text{and}\quad \pi(s,n,d,\delta)=
 \frac{1-2d}{6}( 1 +2 y_{s}^2)+S(nd+\delta)+\delta
 \end{equation}
and 
 \begin{equation*}
  \left|R^{\ref{prop:CFD_approxi}}(s,n,d,\delta)\right|\le C\frac{\log(n)^{2}}{n}\cdot
 \end{equation*}
\end{prop}

\begin{proof}
  In  what follows,  $C$ denotes  a positive  constant which  depends on
  $\varepsilon_0$       (but        not       on        $n\geq       2$,
  $s\in[-\alpha\sqrt{\log          (n)},\alpha\sqrt{\log         (n)}]$,
  $\delta\in[-1,1]$       and      $d\in       K_0$      such       that
  $d+\frac{s}{\sqrt{n}}\in K_{0}$)  and which  may change from  lines to
  lines.  We will also use, without recalling it, that 
  $\sigma_{(.)}$ is uniformly bounded away from $0$ on $K_0$.

For all $\theta\in(0,1]$ such that $d+s\theta\in(0,1)$, we set 
 \begin{equation}\label{eq:def_x_s}
  x_{s}(\theta)=\frac{-s+\delta \theta}{\sigma_{(d+s\theta)}}\cdot
 \end{equation}

Let $\rp=d+\frac{s}{\sqrt{n}}$ and $X$ be a binomial random variable with parameters $(n,\rp)$.
Because $nd+\delta=n\rp+\sqrt{n}\sigma_{(\rp)}\,x_{s}\left(\frac{1}{\sqrt{n}}\right)$, we can write 
\begin{equation}\label{eq:appli_NCZ}
 \mathcal{H}_{n,d,\delta}\left(d+\frac{s}{\sqrt{n}}\right)=\pro\left(X\le n\rp+\sqrt{n}\,\sigma_{(\rp)}\,x_{s}\left(\frac{1}{\sqrt{n}}\right)\right).
\end{equation}
Recall $S$ is defined  in \eqref{eq:def_S}.  Using \reff{eq:thm_NCZ}, we
get that:
 \begin{multline}
\label{eq:NCZ}
  \mathcal{H}_{n,d,\delta}\left(d+\frac{s}{\sqrt{n}}\right)
  =\Phi\left(x_{s}\left(\frac{1}{\sqrt{n}}\right)\right)+\frac{1}{\sqrt{n}}\,Q^{(1)}_{d,\delta}\left(s,\frac{1}{\sqrt{n}}\right)\\
   +\frac{1}{\sqrt{n}}\,\,Q^{(2)}_{d,\delta}\left(s,\frac{1}{\sqrt{n}}\right)+U_{n}\left(d+\frac{s}{\sqrt{n}},x_{s}\left(\frac{1}{\sqrt{n}}\right)\right)
 \end{multline}
 where for $\theta\in(0,1]$ such that $d+s\theta\in(0,1)$,
 \begin{align*}
  Q^{(1)}_{d,\delta}(s,\theta)
&=\frac{2(d+s\theta)-1}{6\sigma_{(d+s\theta)}}\left(x_{s}(\theta)^2-1\right)\varphi(x_{s}(\theta)), \\
  Q^{(2)}_{d,\delta}(s,\theta)
&=\frac{1}{\sigma_{(d+s\theta)}}S\left(d\theta^{-2}+\delta\right)\varphi(x_{s}(\theta)).
 \end{align*}

\subsection*{Study of the first term of the right hand side of \reff{eq:NCZ}}
Let        $\theta\in (0,    1/{\sqrt{2}}]$, and notice        that
$|\log(\theta)|\ge\log(\sqrt{2})>0$.    Recall    the   definition    of
$x_{s}(\theta)$ given by \reff{eq:def_x_s}.  By simple computations, we
get     that      for     all      $0<\theta\le   1/\sqrt{2}$,
$|s|\le  \alpha\sqrt{2|\log (\theta)|}$,  $|\delta|\le 1$,  and $d\in  K_0$
such that $d+s\theta\in K_{0}$,
\begin{equation}\label{eq:x_s_theta}
 \lvert x_{s}(\theta) \rvert \le C\,|\log(\theta)|^{\frac{1}{2}}, \quad
 \lvert x'_{s}(\theta) \rvert \le C\,|\log(\theta)|
 \quad\text{and}\quad  \lvert x''_{s}(\theta) \rvert \le C\,|\log(\theta)|^{\frac{3}{2}}.
\end{equation}
 We define the function $\Psi_{s}(\theta)=\Phi(x_{s}(\theta))$. Applying
 Taylor-Lagrange inequality for $\Psi$ at $\theta=0$, we have:
 \begin{equation*}
  \Psi_{s}(\theta)=\Psi_{s}(0)+\theta\Psi'_{s}(0)+R^{(1)}_{s}(\theta),
 \end{equation*}
where $R^{(1)}_{s}(\theta)=\int_{0}^{\theta}\Psi''_{s}(t)(\theta-t) dt$.
Recall the definition of $y_s=x_s(0)$ given in \reff{eq:def_y_s}.
Elementary calculus give:
 \begin{align}
  \Phi\left(x_{s}(\theta)\right)
  &=\Phi(x_{s}(0))+\theta \,x'_{s}(0)\varphi(x_{s}(0))+R_{s}^{(1)}(\theta) \nonumber\\
  &=\Phi\left(y_{s}\right)  +\theta\left[\frac{(1-2d)}{2\sigma_{(d)}}y_s^2
+\frac{\delta}{\sigma_{(d)}}\right]\varphi\left(y_{s}\right)+R^{(1)}_{s}(\theta),
 \label{align:CDF_1}
 \end{align}
where $R^{(1)}_{s}(\theta)=\int_{0}^{\theta}\Big(x''_{s}(t)-x'_{s}(t)^2 x_{s}(t)\Big)\varphi(x_{s}(t))\left(\theta-t\right) dt$.
Using \reff{eq:x_s_theta} and that $t\varphi(t)$ is bounded, we have:
\begin{equation*}
 \left|R^{(1)}_{s}(\theta)\right|\le C\theta^{2}(|\log(\theta)|^{\frac{3}{2}}+|\log(\theta)|^{2})
 \le  C\theta^{2}|\log(\theta)|^{2}.
\end{equation*}

\subsection*{Study of the second term of the right hand side of \reff{eq:NCZ}}
We have $Q^{(1)}_{d,\delta}(s,\theta)=G_{s}(\theta)H(x_{s}(\theta))$ where
\begin{equation*}
 G_{s}(\theta)=\frac{2(d+s\theta)-1}{6\sigma_{(d+s\theta)}} \quad\text{ and }\quad  H(x)=\left(x^2-1\right)\varphi(x).
\end{equation*}
For the first term, we have
\begin{equation*}
 G_{s}(\theta)=G_{s}(0)+R^{(2)}_{s}(\theta)
=\frac{2d-1}{6\sigma_{(d)}}+R^{(2)}_{s}(\theta), 
\end{equation*}
where $R^{(2)}_{s}(\theta)=\int_{0}^{\theta}G'_{s}(t)dt$. We compute that:
\begin{equation*}
 G'_{s}(t)=s\left[\frac{1}{3\sigma_{(d+st)}}+\frac{[2(d+st)-1]^2}{12\sigma_{(d+st)}^3}\right].
\end{equation*}
We obtain that
$\left|R^{(2)}_{s}(\theta)\right|\le C\theta|s|\le C\theta |\log( \theta)|^{\frac{1}{2}}$.
For the second term, we have
\begin{equation*}
 H(x_{s}(\theta))=H(x_{s}(0))+R^{(3)}_{s}(\theta)
                 =\left(y_{s}^2-1\right)\varphi\left(y_{s}\right)+R^{(3)}_{s}(\theta),
\end{equation*}
where   $R^{(3)}_{s}(\theta)=\int_{0}^{\theta}x'_{s}(t) H'(x_{s}(t))dt=\int_{0}^{\theta}x'_{s}(t)\left[-x_{s}(t)^3+3x_{s}(t)\right]\varphi(x_{s}(t))dt$.
Using \reff{eq:x_s_theta} and that  $(|t|^3+t)\varphi(t)$ is bounded, we
get  that $\left|R^{(3)}_{s}(\theta)\right|\le  C\theta |\log(\theta)|$.
Finally, we obtain that
\begin{equation}\label{eq:CDF_2}
 Q^{(1)}_{d,\delta}(s,\theta)=\frac{2d-1}{6\sigma_{(d)}}\left(y_{s}^2-1\right)\varphi\left(y_{s}\right)+R^{(4)}_{s}(\theta)
\end{equation}
with $\left|R^{(4)}_{s}(\theta)\right|\le C\theta|\log(\theta)|$.
\subsection*{Study of the last term of the right hand side of \reff{eq:NCZ}}
We have
\begin{equation*}
Q^{(2)}_{d,\delta}(s,\theta)=F_{s}(\theta)S\left(\frac{d}{\theta^2}+\delta\right)\varphi(x_{s}(\theta))
\quad \text{with}\quad F_{s}(\theta)=\frac{1}{\sigma_{(d+s\theta)}}\cdot
\end{equation*}
For the first term of the right hand side, we have
\begin{equation*}
 F_{s}(\theta)=F_{s}(0)+R^{(5)}_{s}(\theta)=\frac{1}{\sigma_{(d)}}+R^{(5)}_{s}(\theta),
\end{equation*}
 where $R^{(5)}_{s}(\theta) =\int_{0}^{\theta}F'_{s}(t)dt
 =\int_{0}^{\theta}\frac{s(2(d+st)-1)}{2\sigma_{(d+st)}^{3}}dt$. We get
 that 
 $\left|R^{(5)}_{s}(\theta)\right|\le C \theta|s|\le
  C\theta|\log
 (\theta)|^{\frac{1}{2}}$. 
 For the last term of the right hand side, we have:
 \begin{equation*}
  \varphi(x_{s}(\theta))=\varphi(x_{s}(0))+R^{(6)}_{s}(\theta)
=\varphi\left(y_{s}\right)+R^{(6)}_{s}(\theta),
 \end{equation*}
where
$R^{(6)}_{s}(\theta) =\int_{0}^{\theta}x'_{s}(t)\varphi'(x_{s}(t))dt
=-\int_{0}^{\theta}x_{s}(t)x'_{s}(t)\varphi(x_{s}(t))dt$.  
So, using \reff{eq:x_s_theta} and that $t\varphi(t)$ is bounded, we get that
$\left|R^{(6)}_{s}(\theta)\right|\le C\theta|\log(\theta)|$.
Finally, we obtain that
\begin{equation}\label{eq:CDF_3}
 Q^{(2)}_{d,\delta}(s,\theta)
=\frac{1}{\sigma_{(d)}}S\left(\frac{d}{\theta^2}+\delta\right)
\varphi\left(y_{s}\right)+R^{(7)}_{s}(\theta),
\end{equation}
where $\left|R^{(7)}_{s}(\theta)\right|\le C\theta\left|\log(\theta)\right|$, since $S$ is bounded.

\subsection*{Conclusion}
We deduce from \reff{align:CDF_1}, \reff{eq:CDF_2} and \reff{eq:CDF_3} that
\begin{equation*}
 \Phi\left(x_{s}(\theta)\right)+\theta \,Q^{(1)}_{d,\delta}(s,\theta)+\theta\, Q^{(2)}_{d,\delta}(s,\theta)
 =\Phi(y_s)+\theta\frac{\varphi(y_s)}{\sigma_{(d)}}\pi\left(s,\frac{1}{\theta^2},d,\delta\right)+R^{(8)}_{s}(\theta),
\end{equation*}
where $\left|R^{(8)}_{s}(\theta)\right|\le C\theta^2|\log(\theta)|^{2}$.
We get the  result by taking $\theta=1/{\sqrt{n}}$ and using
\reff{eq:NCZ} and the obvious bound on $U_n$ given by \reff{eq:bd-Un} so
that $|U_n|\leq C /n$. 
\end{proof}

 We state a Lemma which will be useful for the proof of Corollary \ref{lem:ind10}.
 \begin{lem}\label{lem:maj_int_gauss}
  Let $y\in[0,1]$ and $\alpha>0$. For all $n\geq 2$, we have with
  $d=D(y)$, $A=\alpha\sqrt{\log(n)}$ and $y_s=-s/\sigma_ {(d)}$,
  \begin{equation*}
  \Phi\left(-\frac{A}{\sigma_{(d)}}\right)\le \inv{\alpha\, n^{2\alpha^2}},
\quad
   \int_{A}^{+\infty}s\Phi\left(y_s\right)ds\le \inv{\alpha\, 
     n^{2\alpha^2}}
   \quad \text{and} \quad 
   \int_{A}^{+\infty}s^{2}\varphi(y_s)ds \le \inv{\alpha\, 
     n^{\alpha^2}} \cdot
  \end{equation*}
 \end{lem}

 \begin{proof}
  For all $t\ge 0$, we have 
  \begin{equation}\label{eq:Phi_maj}
   \Phi(-t)=\int_{t}^{+\infty}s\frac{\varphi(s)}{s}ds\le \frac{1}{t}\int_{t}^{+\infty}s\varphi(s)ds= \frac{1}{t}\varphi(t) .
  \end{equation}
Because $\sigma_{(d)}\le 1/2$, we get with $t=\frac{A}{\sigma_{(d)}}$
the following rough upper bound:
\begin{equation}\label{eq:Phi_maj_2}
 \Phi\left(-\frac{A}{\sigma_{(d)}}\right)\le
 \frac{\sigma_{(d)}}{A}\varphi\left(\frac{A}{\sigma_{(d)}}\right)\le
 \inv{\alpha\, n^{2\alpha^2} }\cdot
\end{equation}
Using again \reff{eq:Phi_maj} and \reff{eq:Phi_maj_2}, we get, for the second inequality that:
\begin{equation*}
  \int_{A}^{+\infty}s\Phi\left(y_s\right)ds\le \sigma_{(d)}\int_{A}^{+\infty}\varphi(-y_s)ds=\sigma_{(d)}^{2}\Phi\left(-\frac{A}{\sigma_{(d)}}\right)
  \le  \inv{\alpha\, n^{2\alpha^2} }\cdot
\end{equation*}
For the last inequality, we have:
\begin{align*}
 \int_{A}^{+\infty}s^{2}\varphi(y_s)ds
 =\frac{2\sigma_{(d)}^2}{\sqrt{2\pi}}\int_{A}^{+\infty}\frac{s^2}{2\sigma_{(d)}^2}\re^{-\frac{s^2}{2\sigma_{(d)}^2}}ds
 &\le \frac{4\sigma_{(d)}^2}{\sqrt{2\pi}}\int_{A}^{+\infty}\re^{-\frac{s^2}{4\sigma_{(d)}^2}}ds\\
 &=4\sqrt{2}\sigma_{(d)}^3\Phi\left(-\frac{A}{\sqrt{2}\sigma_{(d)}}\right)\\
 &\le  \inv{\alpha\, n^{\alpha^2} }\cdot
\end{align*}
where we used $x\re^{-x}\le 2\re^{-\frac{x}{2}}$ for the first inequality and an inequality similar to \reff{eq:Phi_maj_2}
with $\sigma_{(d)}$ replaced by $\sqrt{2}\sigma_{(d)}$ for the last one.
 \end{proof}

For $f\in\mathcal{C}^{2}([0,1])$, we set $\lVert f\rVert_{3,\infty}= \lVert f \rVert_{\infty} +\lVert f' \rVert_{\infty}+\lVert f'' \rVert_{\infty}$.
 \begin{lem}
\label{lem:ind10}
 Assume that $W$ satisfies condition \reff{eq:condi_W}. Let $y\in(0,1)$ and $\alpha\geq 1$.
There exists a positive constant $C$ such that for all $H\in\mathcal{C}^{2}([0,1])$, $\delta\in[-1,1]$ and $n\geq 2$ such that 
$\left[d\pm\frac{A}{\sqrt{n}}\right]\subset D\big((0,1)\big)$, with $d=D(y)$ and
$A=\alpha\sqrt{\log(n)}$, we have: 
 \begin{multline*}
\sqrt{n}\int_{-A}^{A}H\left(D^{-1}\left(d+\frac{s}{\sqrt{n}}\right)\right)
\left(\mathcal{H}_{n,d,\delta}\left(d+\frac{s}{\sqrt{n}}\right) 
  -\1{s\le 0}\right)ds  \\ 
  =
  \frac{H'(y)}{D'(y)}\frac{\sigma_{(d)}^{2}}{2}+H(y)\left(\frac{1-2d}{2}+\delta+
    S(nd+\delta)\right) +R_{n}^{\ref{lem:ind10}}(H),
 \end{multline*}
 where 
 \begin{equation}
\label{eq:lem:ind10}
  \left|R_{n}^{\ref{lem:ind10}}(H)\right|\le  C\,\lVert H\rVert_{3,\infty}\,
  n^{-1/2} \log(n)^3.
 \end{equation}
 \end{lem}
 
Because of the assumption $\left[d\pm\frac{A}{\sqrt{n}}\right]\subset
D\big((0,1)\big)$, we need to  rule out the cases $y\in \{0, 1\}$, so
that Lemma \ref{lem:ind10} holds only for $y\in(0,1)$. 
 \begin{proof}
   In  what follows,  $C$ denotes  a positive  constant which  depends on
  $\varepsilon_0$   and $W$    (but   in particular     not       on        $n\geq
  2$, 
  $s\in[-\alpha\sqrt{\log          (n)},\alpha\sqrt{\log         (n)}]$, 
  $\delta\in[-1,1]$       and      $d\in       K_0$      such       that
  $d+\frac{s}{\sqrt{n}}\in K_{0}$)  and which  may change from  lines to
  lines.\\

Let $\theta\in(0,1/\sqrt{2}]$ (we shall take $\theta=1/\sqrt{n}$ later
on) and assume that $|s|\leq  \alpha\sqrt{2|\log (\theta)|}$
and $d+s\theta\in K_0$. 
We set $\Psi(\theta)=H\left(D^{-1}(d+s\theta)\right)$. Notice that
 $\Psi'(\theta)=\frac{s}{D'\circ D^{-1}(d+s\theta)}H'(D^{-1}(d+s\theta))$.
 By Taylor-Lagrange equality we have:
 \begin{equation}
\label{eq:TL-psi1}
  \Psi(\theta)=\Psi(0)+\theta\Psi'(0)+R^{(1)}_{s}(\theta)
              =H(y)+\theta \frac{s}{D'(y)}H'(y)+R^{(1)}_{s}(\theta)
 \end{equation}
where $R^{(1)}_{s}(\theta)=\int_{0}^{\theta}\Psi''(t)(\theta-t)dt$.
We have
\begin{equation*}
 \Psi''(\theta)=s^{2}\left[\frac{H''(D^{-1}(d+s\theta))}{(D'\circ D^{-1}(d+s\theta))^2}-\frac{H'(D^{-1}(d+s\theta))(D''\circ D^{-1}(d+s\theta))}{(D'\circ D^{-1}(d+s\theta))^3}\right].
\end{equation*}
Thus, we get that $\left|R^{(1)}_{s}(\theta)\right|\le C\, \left(\lVert H'\rVert_{\infty} + \lVert H''\rVert_{\infty}\right)\,s^{2}\theta^{2}
\le C\, \lVert H \rVert_{3,\infty}\,\theta^{2}|\log(\theta)| $.
Choosing $\theta=1/{\sqrt{n}}$, we deduce from \reff{eq:TL-psi1} that:
\begin{equation}\label{eq:psi_dev}
 H\left(D^{-1}\left(d+\frac{s}{\sqrt{n}}\right)\right)
 =H(y)+\frac{1}{\sqrt{n}}\frac{s}{D'(y)}H'(y)+R^{(1)}_{s}\left(\frac{1}{\sqrt{n}}\right),
\end{equation}
where $\left|R^{(1)}_{s}\left(1/{\sqrt{n}}\right)\right|\le C\,
\lVert H \rVert_{3,\infty}\, \log(n)/n$.
Recall the definition of $y_s$ and $\pi(s,n,d,\delta)$ given by \reff{eq:def_y_s}.
By Proposition \ref{prop:CFD_approxi} and equation \reff{eq:psi_dev}, we get that:
\begin{multline}
\label{eq:HH=}
\sqrt{n}\,H\left(D^{-1}\left(d+\frac{s}{\sqrt{n}}\right)\right)\left(\mathcal{H}_{n,d,\delta}\left(d+\frac{s}{\sqrt{n}}\right)-\1{s\le 0}\right)\\
\begin{aligned}
    &=\sqrt{n}\left(H(y)+\frac{1}{\sqrt{n}}\frac{s}{D'(y)}H'(y)\right)\,
    \left(\left(\Phi(y_s)-\1{s\le 0}\right)
 +\frac{1}{\sqrt{n}}\frac{\varphi(y_{s})}
 {\sigma_{(d)}}\pi(s,n,d,\delta)\right)+R^{(0)}_{n}(s)\\   
 &=\sqrt{n}\, H(y)\Delta^{(1)}(s)+\frac{H'(y)}{D'(y)}
 \Delta^{(2)}(s)+\frac{H(y)}{\sigma_{(d)}}\Delta^{(3)}(s)+R^{(0)}_{n}(s)+\hat{R}_{n}^{(0)}(s), 
\end{aligned}
\end{multline}
where
\[
 \Delta^{(1)}(s)=\left(\Phi(y_s)-\1{s\le 0}\right), \quad
 \Delta^{(2)}(s)=s\left(\Phi(y_s)-\1{s\le 0}\right),\quad 
 \Delta^{(3)}(s)=\varphi(y_{s})\pi(s,n,d,\delta),
\]
\begin{align}
\nonumber
 \left|R^{(0)}_{n}(s)\right|
 &\le \sqrt{n} \lVert H
   \rVert_{\infty}|R^{\ref{prop:CFD_approxi}}(s,n,d,\delta)|
   + \sqrt{n} |R^{(1)}_{s}\left(1/{\sqrt{n}}\right)| + \sqrt{n}
   |R^{\ref{prop:CFD_approxi}}(s,n,d,\delta)R^{(1)}_{s}\left(1/{\sqrt{n}}\right)|\\
&\leq  C\,\lVert H \rVert_{3,\infty}\frac{\log(n)^2}{\sqrt{n}}
\label{eq:majoR0}
\end{align}
and 
\begin{equation}\label{eq:majo_hat_R0}
 \left|\hat{R}^{(0)}_{n}(s)\right|
 =\left|\frac{1}{\sqrt{n}}\frac{H'(y)}{\sigma_{(d)}D'(y)}s\varphi(y_s)\pi(s,n,d,\delta)\right|
 \le C\,\lVert H \rVert_{3,\infty}\frac{\sqrt{\log(n)}}{\sqrt{n}}\cdot
\end{equation}

\subsection*{Study of $\int_{-A}^{A}\Delta^{(1)}(s)ds$}
Since $\Delta^{(1)}$ is an  odd integrable functions on $\R^{*}$, we get that:
\begin{equation}\label{eq:int_delta_1_3}
 \int_{-A}^{A}\Delta^{(1)}(s)ds=0. 
\end{equation}

\subsection*{Study of $\int_{-A}^{A}\Delta^{(2)}(s)ds$}
Because         $\Delta^{(2)}$         is         integrable         and
$\int_{\R}\Delta^{(2)}(s)ds={\sigma_{(d)}^{2}}/{2}$, we get that
\begin{equation}\label{eq:int_delta_2}
 \int_{-A}^{A}\Delta^{(2)}(s)\,ds=\frac{\sigma_{(d)}^{2}}{2}+R^{(2)}_{n},
 \quad\text{ with }\quad R^{(2)}_{n}=-2\int_{A}^{+\infty}s\Phi(y_s)ds \cdot
\end{equation}
Using Lemma \ref{lem:maj_int_gauss}, we get that
\begin{equation}
\label{eq:majoR2}
 |R^{(2)}_{n}|\le C\, n^{-2\alpha^2}.
\end{equation}

\subsection*{Study of $\int_{-A}^{A}\Delta^{(3)}(s)ds$}
We have, using \reff{eq:def_y_s}, that:
\begin{equation*}
 \Delta^{(3)}(s)=\varphi(y_s)\left(\frac{1-2d}{6}(1+2y_{s}^2)+
   \delta+S(nd+\delta)\right). 
\end{equation*}
By elementary  calculus, we have that:
\begin{equation*}
 \int_{\R}\varphi(y_s)ds=\int_{\R}y_{s}^2\varphi(y_s)ds=\sigma_{(d)}.
\end{equation*}
We get that:
\begin{equation} 
\label{eq:Delta3}
\int_{-A}^{A}\Delta^{(3)}(s)ds
 =\sigma_{(d)}\left[\frac{1-2d}{2}+\delta+S\left(nd+\delta\right)\right]+R_{n}^{(3)},
\end{equation}
where
\begin{equation}
\label{eq:majoR3}
R_{n}^{(3)}=-2\sigma_{(d)}\left(\frac{1-2d}{6}+\delta
+S(nd+\delta)\right)\Phi\left(-\frac{A}{\sigma_{(d)}}\right)
-2\,\frac{1-2d}{3}\int_{A}^{+\infty}y_s^{2}\varphi(y_s)ds.
\end{equation}
Using Lemma \ref{lem:maj_int_gauss} and since $|2d-1|\le 1$, $|\delta|\le 1$ and $S$ is bounded by $1$, we have that:
\begin{equation}\label{eq:rest_3_4}
 |R^{(3)}_{n}|\le C n^{-\alpha^2}. 
\end{equation}

\subsection*{Conclusion}
Using \reff{eq:HH=}, \reff{eq:int_delta_1_3}, \reff{eq:int_delta_2},
\reff{eq:Delta3}, we deduce that
\begin{multline*}
 \sqrt{n}\int_{-A}^{A}H\left(D^{-1}\left(d+\frac{s}
     {\sqrt{n}}\right)\right)\left(\mathcal{H}_{n,d,\delta}
   \left(d+\frac{s}{\sqrt{n}}\right)    -\1{s\le 0}\right)ds \\
= \frac{H'(y)}{D'(y)}\frac{\sigma_{(d)}^{2}}{2}
+H(y)\left[\frac{1-2d}{2} +\delta+S\left(nd+\delta\right)\right]
+R_{n}^{\ref{lem:ind10}}(H), 
\end{multline*}
where $R_{n}^{\ref{lem:ind10}}(H)=\int_{-A}^A (R^{(0)}_{n}(s)+\hat{R}^{(0)}_{n}(s))\, ds
+(H'(y)/D'(y)) R^{(2)}_{n}+(H(y)/\sigma_{(d)})R^{(3)}_{n}$.
Using the upper bounds \reff{eq:majoR0} and \reff{eq:majo_hat_R0} (to be integrated
over $[-A, A]$), \reff{eq:majoR2} and \reff{eq:rest_3_4} with
$\alpha\geq 1$, we get that 
$|R_{n}^{\ref{lem:ind10}}(H)|\leq  C \,\lVert H
\rVert_{3,\infty}\,{\log(n)^3}/{\sqrt{n}}$. 
 \end{proof}

We give a direct application of the previous lemma. 

\begin{lem}\label{lem:ind20}
Assume that $W$ satisfies condition \reff{eq:condi_W}. Let $y\in(0,1)$ and $\alpha\geq 1$.
There exists a positive constant $C$ such that for all
$G\in\mathcal{C}^{2}([0,1])$, $\delta \in [-1, 1]$, $n\geq 2$ such that 
$\left[d\pm\frac{A}{\sqrt{n}}\right]\subset D((0,1))$, with $d=D(y)$ and
$A=\alpha\sqrt{\log(n)}$, we have: 
\begin{multline*}
n\int_{0}^{1}G(x)\left(\mathcal{H}_{n,d,\delta}
    \left(D(x)\right)-\1{x\le y}\right) \ind_{\{\sqrt{n}|D(x)-d|\le
    A\}}dx\\
= \frac{G'(y)D'(y)-G(y)D''(y)}{D'(y)^3} \frac{\sigma_{(d)}^{2}}{2}
+\frac{G(y)}{D'(y)}\left[\frac{1-2d}{2}+\delta+ 
    S(nd+\delta) \right]+R^{\ref{lem:ind20}}_{n}(G),
 \end{multline*}
 where
 \begin{equation*}
  \left|R^{\ref{lem:ind20}}_{n}(G)\right|\le   C\,\lVert G\rVert_{3,\infty}\,
  n^{-1/2} \log(n)^3.
 \end{equation*}
 \end{lem}

\begin{proof}
 Let $G$ be a function in $\mathcal{C}^2([0,1])$. Define the function $H$ on $[0,1]$ by $H(z)=\frac{G(z)}{D'(z)}$ for all $z\in[0,1]$.
Use the change of variables $s=\sqrt{n}\left(D(x)-d\right)$ to get that:
\begin{multline*}
 \int_{0}^{1}G(x)\left(\mathcal{H}_{n,d,\delta}(D(x))-\ind_{\{x\le
     y\}}\right) \ind_{\{\sqrt{n}\left|D(x)-d\right|\le A\}}dx\\
 =\frac{1}{\sqrt{n}}\int_{-A}^{A}H\left(D^{-1}\left(d+\frac{s}{\sqrt{n}}\right)\right)
 \left(\mathcal{H}_{n,d,\delta}\left(d+\frac{s}{\sqrt{n}}\right)-\1{s\le 0}\right)ds.
\end{multline*}
By Lemma \ref{lem:ind10}, we obtain that:
\begin{multline*}
n\int_{0}^{1}G(x)\left(\mathcal{H}_{n,d,\delta}\left(D(x)\right)-\ind_{\{x\le
    y\}}\right)\1{\sqrt{n}|D(x)-d|\le
  A}dx\\
\begin{aligned}
 &=\frac{H'(y)}{D'(y)}\frac{\sigma_{(d)}^{2}}{2}+H(y)\left[\frac{1-2d}{2}+\delta+ 
    S(nd+\delta)\right]+R_{n}^{\ref{lem:ind10}}(H)\\
 &=\frac{G'(y)D'(y)-G(y)D''(y)}{D'(y)^3}
 \frac{\sigma_{(d)}^{2}}{2}+\frac{G(y)}{D'(y)}\left[\frac{1-2d}{2}+\delta+  
    S(nd+\delta)\right]+R_{n}^{\ref{lem:ind10}}(G/D').
\end{aligned}
\end{multline*}
Set  $R^{\ref{lem:ind20}}_{n}(G)=R_{n}^{\ref{lem:ind10}}\left(G/D'\right)$ 
and use \reff{eq:lem:ind10} to conclude. 
\end{proof}

\begin{lem}\label{lem:binom_Tchebychev}
Let $y\in(0,1)$ and $\alpha>0$.
 For all $u\in(0,1)$,  $\delta \in [-1, 1]$ and $n\in\Ne$ such that
 $\sqrt{n}|u-d|\ge A$ with $d=D(y)$ and $A=\alpha\sqrt{\log(n)}$, we
 have 
 \begin{equation*}
  \left|\mathcal{H}_{n,d,\delta}(u)-\1{u\le d}\right| \le n^{-\alpha+2}.
 \end{equation*}
\end{lem}

\begin{proof}
Let $X$ be a binomial random variable with parameters $(n,u)$. Assume
first that 
$u\ge d+\frac{A}{\sqrt{n}}$. Let $\lambda\ge 0$. Using Chernov 
inequality, we get:
\begin{equation}
 \label{eq:ind5}
 \mathcal{H}_{n,d,\delta}(u)-\1{u\le d}
 =\pro\left(X\le nd+\delta\right) 
\le \re^{\lambda(nd+\delta)}\,\E\left[\re^{-\lambda X}\right]
=\exp\left[\lambda(nd+\delta)+n\Psi(\lambda)\right],
\end{equation}
with  $\Psi(\lambda)=\log(1+u(\re^{-\lambda}-1))$. By Taylor-Lagrange equality, we have
\begin{equation}\label{eq:TL1}
 \Psi(\lambda)=\Psi(0)+\lambda\Psi'(0)+R(\lambda)
              =0-u\lambda+R(\lambda),
\end{equation}
where $R(\lambda)=\int_{0}^{\lambda}(\lambda-t)\Psi''(t)dt$.
Because $\Psi''(t)\ge 0$ and 
$\Psi''(t)=\frac{(1-u)u\re^{-\lambda}}{\left(1+u(\re^{-\lambda}-1)\right)^2}\le \frac{1}{4}$ 
(applying the following inequality $\frac{xy}{(x+y)^2}\le\frac{1}{4}$ with 
$x=1-u$ and $y=u\re^{-\lambda}$), we get that $|R(\lambda)|\le
\frac{\lambda^{2}}{8}\leq  \lambda^2$.
Finally, applying \reff{eq:TL1} with $\lambda=\sqrt{\frac{\log(n)}{n}}$, we get that
\begin{equation}\label{eq:TL2}
 n\Psi(\lambda) =-u\sqrt{n\log(n)}+R^{(2)}(n),
\end{equation}
with $|R^{(2)}(n)|\le \log(n)$. Using \reff{eq:ind5} and \reff{eq:TL2}, we get that
\begin{align*}
  \mathcal{H}_{n,d,\delta}(u)-\1{u\le d}
  &\le \exp\left[\sqrt{\frac{\log(n)}{n}}(nd+\delta)-u\sqrt{n\log(n)}+R^{(2)}(n)\right]\\
  &=\exp\left[\sqrt{n\log(n)}(d-u)+R^{(3)}(n)\right],
\end{align*}
where $|R^{(3)}(n)|\le 2\log(n)$, since $|\delta|\le 1$.
Because $d-u\le\frac{-A}{\sqrt{n}}$ with $A=\alpha\sqrt{\log(n)}$, we have that
\begin{equation*}
 \mathcal{H}_{n,d,\delta}(u)-\1{u\le d}\le
 \re^{-\alpha\log(n)+R^{(3)}(n)}\le
 \re^{(-\alpha+2)\log(n)}=n^{-\alpha+2}. 
\end{equation*}

In the case where $u\le d-\frac{A}{\sqrt{n}}$, we have that
\begin{align*}
 0\geq \mathcal{H}_{n,d,\delta}(u)-\1{u\le d}
 &=\pro\left(X\le nd+\delta\right)-1\\
 &\ge -\pro\left(X\ge  nd+\delta\right)\\
 &=-\pro\left(n-X \le n(1-d)-\delta \right). 
\end{align*}
Since $n-X$ is  a binomial random variable with parameters $(n,1-u)$,
using similar argument as in the first part of the proof (with $u$ and $X$
replaced by $1-u$ and $n-X$),  we get that, for $u\le d-\frac{A}{\sqrt{n}}$:
\begin{equation*}
 \mathcal{H}_{n,d,\delta}(u)-\1{u\le d} \ge -n^{-\alpha+2}.
\end{equation*}

We deduce that $\left|\mathcal{H}_{n,d,\delta}(u)-\1{u\le d}\right|\le
n^{-\alpha+2}$.
\end{proof}

The following lemma  is a direct application of Lemma
\ref{lem:binom_Tchebychev} with $u=D(x)$. 
\begin{lem}\label{lem:ind40}
Assume that $W$ satisfies condition \reff{eq:condi_W}. Let $y\in(0,1)$ and $\alpha\geq 1$.
For all $G\in\mathcal{B}([0,1])$,  $\delta \in [-1, 1]$ and $n\in\Ne$, we have with
 $d=D(y)$ and $A=\alpha\sqrt{\log(n)}$:
 \begin{equation*}
  n\int_{0}^{1}G(x)\left|\mathcal{H}_{n,d,\delta}\left(D(x)\right)-\ind_{\{x\le y\}}\right|\ind_{\{\sqrt{n}|D(x)-d|\ge A\}}dx=R_{n}^{\ref{lem:ind40}}(G),
 \end{equation*}
 where 
 \begin{equation*}
  \left|R_{n}^{\ref{lem:ind40}}(G)\right|\le \lVert
  G\rVert_{\infty}\,n^{-\alpha+3}. 
 \end{equation*}
\end{lem}

 Combining Lemma \ref{lem:ind20} with Lemma \ref{lem:ind40} for
 $\alpha=3$, we deduce the following proposition. 
 \begin{prop}\label{prop:ind8}
  Assume that $W$ satisfies condition \reff{eq:condi_W}. Let $y\in(0,1)$.
 There exists a positive constant $C$ such that for all
 $G\in\mathcal{C}^{2}([0,1])$,  $\delta \in [-1, 1]$ and $n\in \N^*$ such that 
$\left[d\pm\frac{A}{\sqrt{n}}\right]\subset D((0,1))$, with $d=D(y)$ and
$A=4\sqrt{\log(n)}$, we have: 
\begin{multline*}
 n\int_{0}^{1}G(x)\left(\mathcal{H}_{n,d,\delta}\left(D(x)\right)-\1{x\le y}\right)dx\\
  =\frac{G'(y)D'(y)-G(y)D''(y)}{D'(y)^3}
  \frac{\sigma_{(d)}^{2}}{2}+\frac{G(y)}{D'(y)}\left[\frac{1-2d}{2}+\delta
    + S(nd+\delta)  \right] +R_{n}^{\ref{prop:ind8}}(G),  
\end{multline*}
 with
 \begin{equation}
\label{eq:prop:ind8}
  \left|R_{n}^{\ref{prop:ind8}}(G)\right|\le C\,  \rVert
  G\rVert_{3,\infty}\, n^{-\frac{1}{4}}. 
 \end{equation}
 \end{prop}
 
\section{Appendix B: Proof of Proposition  \ref{prop:stein}}
\label{sec:proof-prop-stein}

We first state a preliminary lemma in Section \ref{sec:prel-prel} and
then provide the proof of Proposition  \ref{prop:stein} in Section
\ref{sec:proof-proof}. 

\subsection{A preliminary result}
\label{sec:prel-prel}
For $\ry=(y_1, y_2)\in [0, 1]^2$, let  $M(\ry)$
 be the  covariance matrix of a couple $(Y_1, Y_2)$ of Bernoulli random
 variables such that 
 $\P(Y_i=1)=D(y_i)$ for $i\in \{1,2\}$ and $\P(Y_1=Y_2=1)=\int_{[0, 1]} W(y_1,z)
 W(y_2, z)\, dz$. 

\begin{lem}
   \label{lem:minor-det}
 Assume that $W$ satisfies condition \reff{eq:condi_W}. There exists
 $\varepsilon'>0$ such that for all $\ry\in
 [0, 1]^2$, we have $\det(M(\ry))>\varepsilon'$.
\end{lem}

\begin{proof}
Let $\mathbb{M}_2$ be the set of matrices of size $2\times 2$, and
$\norm{\cdot}_\infty $ be the norm on $\mathbb{M}_2$ defined in
\reff{eq:normMi}.  We consider the 
closed set on $\mathbb{M}_2$:
\[
\cf=\cf_+\bigcup \cf_- \quad\text{where}
\quad
\cf_\pm=\left\{r(I_2 \pm \begin{pmatrix}
  0 &1\\1&0\end{pmatrix}); \, r\in [0, 1/4]\right\}
\]
where $I_2\in \mathbb{M}_2$ is the  identity matrix. Notice $\cf$ is the
set of all covariance matrices  of couples of Bernoulli random variables
having determinant  equal to  0. Since the  determinant is  a continuous
real-valued    function    on    $\mathbb{M}_2$,    to    prove    Lemma
\ref{lem:minor-det}, it is enough to prove that for all $\ry=(y,y')\in [0, 1]^2$
and all $M_0\in \cf$:
\begin{equation}
   \label{eq:norm-mat}
\norm{M(\ry) -M_0}_\infty \geq \varepsilon^2/4.
\end{equation}

We set  $p=D(y)$, $p'=D(y')$ and $\alpha=\int_{[0, 1]} W(y,z)W(y',z) \,
dz$ so that:
\[
M(\ry)=\begin{pmatrix}
   p(1-p) & \alpha -pp'\\ \alpha -pp' & p'(1-p')
\end{pmatrix}. 
\]
And the elements $M_0\in \cf$ are of the form, with $r\in [0, 1/4]$:
\[
M_0=\begin{pmatrix}
   r & \pm r\\ \pm r & r
\end{pmatrix}. 
\]

The proof of \reff{eq:norm-mat} is divided
in three  cases. Recall that $W$ satisfies condition \reff{eq:condi_W}. Without
loss of generality, we can assume that $p\leq p'$ and thus:
\begin{equation}
   \label{eq:majo-pp}
\varepsilon\leq p\leq p'\leq 1-\varepsilon.
\end{equation}
Since $(1-W(y,z))(1-W(y',z))$ is non negative, by integrating with
respect to $z$ over $[0, 1]$, we get that $\alpha\geq  p+p'-1$. Using
that $W\leq 1-\varepsilon$, we deduce, denoting by $x^+=\max(x,0)$ the positive
par of $x\in \R$,  that:
\begin{equation}
   \label{eq:majo-alpha}
(p+p'-1)_+\leq  \alpha\leq (1-\varepsilon) p.
\end{equation}

\subsection*{The case  $M_0\in \cf_+$}
Recall  that $p\leq p'$. 
If $|r-p(1-p)|\geq  \varepsilon^2/4$, then, by considering  the first term on the
diagonal,  we have $\norm{M(\ry) -M_0}_\infty \ge \varepsilon^2/4$. 

If $|r-p(1-p)|\le  \varepsilon^2/4$, then, by considering the term out the
diagonal,  we have:
\[
\norm{M(\ry) -M_0}_\infty \ge |\alpha- pp' -r|.
\]
For $\delta'=r-p(1-p)\in [-\varepsilon^2/4, \varepsilon^2/4]$, we get, using that
$\alpha\leq  (1-\varepsilon)p$ and $p\leq p'$:
\begin{align*}
   \alpha -pp' -r 
&\leq (1-\varepsilon) p -p^2  - p(1-p) - \delta'\\
&\leq  -\varepsilon^2 + \varepsilon^2/4 = -3\varepsilon^2/4.
\end{align*}

We deduce that \reff{eq:norm-mat} holds if $M_0\in \cf_+$. 

\subsection*{The case $|1-p-p'|>  \varepsilon/2$ and $M_0\in
  \cf_-$} 
If $|r-p(1-p)|\geq  \varepsilon^2/4$,  then, by considering  the first term on the
diagonal,  we have $\norm{M(\ry) -M_0}_\infty \ge \varepsilon^2/4$.

If $|r-p(1-p)|\le  \varepsilon^2/4$, then, by considering the term out the
diagonal,   we have:
\[ 
\norm{M(\ry) -M_0}_\infty \ge |\alpha- pp' +r|.
\]
Assume first that  $1-p-p'>\varepsilon/2$. 
For       $\delta'=r-p(1-p)\in        [-\varepsilon^2/4,  \varepsilon^2/4]$,     we
get,      using   $\alpha\geq 0$, that: 
\begin{align*}
   \alpha -pp' +r 
&\geq     p(1-p-p') + \delta'\\
&\geq  \varepsilon^2/2  -  \varepsilon^2/4 = \varepsilon^2/4.
\end{align*}
Assume then that $1-p-p'<-\varepsilon/2$. 
For       $\delta'=r-p(1-p)\in        [-\varepsilon^2/4,  \varepsilon^2/4]$,     we
get,      using    the lower bound $\alpha\geq p+p'-1$ from \reff{eq:majo-alpha}, that:
\begin{align*}
   \alpha -pp' +r 
&\geq   (1-  p)(p+p'-1) + \delta'\\
&\geq  \varepsilon^2/2  -  \varepsilon^2/4 = \varepsilon^2/4.
\end{align*}
We get $\norm{M(\ry) -M_0}_\infty \ge \varepsilon^2/4$.

We deduce that \reff{eq:norm-mat} holds if  $|1-p-p'|>  \varepsilon/2$ and $M_0\in
  \cf_-$.

\subsection*{The case $|1-p-p'|\leq  \varepsilon/2$ and $M_0\in \cf_-$}
Applying Lemma \ref{lem:majo-intfg} below, with $f=W(y, \cdot)$,  $g=W(y',
\cdot)$ and $\delta=1-p-p'$, we get  that:
\begin{equation}
   \label{eq:mino-alpha-fg}
\alpha\geq (1-\varepsilon)(\varepsilon-\delta). 
\end{equation}

If $|r-p(1-p)|\geq  \varepsilon^2/4$,  then, by considering  the first term on the
diagonal,  we have $\norm{M(\ry) -M_0}_\infty \ge \varepsilon^2/4$.

If $|r-p(1-p)|\le  \varepsilon^2/4$, then, by considering the term out the
diagonal,   we have:
\[ 
\norm{M(\ry) -M_0}_\infty \ge |\alpha- pp' +r|.
\]
For $\delta'=r-p(1-p)\in [-\varepsilon^2/4, \varepsilon^2/4]$, using
\reff{eq:mino-alpha-fg}, we get that:
\begin{align*}
  \alpha -pp' +r 
&=\alpha - p(1-p-\delta) + p(1-p) + \delta'\\
&\geq (1-\varepsilon) \varepsilon -\delta(1-\varepsilon-p)  + \delta'\\
&\geq  (1-\varepsilon) \varepsilon - (1-2\varepsilon)\varepsilon/2- \varepsilon^2/4
\geq \varepsilon^2/4.
\end{align*}
We deduce that \reff{eq:norm-mat} holds if $|1-p-p'|\leq  \varepsilon/2$
and $M_0\in \cf_-$.

\subsection*{Conclusion}
Since \reff{eq:norm-mat} holds when $M_0\in \cf_+$, when 
$M_0\in \cf_-$ and either 
$|1-p-p'|>  \varepsilon/2$ or $|1-p-p'|\leq   \varepsilon/2$,  we
deduce that \reff{eq:norm-mat} holds under the condition of Lemma
\ref{lem:minor-det}. 
\end{proof}

\begin{lem}
   \label{lem:majo-intfg}
Let $\varepsilon\in (0, 1/2)$, $\delta\in [-\varepsilon/2,\varepsilon/2]$, $f, g\in \cb([0, 1])$ such that $0\leq
f,g\leq 1-\varepsilon$ and $\int_{[0, 1]} (f+g)=1-\delta$. Then we have
$\int_{[0, 1]} fg\geq (1-\varepsilon)(\varepsilon-\delta)$, and this
lower bound is sharp.
\end{lem}

\begin{proof}
Set $f_1=\min(f, g)$ and $g_1=\max(f, g)$ so that $0\leq f_1\leq  g_1\leq
1-\varepsilon$ and  $\int_{[0, 1]} (f_1+g_1)=1-\delta$ and 
$\int_{[0, 1]} f_1\, g_1=\int_{[0, 1]} fg$.   Set $h=\min(f_1,
(1-\varepsilon-g_1))$ as well as $f_2=f_1-h$ and $g_2=g_1+h$ so that $0\leq
f_2\leq  g_2\leq  1-\varepsilon$, $\int_{[0, 1]} (f_2+ g_2)= 1-\delta$ and 
\[
 \int_{[0, 1]} f_2 \,g_2
= \int_{[0, 1]} (f_1-h) \,(g_1+h)
= \int_{[0, 1]} f_1 \, g_1 - \int_{[0,
  1]} (h(g_1-f_1) +h^2) \leq \int_{[0, 1]} f_1 \, g_1=\int_{[0, 1]} f g.
\]
Since by construction either $f_2(x)=0$ or $g_2(x)=1-\varepsilon$, we
deduce that:
\[
\int_{[0, 1]} f g\geq \int_{[0, 1]} f_2\,  g_2 \geq (1-\varepsilon) \int_{[0, 1]} f_2
=(1-\varepsilon) \left(1- \delta - \int_{[0, 1]} g_2\right) 
\geq  (1-\varepsilon)(\varepsilon-\delta).
\]
To see this lower bound is sharp, consider $g=1-\varepsilon$ and
$f=\varepsilon-\delta$.
\end{proof}

\subsection{Proof of Proposition  \ref{prop:stein}}
\label{sec:proof-proof}
We set:
\[
\hat Z_n= (n-1)^{-1/2}M(\rx)^{-1/2}(\hat D^{(n+1)} -\mu(\rx)),
\]
which  is,   conditionally  on   $\{X_{[2]}=\rx\}$,  distributed   as  the
normalized  and  centered  sum  of  $n-1$  independent  random  variables
distributed as  $Y=(Y_1, Y_2)$, with  $Y_1$ and $Y_2$  Bernoulli random
variables   such   that $\E[Y]=\mu(\rx)/(n-1)$ and $\Cov(Y, Y)=M(\rx)$.

Using Theorem $3.5$ from \cite{Chen_Fang_2015_article} or
Theorem $1.1$ from  \cite{Bentkus_2003_article}, we get that:
\[
 \sup_{\conv \in\convset}\left|\P\left(\hat Z_n \in \conv\big|\,  X_{[2]}=\rx\right)
 - \P\left(Z\in \conv \right)\right|\leq 115\, \sqrt{2} \, \gamma,
\]
 where 
\[
\gamma=(n-1) \E\left[\left|(n-1)^{-1/2} M(\rx)^{-1/2}(Y- \E[Y])\right|^3\right]. 
\]
Let $\norm{\cdot}_\infty $ denote the matrix norm on the set $\mathbb{M}_2$ of
real matrices of dimension $2\times 2$ induced by the maximum vector norm on $\R^2$, 
which is the maximum absolute line sum:
\begin{equation}
   \label{eq:normMi}
\norm{M}_\infty=\max_{1\le i\le2}\sum_{j=1}^{2}|M(i,j)|, \quad \text{for
  all } M\in\mathbb{M}_2.
\end{equation}
Recall that  $\norm{\cdot}_\infty $ is an induced norm (that is
$\norm{AB}_\infty \leq \norm{A}_\infty \norm{B}_\infty $). For $M\in \mathbb{M}_2$ and $\rx\in\R^2$, 
we have $|M\rx|\leq \sqrt{2}\,   \norm{M}_\infty \, |\rx|$.
If $M\in \mathbb{M}_2$ is symmetric positive definite (which is only used for the second
inequality and the equality), we get:
\begin{equation}
\label{eq:norm_infinity}
  \norm{M}_\infty ^{1/2} \leq \norm{M ^{1/2}}_\infty \leq
\sqrt{2}\norm{M}_\infty ^{1/2}
\quad\text{and}\quad
\norm{M ^{-1}}_\infty=
\frac{\norm{M }_\infty}{|\det(M)|}\cdot
\end{equation}
We deduce that if $M\in \mathbb{M}_2$ is symmetric positive definite, then 
\[
\norm{M^{-1/2} }_\infty \le \sqrt{2}|\det(M)|^{-1/2}\norm{M
}_\infty^{1/2}.
\]
We obtain  that for $n\geq 2$:
\begin{align*}
   \gamma
&\leq  2^{3}(n-1)^{-1/2} \norm{M (\rx)}_\infty ^{3/2}\det(M(\rx))^{-3/2} \,\E\left[\left|Y-
  \E[Y]\right|^3\right]\\
& \leq 2^{5/2} n^{-1/2} \det(M(\rx))^{-3/2} \,\E\left[|Y_1-\E[Y_1]|^3 +
  |Y_2-\E[Y_2]|^3\right]\\
&\leq 2^{3/2} n^{-1/2}\det(M(\rx))^{-3/2} ,
\end{align*}
where we used that $\norm{M(\rx)}_\infty \leq  1/2$ for the second
inequality and  the convex inequality $(x+y)^{p}\le
2^{p-1}(x^{p}+y^{p})$ for the third and that $|Y_i-\E[Y_i]|\leq 1$ so
that $\E[|Y_i-\E[Y_i]|^3]\leq \Var(Y_i)\leq 1/4$. 
We deduce from Lemma \ref{lem:minor-det}, there
exists $C_0>0$ such that for all $\rx=(x_1, x_2)\in [0, 1]^2$ with
$x_1\neq x_2$ and all $n\geq 2$: 
\[
 \sup_{\conv \in\convset}\left|\P\left(\hat Z_n \in \conv\big|\, X_{[2]}=\rx\right)
 - \P\left(Z\in \conv \right)\right|\leq C_0\, n^{-1/2}.
\]
To conclude, replace the convex set $K$ in this formula by the convex
set $\frac{M(\rx)^{-\frac{1}{2}}}{\sqrt{n-1}}(\conv-\mu(\rx))$.

\end{document}